\documentclass[11pt,english]{amsart}

\usepackage{amsmath,amsfonts,amssymb,amsthm,amscd}
\usepackage{dsfont, mathrsfs,enumerate,eucal,fontenc}
\usepackage[varg]{txfonts}
\usepackage{xspace}
\usepackage[english]{babel}

\usepackage{xy, xypic}
\usepackage{color} 
\usepackage{graphicx}
\usepackage{psfrag}  
\usepackage{pst-all}
\usepackage{youngtab}

\usepackage{xcolor}
\definecolor{rouge}{rgb}{0.85,0.1,.4}
\definecolor{bleu}{rgb}{0.1,0.2,0.9}
\definecolor{violet}{rgb}{0.7,0,0.8}
\usepackage[colorlinks=true,linkcolor=bleu,urlcolor=violet,citecolor=rouge]{hyperref}
\usepackage{breakurl}
\usepackage{xspace}

\usepackage{fullpage}

\DeclareMathAlphabet{\mathpzc}{OT1}{pzc}{m}{it}

\theoremstyle{plain}
\newtheorem{theorem}{Theorem}[section]
\newtheorem{lemma}[theorem]{Lemma}
\newtheorem{coro}[theorem]{Corollary}
\newtheorem{prop}[theorem]{Proposition}

\newtheorem{theo}[theorem]{Theorem}
\theoremstyle{definition}
\newtheorem{defi}[theorem]{Definition}
\theoremstyle{remark}

\newtheorem{claim}[theorem]{Claim}

\newtheorem{quest_alpha}{Question}
\def\quest_alpha{\Alph{quest_alpha}}

\def\k{{\Bbbk}}

\def\rk#1{{\mathrm {r}}_{#1}}
\def\rg{\ell}               
\def\r{{\rm reg}}

\def\loc{{\mathrm {sm}}}

\def\geq{\geqslant}
\def\leq{\leqslant}

\def\poie#1#2#3#4#5#6#7#8#9{\def\un{#5#6#7#8#9}\def\deux{#6#7#8#9}\def\trois{#2#4#8#9}
\def\quatre{#8#9}\def\cinq{#5#6#7}\def\six{#6#7}\def\sept{#2#4}
\ifx\un\empty {#1}_{#2}{#3 \hskip 0.15em}{#1}_{#4} \else \ifx\deux\empty 
{#5}(#1_{#2}){#3 \hskip 0.15em}{#5}(#1_{#4})
\else \ifx\trois\empty {#5}_{#6}(#1){#3 \hskip 0.15em}{#5}_{#7}(#1) 
\else \ifx\quatre\empty {#5}_{#6}(#1_#2){#3 \hskip 0.15em}{#5}_{#7}(#1_#4) 
\else \ifx\cinq\empty {#1}_{#2}^{#8}{#3 \hskip 0.15em}#1_#4^{#9} 
\else \ifx\six\empty {#5}(#1_{#2}^{#8}){#3 \hskip 0.15em}{#5}(#1_{#4}^{#9}) 
\else \ifx\sept\empty {#5}_{#6}^{#8}(#1){#3 \hskip 0.15em}{#5}_{#7}^{#9}(#1) \else
{#5}_{#6}(#1_{#2}^{#8})^{#9}{#3 \hskip 0.15em}{#5}_{#7}(#1_{#4}^{#8})^{#9} 
\fi \fi \fi \fi \fi \fi \fi}
\def\poi#1#2#3#4#5#6#7{\def\un{#5#6#7}\def\deux{#6#7}
\def\trois{#2#4} \def\cinq{#3#4#5}
\ifx\un\empty {#1}_{#2}{#3 \hskip 0.15em}{#1}_{#4} \else
\ifx\deux\empty {#5}(#1_{#2}){#3 \hskip 0.15em}{#5}(#1_{#4}) \else
\ifx\trois\empty {#5}_{#6}(#1){#3 \hskip 0.15em}{#5}_{#7}(#1) \else
{#5_{#6}}(#1_{#2}){#3 \hskip 0.15em}{#5_{#7}}(#1_{#4}) \fi \fi \fi}
\def\rond{\raisebox{.3mm}{\scriptsize$\circ$}}
\def\mul{\raisebox{.3mm}{\scriptsize\hskip 0.15em$\times$\hskip 0.15em}}

\def\tk#1#2{{#2}\otimes _{#1}}

\def\ec#1#2#3#4#5{\def\un{#3#4#5}\def\deux{#3#5}\def\trois{#3}
\def\four{#2#4#5}\def\five{#2#5}\def\six{#2}\def\seven{#3#4}
\def\eight{#2#4} \def\nine{#2#3#4}
\ifx\nine\empty {\rm #1}_{#5} \else
\ifx\un\empty {\rm #1}({\goth #2}) \else
\ifx\deux\empty {\rm #1}({\goth #2}_{#4}) \else
\ifx\trois\empty {\rm #1}_{#5}({\goth #2}_{#4}) \else
\ifx\four\empty {\rm #1}(#3) \else
\ifx\five\empty {\rm #1}(#3_{#4}) \else
\ifx\six\empty {\rm #1}_{#5}(#3_{#4}) \else
\ifx\seven\empty {\rm #1}_{#5} ({\goth#2})\else
\ifx\eight\empty {\rm #1}_{#5}({#3})
\fi \fi \fi \fi \fi \fi \fi \fi \fi}
\def\hec#1#2#3#4#5{\def\un{#3#4#5}\def\deux{#3#5}\def\trois{#3}
\def\four{#2#4#5}\def\five{#2#5}\def\six{#2}\def\seven{#3#4}
\def\eight{#2#4} \def\nine{#2#3#4}
\ifx\nine\empty \hat{{\rm #1}}_{#5} \else
\ifx\un\empty \hat{{\rm #1}}({\goth #2}) \else
\ifx\deux\empty \hat{{\rm #1}}({\goth #2}_{#4}) \else
\ifx\trois\empty \hat{{\rm #1}}_{#5}({\goth #2}_{#4}) \else
\ifx\four\empty \hat{{\rm #1}}(#3) \else
\ifx\five\empty \hat{{\rm #1}}(#3_{#4}) \else
\ifx\six\empty \hat{{\rm #1}}_{#5}(#3_{#4}) \else
\ifx\seven\empty \hat{{\rm #1}}_{#5} ({\goth#2})  \else
\ifx\eight\empty \hat{{\rm #1}}_{#5}({#3})
\fi \fi \fi \fi \fi \fi \fi \fi \fi}

\def\es#1#2{\ec {#1}{}{#2}{}{}}

\def\ai#1#2#3{\def\deux{#2#3} \def\trois{#3} \def\quatre{#2} 
\ifx\deux\empty \es S{{\goth #1}}^{{\goth #1}} \else
\ifx\trois\empty \es S{{\goth #1}^{#2}}^{{\goth #1}^{#2}} \else
\ifx\quatre\empty \es S{{\goth #1}_{#3}}^{{\goth #1}_{#3}} \else
\es S{{\goth #1}_{#3}^{#2}}^{{\goth #1}_{#3}^{#2}} \fi \fi \fi}
\def\aii#1#2#3#4{\def\deux{#2#3} \def\trois{#3} \def\quatre{#2} 
\ifx\deux\empty \sy {#4}{{\goth #1}}^{{\goth #1}} \else
\ifx\trois\empty \sy {#4}{{\goth #1}^{#2}}^{{\goth #1}^{#2}} \else
\ifx\quatre\empty \sy {#4}{{\goth #1}_{#3}}^{{\goth #1}_{#3}} \else
\sy {#4}{{\goth #1}_{#3}^{#2}}^{{\goth #1}_{#3}^{#2}} \fi \fi \fi}

\def\hhom{\mathscr {H}\hskip -.15em om}

\def\Bbb{\mathbb}
\def\goth{\mathfrak}
\def\cal{\mathcal}

\def\gi#1#2#3#4{\def\trois{#3#4} \def\quatre{#4}\def\cinq{#3}\ifx\trois\empty {\rm i}_{#1,{\goth #2}}
\else \ifx\quatre\empty {\rm i}_{#1_{#3},{\goth #2}} \else\ifx\cinq\empty {\rm i}_{#1,{\goth #2}_{#4}} \else {\rm i}_{#1_{#3},{\goth #2}_{#4}} \fi \fi \fi}
\def\j#1#2{\def\deux{#2} \ifx\deux\empty {\rm rk}\hskip .125em{{\goth #1}} \else {\rm rk}\hskip .125em{{\goth #1}_{#2}} \fi}
\def\aj#1#2{\def\deux{#2} \ifx\deux\empty {\rm j}_{{\goth #1}} \else {\rm j}_{{\goth #1}_{#2}} \fi}
\def\an#1#2{\def\deux{#2} \ifx\deux\empty {\cal O}_{#1} \else {\cal O}_{#1,#2} \fi }
\def\han#1#2{\def\deux{#2} \ifx\deux\empty {\hat{{\cal O}}}_{#1} \else {\hat{{\cal O}}}_{#1,#2} \fi }
\def\dim{{\rm dim}\hskip .125em}

\def\ad{{\rm ad}\hskip .1em}
\def\Ad{{\rm Ad}\hskip .1em}   

\def\det{{\rm det}\hskip .125em}

\def\pr#1{{\rm pr}_{#1}}

\def\n{{\rm n}}
\def\s{{\rm s}}

\def\sy#1#2{{\rm S}^{#1}(#2)}

\def\mycom#1#2{\genfrac{}{}{0pt}{}{#1}{#2}}

\begin{document}

\title
[Commuting variety]
{On a variety related to the commuting variety of a reductive Lie algebra.}

\author[Jean-Yves Charbonnel]{Jean-Yves Charbonnel}
\address{Jean-Yves Charbonnel, Universit\'e Paris Diderot - CNRS \\
Institut de Math\'ematiques de Jussieu - Paris Rive Gauche\\
UMR 7586 \\ Groupes, repr\'esentations et g\'eom\'etrie \\
B\^atiment Sophie Germain \\ Case 7012 \\ 
75205 Paris Cedex 13, France}
\email{jean-yves.charbonnel@imj-prg.fr}



\date\today

\begin{abstract}
For a reductive Lie algbera over an algbraically closed field of charasteristic zero, 
we consider a Borel subgroup $B$ of its adjoint group, a Cartan subalgebra contained in 
the Lie algebra of $B$ and the closure $X$ of its orbit under $B$ in the Grassmannian.
The variety $X$ plays an important role in the study of the commuting variety. In this
note, we prove that $X$ is Gorenstein with rational singularities.  
\end{abstract}

\maketitle

\tableofcontents

\section{Introduction} \label{int}
In this note, the base field $\k$ is algebraically closed of characteristic $0$, 
${\goth g}$ is a reductive Lie algebra of finite dimension, $\rg$ is its rank,
$\dim {\goth g}=\rg + 2n$ and $G$ is its adjoint group. As usual, ${\goth b}$ denotes a 
Borel subalgebra of ${\goth g}$, ${\goth h}$ a Cartan subalgebra of ${\goth g}$, 
contained in ${\goth b}$, and $B$ the normalizer of ${\goth b}$ in $G$.

\subsection{Main results.} \label{int1}
Let $X$ be the closure in $\ec {Gr}g{}{}{\rg}$ of the orbit of ${\goth h}$ under the 
action of $B$. By a well known result, $G.X$ is the closure in $\ec {Gr}g{}{}{\rg}$
of the orbit of ${\goth h}$ under the action of $G$. By~\cite{Ric}, the commuting 
variety of ${\goth g}$ is the image by the canonical projection of the restriction to
$G.X$ of the canonical vector bundle of rank $2\rg$ over $\ec {Gr}g{}{}{\rg}$. So
$X$ and $G.X$ play an important role in the study of the commuting variety. As it is 
explained in \cite{CZ}, $X$ and $G.X$ play the same role for the so called generalized 
commuting varieties and the so called generalized isospectral commuting varieties. The 
main result of this note is the following theorem:

\begin{theo}\label{tint}
The variety $X$ is Gorenstein with rational singulatrities.
\end{theo}

An induction is used to prove this theorem. So we introduce the categories 
${\cal C}'_{{\goth t}}$ and ${\cal C}_{{\goth t}}$ with ${\goth t}$ a commutative Lie 
algebra of finite dimension. Their objects are nilpotent Lie algebras of finite 
dimension, normalized by ${\goth t}$ with additional conditions analogous to those of the
action of ${\goth h}$ in ${\goth u}$. In particular the minimal dimension of the objects 
in ${\cal C}_{{\goth t}}$ is the dimension of ${\goth t}$ and an object of dimension 
$\dim {\goth t}$ is a commutative algebra. The category ${\cal C}_{{\goth t}}$ is a 
full subcategory of ${\cal C}'_{{\goth t}}$. For ${\goth a}$ in ${\cal C}'_{{\goth t}}$, 
we consider the solvable Lie algebra ${\goth r} := {\goth t}+{\goth a}$ and $R$ the 
adjoint group of ${\goth r}$. Denoting by $X_{R}$ the closure in 
$\ec {Gr}r{}{}{\dim {\goth t}}$ of the orbit of ${\goth t}$ under $R$, we prove by 
induction on $\dim {\goth a}$ the following theorem:

\begin{theo}\label{t2int}
The variety $X_{R}$ is normal and Cohen-Macaulay.
\end{theo}

The result for the category ${\cal C}'_{{\goth t}}$ is easily deduced from the 
result for the category ${\cal C}_{{\goth t}}$ by Corollary~\ref{csa1}.
One of the key argument in the proof is the consideration of the fixed points under the 
action of $R$ in $X_{R}$. As a matter of fact, since the closure of all orbit under $R$
in $X_{R}$ contains a fixed point, $X_{R}$ is Cohen-Macaulay if so are the fixed points 
by openness of the set of Cohen-Macaulay points. Then, by Serre's normality criterion, it
suffices to prove that $X_{R}$ is smooth in codimension $1$. For that purpose the 
consideration of the restriction to $X_{R}$ of the tautological vector bundle of rank 
$\dim {\goth t}$ over $\ec {Gr}r{}{}{\dim {\goth t}}$ is very useful.

For the study of the fixed points, we introduce Property $({\bf P})$ and Property
$({\bf P}_{1})$ for the objects of ${\cal C}'_{{\goth t}}$: 

\begin{itemize}
\item Property $({\bf P})$ for ${\goth a}$ in ${\cal C}'_{{\goth t}}$ says that for $V$ 
in $X_{R}$, contained in the centralizer ${\goth r}^{s}$ of an element $s$ of 
${\goth t}$, $V$ is in the closure of the orbit of ${\goth t}$ under the centralizer 
$R^{s}$ of $s$ in $R$,
\item Property $({\bf P}_{1})$ for ${\goth a}$ in ${\cal C}'_{{\goth t}}$ says that 
for $V$ in $X_{R}$ normalized by ${\goth t}$ and such that $V\cap {\goth t}$ is the
center of ${\goth r}$, then the non zero weights of ${\goth t}$ in $V$ are linearly 
independent.
\end{itemize}

Property $({\bf P}_{1})$ for ${\goth a}$ results from Property $({\bf P})$ for 
${\goth a}$ and Property $({\bf P})$ for ${\goth a}$ results from Property $({\bf P}_{1})$
for ${\goth a}$ and Property $({\bf P})$ for the objects of ${\cal C}'_{{\goth t}}$ of 
dimension smaller than $\dim {\goth a}$. So, the main result for the objects of 
${\cal C}'_{{\goth t}}$ is the following proposition:

\begin{prop}\label{pint}
The objects of ${\cal C}'_{{\goth t}}$ have Property $({\bf P})$.
\end{prop}

From this proposition, we deduce some structure property for the points of $X_{R}$.

The second part of Theorem~\ref{tint}, that is Gorensteinness property and Rational 
singularities, is obtained by considering a subcategory ${\cal C}_{{\goth t},*}$ of 
${\cal C}_{{\goth t}}$. This category is defined by an additional condition on 
the objects. The main point for ${\goth a}$ in ${\cal C}_{{\goth t},*}$ is the following 
result:

\begin{prop}\label{p2int}
Let $k\geq 2$ be an integer. Denote by ${\cal E}^{(k)}$ the $R$-equivariant vector 
subbundle of $X_{R}\times {\goth r}^{k}$ whose fiber at ${\goth t}$ is ${\goth t}^{k}$.
Then there exists on the smooth locus of ${\cal E}^{(k)}$ a regular differential form of 
top degree without zero. 
\end{prop}

From Proposition~\ref{p2int} and Theorem~\ref{t2int}, we deduce that ${\cal E}^{(k)}$ 
and $X_{R}$ are Gorenstein with rational singularities.

This note is organized as follows. In Section~\ref{sa}, categories ${\cal C}'_{{\goth t}}$
and ${\cal C}_{{\goth t}}$ are introduced for some space ${\goth t}$. In particular, 
${\goth u}$ is an object of ${\cal C}_{{\goth h}}$. In Subsection~\ref{sa3}, we define 
Property $({\bf P})$ for the objects of ${\cal C}'_{{\goth t}}$ and we deduce some
result on the structure of points of $X_{R}$. In Subsection~\ref{sa4}, we define
Property $({\bf P}_{1})$ for the objects of ${\cal C}'_{{\goth t}}$ and we prove
that Property $({\bf P})_{1}$ is a consequence of Property $({\bf P})$. In 
Subsection~\ref{sa5}, we give some geometric constructions to prove 
Property $({\bf P})$ by induction on the dimension of ${\goth a}$. At last, in
Subsection~\ref{sa6}, we prove Proposition~\ref{pint}. In particular, the proof of 
\cite[Lemma 4.4,(i)]{CZ} is completed. In Section~\ref{sav}, we are interested 
in the singular locus of $X_{R}$. In Subsection~\ref{sav3}, regularity in codimension $1$
is proved with some additional properties analogous to those of~\cite[Section 3]{CZ}.
Moreover, the constructions of Subsection~\ref{sa5} are used to prove the results by 
induction on the dimension of ${\goth a}$. In Section~\ref{ns}, Cohen-Macaulayness 
property is proved by induction. In Section~\ref{rss}, the category 
${\cal C}_{{\goth t},*}$ is introduced and Proposition~\ref{p2int} is proved. Then with 
some results given in the appendix, we finish the proof of Theorem~\ref{tint}.

\subsection{Notations} \label{int2}
$\bullet$ An algebraic variety is a reduced scheme over $\k$ of finite type. For $X$ an
algebraic variety, its smooth locus is denoted by $X_{\loc}$.

$\bullet$ Set $\k^{*} := \k \setminus \{0\}$. For $V$ a vector space, its dual is 
denoted by $V^{*}$. 

$\bullet$
All topological terms refer to the Zariski topology. If $Y$ is a subset of a topological
space $X$, denote by $\overline{Y}$ the closure of $Y$ in $X$. For $Y$ an open subset
of the algebraic variety $X$, $Y$ is called {\it a big open subset} if the codimension
of $X\setminus Y$ in $X$ is at least $2$. For $Y$ a closed subset of an algebraic 
variety $X$, its dimension is the biggest dimension of its irreducible components and its
codimension in $X$ is the smallest codimension in $X$ of its irreducible components. For 
$X$ an algebraic variety, $\an X{}$ is its structural sheaf, $\k[X]$ is the algebra of 
regular functions on $X$, $\k(X)$ is the field of rational functions on $X$ when $X$ 
is irreducible and $\Omega _{X}$ is the sheaf of regular differential forms of top
degree on $X$ when $X$ is smooth and irreducible.

$\bullet$
If $E$ is a subset of a vector space $V$, denote by span($E$) the vector subspace of
$V$ generated by $E$. The grassmannian of all $d$-dimensional subspaces of $V$ is denoted
by Gr$_d(V)$. 

$\bullet$ For ${\goth a}$ a Lie algebra,$V$ a subspace of ${\goth a}$ and $x$ in 
${\goth a}$, $V^{x}$ denotes the centralizer of $x$ in $V$. For $A$ a subgroup of 
the group of automorphisms of ${\goth a}$, $A^{x}$ denotes the centralizer of $x$ in 
$A$. An element $x$ of ${\goth g}$ is regular if ${\goth g}^{x}$
has dimension $\rg$ and the set of regular elements of ${\goth g}$ is denoted by 
${\goth g}_{\r}$.

$\bullet$ The Lie algebra of an algberaic torus is also called a torus. In this note, 
a torus denoted by a gothic letter means the Lie algebra of an algebraic torus. 

$\bullet$ For ${\goth a}$ a Lie algebra, the Lie algebra of derivations of 
${\goth a}$ is denoted by ${\mathrm {Der}}({\goth a})$. By definition 
${\mathrm {Der}({\goth a})}$ is the Lie algebra of the group ${\mathrm {Aut}({\goth a})}$
of the automorphisms of ${\goth a}$.

$\bullet$
Let ${\goth b}$ be a Borel subalgebra of ${\goth g}$, ${\goth h}$ a Cartan 
subalgebra of ${\goth g}$ contained in ${\goth b}$ and ${\goth u}$ the nilpotent 
radical of ${\goth b}$.  

\section{On solvable algebras} \label{sa}
Let ${\goth t}$ be a vector space of positive dimension $d$. Denote by 
$\tilde{{\cal C}}_{{\goth t}}$ the subcategory of the category of finite dimensional
Lie algebras whose objects are finite dimensional nilpotent Lie algebras 
${\goth a}$ such that there exists a morphism 
$$\xymatrix{ {\goth t} \ar[rr]^{\varphi _{{\goth a}}} && {\mathrm {Der}({\goth a})}}$$
whose image is the Lie algebra of a subtorus of ${\mathrm {Aut}}({\goth a})$. For 
${\goth a}$ and ${\goth a}'$ in $\tilde{{\cal C}}_{{\goth t}}$, a
morphism $\psi $ from ${\goth a}$ to ${\goth a}'$ is a morphism of Lie algebras such 
that $\psi \rond \varphi _{{\goth a}}(t)= \varphi _{{\goth a}'}(t)\rond \psi $ for all
$t$ in ${\goth t}$. For $x$ in ${\goth t}$, $x$ is a semisimple derivation of 
${\goth a}$. Denote by ${\cal R}_{{\goth t},{\goth a}}$ the set of weights of ${\goth t}$
in ${\goth a}$. Let ${\cal C}'_{{\goth t}}$ be the full subcategory of objects 
${\goth a}$ of $\tilde{{\cal C}}_{{\goth t}}$ verifying the following conditions:
\begin{itemize}
\item [{\rm (1)}] $0$ is not in ${\cal R}_{{\goth t},{\goth a}}$,
\item [{\rm (2)}] for $\alpha $ in ${\cal R}_{{\goth t},{\goth a}}$, the weight space of 
weight $\alpha $ has dimension $1$,
\item [{\rm (3)}] for $\alpha $ in ${\cal R}_{{\goth t},{\goth a}}$, 
$\k \alpha \cap ({\cal R}_{{\goth t},{\goth a}}\setminus \{\alpha \})$ is empty.
\end{itemize}
For ${\goth a}$ in ${\cal C}'_{{\goth t}}$ and ${\goth a}'$ a subalgebra of 
${\goth a}$, invariant under the adjoint action of ${\goth t}$, ${\goth a}'$ is in 
${\cal C}'_{{\goth t}}$. Denote by ${\cal C}_{{\goth t}}$ the full subcategory of objects
${\goth a}$ of ${\cal C}'_{{\goth t}}$ such that $\varphi _{{\goth a}}$ is an embedding.
For example ${\goth u}$ is in ${\cal C}_{{\goth h}}$. 

For ${\goth a}$ in $\tilde{{\cal C}}_{{\goth t}}$, denote by 
${\goth r}_{{\goth t},{\goth a}}$ the solvable algebra ${\goth t}+{\goth a}$, 
$\pi _{{\goth t},{\goth a}}$ the quotient morphism from ${\goth r}_{{\goth t},{\goth a}}$
to ${\goth t}$, $R_{{\goth t},{\goth a}}$ the adjoint group of 
${\goth r}_{{\goth t},{\goth a}}$, $A_{{\goth t},{\goth a}}$ the connected 
closed subgroup of $R_{{\goth t},{\goth a}}$ whose Lie algebra is $\ad {\goth a}$, 
$X_{R_{{\goth t},{\goth a}}}$ the closure in $\ec {Gr}r{}{{\goth t},{\goth a}}{d}$ of the
orbit of ${\goth t}$ under $R_{{\goth t},{\goth a}}$ and ${\cal E}_{{\goth t},{\goth a}}$
the restriction to $X_{R_{{\goth t},{\goth a}}}$ of the tautological vector bundle over 
$\ec {Gr}r{}{{\goth t},{\goth a}}d$. The variety $X_{R_{{\goth t},{\goth a}}}$ is called
{\it the main variety related to ${\goth r}_{{\goth t},{\goth a}}$}. For $\alpha $ in 
${\cal R}_{{\goth t},{\goth a}}$, let ${\goth a}^{\alpha }$ be the weight space of weight
$\alpha $ under the action of ${\goth t}$ in ${\goth a}$. 

In the following subsections, a vector space ${\goth t}$ of positive dimension $d$ and an 
object ${\goth a}$ of ${\cal C}'_{{\goth t}}$ are fixed. We set:
$$ {\cal R} := {\cal R}_{{\goth t},{\goth a}}, \qquad 
{\goth r} := {\goth r}_{{\goth t},{\goth a}} \qquad \pi := \pi _{{\goth t},{\goth a}},
\qquad R := R_{{\goth t},{\goth a}}, \qquad A := A_{{\goth t},{\goth a}}, \qquad 
n := \dim {\goth a}.$$
Let ${\goth z}$ be the orthogonal complement of ${\cal R}$ in ${\goth t}$ and $d^{\#}$ 
its codimension in ${\goth t}$. Then $n\geq d^{\#}$.

\subsection{General remarks on ${\cal C}'_{{\goth t}}$} \label{sa1}
For $x$ in ${\goth r}$, we say that $x$ is semisimple if so is $\ad x$ and $x$ is 
nilpotent if so is $\ad x$. For ${\goth s}$ a commutative subalgebra 
of ${\goth r}$, we say that ${\goth s}$ is a torus if $\ad {\goth s}$ is the Lie algebra 
of a subtorus of ${\mathrm {GL}}({\goth r})$.  

\begin{lemma}\label{lsa1}
Let $x$ be in ${\goth r}$ and ${\goth s}$ a commutative subalgebra of ${\goth r}$.

{\rm (i)} The center of ${\goth r}$ is equal to ${\goth z}$.  

{\rm (ii)} The element $x$ is semisimple if and only if $R.x\cap {\goth t}$ is not empty.

{\rm (iii)} The element $x$ is nilpotent if and only if $x$ is in ${\goth z}+{\goth a}$.

{\rm (iv)} The algebra ${\goth a}$ is in ${\cal C}_{{\goth t}}$ if and only if 
${\goth z}=\{0\}$. In this case, $x$ has a unique decomposition $x=x_{\s}+x_{\n}$ with 
$[x_{\s},x_{\n}]=0$, $x_{\s}$ semisimple and $x_{\n}$ nilpotent.

{\rm (v)} The algebra ${\goth s}$ is a torus if and only if 
${\goth s}\cap {\goth a} = \{0\}$ and $\pi ({\goth s})$ is a subtorus of ${\goth t}$. In 
this case, ${\goth s}$ and $\pi ({\goth s})$ are conjugate under $R$.
\end{lemma}

\begin{proof}
By definition $\ad {\goth r}_{{\goth t},{\goth a}}$ is an algebraic solvable 
subalgebra of ${\goth {gl}}({\goth r}_{{\goth t},{\goth a}})$ and $\ad {\goth t}$ is 
a maximal subtorus of $\ad {\goth r}_{{\goth t},{\goth a}}$.

(i) Let ${\goth z}'$ be the center of ${\goth r}$. As $[{\goth t},{\goth z}']=\{0\}$, 
$$ {\goth z}' = {\goth z}'\cap {\goth t} \oplus 
\bigoplus _{\alpha \in {\cal R}} {\goth z}'\cap {\goth a}^{\alpha }. $$
So, by Condition (1), ${\goth z}'$ is contained in ${\goth t}$. For $t$ in ${\goth t}$, 
$t$ is in ${\goth z}'$ if and only if $\alpha (t)=0$ for all $\alpha $ in 
${\cal R}_{{\goth t},{\goth a}}$, whence ${\goth z}'={\goth z}$.

(ii) As the elements of ${\goth t}$ are semisimple by defintion, the condition is 
sufficient since the set of semisimple elements of ${\goth r}$ is invariant under the 
adjoint action of $R$. Suppose that $x$ is semisimple. By ~\cite[Ch. VII]{Hu}, for some 
$g$ in $R$, $\Ad g(x)$ is in $\ad {\goth t}$, whence $g(x)$ is in ${\goth t}$ by (i).

(iii) As $\ad {\goth a}$ is the set of nilpotent elements of 
$\ad {\goth r}$, $x$ is in ${\goth z}+{\goth a}$ if and only if it is nilpotent by (i).

(iv) By definition, ${\goth z}$ is the kernel of $\varphi _{{\goth a}}$. Hence 
${\goth z}=\{0\}$ if and only if ${\goth a}$ is in ${\cal C}_{{\goth t}}$. As 
$\ad {\goth r}$ is an algebraic subalgebra of ${\mathrm {gl}}({\goth r})$, it contains 
the components of the Jordan decomposition of $\ad x$. As a result, when ${\goth a}$ is 
in ${\cal C}_{{\goth t}}$, $x$ has a unique decomposition $x=x_{\s}+x_{\n}$ with 
$[x_{\s},x_{\n}]=0$, $x_{\s}$ semisimple and $x_{\n}$ nilpotent. 

(v) Suppose that ${\goth s}$ is a torus. By (i), ${\goth s}\cap {\goth a}=\{0\}$ and 
by~\cite[Ch. VII]{Hu}, for some $g$ in $R$, $\ad g({\goth s})$
is contained in $\ad {\goth t}$ since $\ad {\goth t}$ is a maximal torus of 
$\ad {\goth r}$. Then, by (i), $g({\goth s})$ is a subtorus of ${\goth t}$. Moreover, 
$g({\goth s})=\pi ({\goth s})$ since $g(y)-y$ is in ${\goth a}$ for all $y$ in 
${\goth r}$. Conversely, if ${\goth s}\cap {\goth a} = \{0\}$ and $\pi ({\goth s})$ is a 
subtorus of ${\goth t}$, $\ad {\goth s}$ is conjugate to the subtorus  
$\ad \pi ({\goth s})$ of $\ad {\goth t}$ by~\cite[Ch. VII]{Hu} so that ${\goth s}$ and 
$\pi ({\goth s})$ are conjugate under $R$.
\end{proof}

Denoting by ${\goth t}^{\#}$ a complement to ${\goth z}$ in ${\goth t}$, ${\goth a}$ is 
an object of ${\cal C}_{{\goth t}^{\#}}$ since 
$\varphi _{{\goth a}}({\goth t})=\varphi _{{\goth a}}({\goth t}^{\#})$ and the 
restriction of $\varphi _{{\goth a}}$ to ${\goth t}^{\#}$ is injective. Set
${\goth r}^{\#} := {\goth t}^{\#}+{\goth a}$ and denote by $R^{\#}$ the adjoint group
of ${\goth r}^{\#}$. Let $X_{R^{\#}}$ be the closure in 
$\ec {Gr}{}{{{\goth r}^{\#}}}{}{{d^{\#}}}$ of the orbit of ${\goth t}^{\#}$ under
$R^{\#}$.

\begin{coro}\label{csa1}
All element of $X_{R}$ is a commutative algebra containing ${\goth z}$. Moreover, 
the map 
$$ \xymatrix{ X_{R^{\#}} \ar[rr] && X_{R}}, \qquad V \longmapsto V \oplus {\goth z}$$
is an isomorphism.
\end{coro}

\begin{proof}
As the set of commutative subalgebras of dimension $d$ of ${\goth r}$ is a closed subset
of $\ec {Gr}r{}{}d$ containing ${\goth t}$ and invariant under $R$, all element of 
$X_{R}$ is a commutative algebra. According to Lemma~\ref{lsa1}(i), all element of 
$R.{\goth t}$ contains ${\goth z}$ and so does all element of $X_{R}$. For $g$ in 
$R$, denote by $\overline{g}$ the image of $g$ in $R^{\#}$ by the restriction morphism.
Then 
$$g({\goth t}) = \overline{g}({\goth t}^{\#})+{\goth z} \quad  \text{and} \quad 
\overline{g}({\goth t}^{\#}) = g({\goth t}) \cap {\goth r}^{\#} .$$ 
Hence the map
$$ \xymatrix{ X_{R^{\#}} \ar[rr] && X_{R}}, \qquad V \longmapsto V \oplus {\goth z}$$
is an isomorphism whose inverse is the map $V\mapsto V\cap {\goth r}^{\#}$.
\end{proof}

For ${\goth a}$ of dimension $d^{\#}$, 
${\cal R} := \{\poi {\beta }1{,\ldots,}{{d^{\#}}}{}{}{}\}$, and for $I$ subset of 
$\{1,\ldots,d^{\#}\}$, denote $X_{R,I}$ the image of $\k^{I}$ by the map
$$ \xymatrix{\k^{I} \ar[rr] && X_{R}}, \qquad 
(z_{i},\; i\in I) \longmapsto {\goth z}\oplus 
{\mathrm {span}}(\{t_{i}+z_{i}x_{i}, \; i\in I\}) \oplus 
\bigoplus _{i \not \in I} {\goth a}^{\beta _{i}}$$
with $x_{i}$ in ${\goth a}^{\beta _{i}}$ for $i=1,\ldots,d^{\#}$ and 
$\poi t1{,\ldots,}{d^{\#}}{}{}{}$ in ${\goth t}$ such that 
$\beta _{i}(t_{j})=\delta _{i,j}$ for $1\leq i,j\leq d^{\#}$, with $\delta _{i,j}$ the
Kronecker symbol. 

\begin{lemma}\label{l2sa1}
Suppose that ${\goth a}$ has dimension $d^{\#}$. Denote by 
$\poi {\beta }1{,\ldots,}{{d^{\#}}}{}{}{}$ the elements of ${\cal R}$. 

{\rm (i)} The algebra ${\goth a}$ is commutative.

{\rm (ii)} The set $X_{R}$ is the union of $X_{R,I}, \; I \subset \{1,\ldots,d^{\#}\}$.
\end{lemma}

\begin{proof}
(i) As ${\goth z}$ has codimension $d^{\#}$ in ${\goth t}$, 
$\poi {\beta }1{,\ldots,}{{d^{\#}}}{}{}{}$ are linearly independent. Hence for $i\neq j$, 
$\beta _{i}+\beta _{j}$ is not in ${\cal R}$. As a result, ${\goth a}$ is commutative.

(ii) According to Corollary~\ref{csa1}, we can suppose $d=d^{\#}$ so that 
$\poi t1{,\ldots,}{d}{}{}{}$ is the dual basis of $\poi {\beta }1{,\ldots,}{d}{}{}{}$.
For $I$ subset of $\{1,\ldots,d\}$, denote by $I'$ the complement to $I$ in 
$\{1,\ldots,d\}$ and ${\goth z}_{I'}$ the orthogonal complement to 
$\beta _{i},\; i\in I'$ in ${\goth t}$ and set:
$$ V_{I} := {\goth z}_{I'} \oplus \bigoplus _{i\in I'} {\goth a}^{\beta _{i}} .$$
By (i), for $i$ in $I$,
$$ \exp(z_{1}\ad x_{1} + \cdots + z_{d}\ad x_{d})(t_{i}) = t_{i}-z_{i}x_{i} .$$ 
Hence $X_{R,I}$ is the orbit of $V_{I}$ under $A$ and its closure in $X_{R}$ is 
the union of $X_{R,J}, \; J \subset I$. As a result, $X_{R}$ is the
union of $X_{R,I},\; I\subset \{1,\ldots,d\}$ since $X_{R,\{1,\ldots,d\}}$ is the orbit 
of ${\goth t}$ under $A$.
\end{proof}

\subsection{On some subsets of ${\cal R}$} \label{sa2}
For $\alpha $ in ${\cal R}$, let $x_{\alpha }$ be in ${\goth a}^{\alpha }\setminus \{0\}$.
For $\Lambda $ subset of ${\cal R}$, denote by ${\goth t}_{\Lambda }$ the intersection 
of the kernels of its elements and set:
$$ {\goth a}_{\Lambda } := \bigoplus _{\alpha \in \Lambda } {\goth a}^{\alpha }
\quad  \text{and} \quad {\goth r}_{\Lambda } := {\goth t}\oplus {\goth a}_{\Lambda } .$$
When $\Lambda $ has only one element $\alpha $, set
${\goth t}_{\alpha } := {\goth t}_{\Lambda }$.

\begin{defi}\label{dsa2}
Let $\Lambda $ be a subset of ${\cal R}$. We say that $\Lambda $ is a complete subset of 
${\cal R}$ if it contains all element of ${\cal R}$ whose kernel contains 
${\goth t}_{\Lambda }$ 
\end{defi}

For $\Lambda $ complete subset of ${\cal R}$, ${\goth a}_{\Lambda }$ is a subalgebra
of ${\goth a}$ and ${\goth r}_{\Lambda }$ is a subalgebra of ${\goth r}$. In particular,
${\goth a}_{\Lambda }$ is in ${\cal C}'_{{\goth t}}$. In this case, denote by 
$R_{\Lambda }$ the connected closed subgroup of $R$ whose Lie algebra is 
$\ad {\goth r}_{\Lambda }$.

\begin{lemma}\label{lsa2}
Let $\Lambda $ be a complete subset of ${\cal R}$, strictly contained in ${\cal R}$.
Then ${\goth a}_{\Lambda }$ is contained in an ideal ${\goth a}'$ of ${\goth r}$ of
dimension $\dim {\goth a}-1$ and contained in ${\goth a}$.
\end{lemma}

\begin{proof}
As $\Lambda $ is complete and strictly contained in ${\cal R}$, ${\goth a}_{\Lambda }$ is
a subalgebra of ${\goth r}$, strictly contained in ${\goth a}$. Then, by Lie's Theorem, 
there is a sequence
$$ {\goth a}_{\Lambda }=\poi {{\goth a}}0{\subset \cdots \subset}{m}{}{}{}={\goth a}$$
of subalgebras of ${\goth r}$ such that ${\goth a}_{i}$ is an ideal of codimension $1$ of 
${\goth a}_{i+1}$ for $i=0,\ldots,m-1$, whence the lemma.
\end{proof}

For $s$ in ${\goth t}$, denote by $\Lambda _{s}$ the subset of elements of ${\cal R}$
whose kernel contains $s$.

\begin{lemma}\label{l2sa2}
Let $s$ be in ${\goth t}$. 

{\rm (i)} The centralizer ${\goth r}^{s}$ of $s$ in ${\goth r}$ is the direct sum of 
${\goth t}$ and ${\goth a}_{\Lambda _{s}}$.

{\rm (ii)} The center of ${\goth r}^{s}$ is equal to ${\goth t}_{\Lambda _{s}}$. 
\end{lemma}

\begin{proof}
By definition, $\Lambda _{s}$ is a complete subset of ${\cal R}$. Let $x$ be in 
${\goth r}$. Then $x$ has a unique decomposition
$$ x = x_{0} + \sum_{\alpha \in {\cal R}} c_{\alpha }x_{\alpha }$$
with $x_{0}$ in ${\goth t}$ and $c_{\alpha },\alpha \in {\cal R}$ in $\k$. 

(i) Since $s$ is in ${\goth t}$, $x$ is in ${\goth r}^{s}$ if and 
only if $c_{\alpha }=0$ for $\alpha \in {\cal R}\setminus \Lambda _{s}$, whence the 
assertion.

(ii) The algebra ${\goth a}_{\Lambda _{s}}$ is in ${\cal C}'_{{\goth t}}$ and 
${\goth t}_{\Lambda _{s}}$ is the orthogonal complement to $\Lambda _{s}$ in ${\goth t}$.
So, by (i) and Lemma~\ref{lsa1}(i), ${\goth t}_{\Lambda _{s}}$ is the center of 
${\goth r}^{s}$.
\end{proof}

\subsection{Property $({\bf P})$ for objects of ${\cal C}_{{\goth t}}$.} \label{sa3}
Let ${\bf T}$ be the connected closed subgroup of $R$ whose Lie algebra is 
$\ad {\goth t}$. For $s$ in ${\goth t}$, denote by $X_{R}^{s}$ the subset of elements of 
$X_{R}$ contained in ${\goth r}^{s}$ and $\overline{R^{s}.{\goth t}}$ the closure 
in $\ec {Gr}r{}{}d$ of the orbit of ${\goth t}$ under $R^{s}$. Then 
$\overline{R^{s}.{\goth t}}$ is contained in $X_{R}^{s}$. 

\begin{defi}\label{dsa3}
Say that ${\goth a}$ has Property $({\bf P})$ if $X_{R}^{s}$ is equal to
$\overline{R^{s}.{\goth t}}$ for all $s$ in ${\goth t}$.
\end{defi}

By Corollary~\ref{csa1}, ${\goth a}$ has Property $({\bf P})$ if and only if the object 
${\goth a}$ of ${\cal C}_{{\goth t}^{\#}}$ has Property $({\bf P})$.

\begin{lemma}\label{lsa3}
If ${\goth a}$ has dimension $d^{\#}$, then ${\goth a}$ has Property $({\bf P})$.
\end{lemma}

\begin{proof}
According to Corollary~\ref{csa1}, we can suppose $d=d^{\#}$. Denote by 
$\poi {\beta }1{,\ldots,}{d}{}{}{}$ the elements of ${\cal R}$. Then 
$\poi {\beta }1{,\ldots,}{d}{}{}{}$ is a basis of ${\goth t}^{*}$. Let
$\poi t1{,\ldots,}{d}{}{}{}$ be the dual basis, $s$ in ${\goth t}$ and $V$ in 
$X_{R}^{s}$. By Lemma~\ref{l2sa1}(ii), for some subset $I$ of $\{1,\ldots,d\}$, $V$ is 
in $X_{R,I}$. Then for some $(z_{i},\; i\in I)$, 
$$ V = {\mathrm {span}}(\{t_{i}+z_{i}x_{i} \; i\in I\}) \oplus \bigoplus _{i\in I'}
{\goth a}^{\beta _{i}}$$ 
with $I'$ the complement to $I$ in $\{1,\ldots,d\}$ and $x_{i}$ in 
${\goth a}^{\beta _{i}}$ for $i=1,\ldots,d$. Setting
$$ I'' := I' \cup \{i \in I \; \vert \; z_{i} \neq 0\},$$ 
for $i$ in $\{1,\ldots,d\}$, $i$ is in $I''$ if and only if $\beta _{i}(s)=0$. So, 
by Lemma~\ref{lsa2}(i),
$$ {\goth r}^{s} = {\goth t} \oplus \bigoplus _{i\in I''} {\goth a}^{\beta _{i}} .$$
Then by Lemma~\ref{l2sa1}(ii), $V$ is in $\overline{R^{s}.{\goth t}}$.
\end{proof}

By definition, an {\it algebraic subalgebra} ${\goth k}$ of ${\goth r}$ is the 
semi-direct product of a torus ${\goth s}$ contained in ${\goth k}$ and 
${\goth k}\cap {\goth a}$.

\begin{lemma}\label{l2sa3}
Suppose that ${\goth a}$ has Property $({\bf P})$. Let $V$ be in $X_{R}$, $x$ in
$V$ and $y$ in ${\goth r}$ such that $\ad y$ is the semisimple component of $\ad x$.
Then the center of ${\goth r}^{y}$ is contained in $V$. 
\end{lemma}

\begin{proof}
By Corollary~\ref{csa1}, we can suppose ${\goth a}$ in ${\cal C}_{{\goth t}}$ so that $y$
is the semisimple component of $x$ by Lemma~\ref{lsa1}(iv). By Lemma~\ref{lsa1}(ii), for 
some $g$ in $R$, $g(y)$ is in ${\goth t}$. Denote by ${\goth z}_{g(y)}$ the center of 
${\goth r}^{g(y)}$. By Lemma~\ref{l2sa2}(ii), ${\goth z}_{g(y)}$ is contained in 
${\goth t}$. As $V$ is a commutative algebra, $g(V)$ is in $X_{R}^{g(y)}$. So, by 
Property $({\bf P})$, ${\goth z}_{g(y)}$ is contained in $g(V)$ since ${\goth z}_{g(y)}$ 
is in $k({\goth t})$ for all $k$ in $R^{g(y)}$, whence the lemma.
\end{proof}

\begin{coro}\label{csa3}
Suppose that ${\goth a}$ has Property $({\bf P})$. Let $V$ be in $X_{R}$. Then $V$ is a 
commutative algebraic subalgebra of ${\goth r}$ and for some subset $\Lambda $ of 
${\cal R}$, the biggest torus contained in $V$ is conjugate to ${\goth t}_{\Lambda }$ 
under $R$.
\end{coro}

\begin{proof}
According to Corollary~\ref{csa1}, $V$ is a commutative subalgebra of ${\goth r}$ and we 
can suppose $d=d^{\#}$. Let ${\goth s}$ be the set of semisimple elements of $V$. Then 
${\goth s}$ is a subspace of $V$. By Lemma~\ref{l2sa3}, $V$ contains the semisimple 
components of its elements so that $V$ is the direct sum of ${\goth s}$ and 
$V\cap {\goth a}$. Let $s$ be in ${\goth s}$ such that the center of ${\goth r}^{s}$ has 
maximal dimension. After conjugation by an element of $R$, we can suppose that $s$ is in 
${\goth t}$. By Lemma~\ref{l2sa2}(ii), ${\goth t}_{\Lambda _{s}}$ is the center of 
${\goth r}^{s}$. Hence, by Lemma~\ref{l2sa3}, ${\goth t}_{\Lambda _{s}}$ is contained in 
${\goth s}$. Suppose that the inclusion is strict. A contradiction is expected. Let $s'$ 
be in ${\goth s}\setminus {\goth t}_{\Lambda _{s}}$. Since $V$ is 
contained in ${\goth r}^{s}$, for some $g$ in $R^{s}$, $g(s')$ is in ${\goth t}$. 
Moreover, $g({\goth s})$ is the set of semisimple elements of $g(V)$ and 
${\goth t}_{\Lambda _{s}}$ is contained in $g({\goth s})$. Denoting by $\Lambda '$ the 
set of elements of $\Lambda _{s}$ whose kernel contains $g(s')$, for some $z$ in 
$\k^{*}$, $\Lambda '$ is the set of elements of ${\cal R}$ such that 
$\alpha (s+zg(s'))=0$. By Lemma~\ref{l2sa3}, ${\goth t}_{\Lambda '}$ is contained in 
$g(V)$. So, by minimality of $\vert \Lambda  _{s}\vert$, $\Lambda '=\Lambda _{s}$ and 
$g(s')$ is in ${\goth t}_{\Lambda _{s}}$, whence the contradiction since $g(s')$ is in 
$g({\goth s})\setminus {\goth t}_{\Lambda _{s}}$. As a result, 
${\goth t}_{\Lambda _{s}}={\goth s}$ and $V={\goth t}_{\Lambda _{s}}+V\cap {\goth a}$, 
whence the corollary. 
\end{proof}

\subsection{Fixed points in $X_{R}$ under ${\bf T}$ and $R$} \label{sa4}
For $V$ subspace of dimension $d$ of ${\goth r}$, denote by ${\cal R}_{V}$ the 
set of elements $\beta $ of ${\cal R}$ such that ${\goth a}^{\beta }$ 
is contained in $V$, $r_{V}$ the rank of ${\cal R}_{V}$ and ${\goth z}_{V}$ its 
orthogonal complement in ${\goth t}$ so that $\dim {\goth z}_{V}=d-r_{V}$.
As $\ec {Gr}r{}{}d$ and $X_{R}$ are projective varieties, the actions of ${\bf T}$ and 
$R$ in these varieties have fixed points since ${\bf T}$ and $R$ are connected and 
solvable. 

\begin{defi}\label{dsa4}
We say that ${\goth a}$ has Property $({\bf P}_{1})$ if for  $V$ fixed point under 
${\bf T}$ in $X_{R}$ such that $V\cap {\goth t}={\goth z}$, 
$r_{V}= \vert {\cal R}_{V}\vert$.
\end{defi}

\begin{lemma}\label{lsa4}
Suppose that ${\goth a}$ has Property $({\bf P})$. Let $V$ be in $\ec {Gr}r{}{}d$.

{\rm (i)} The element $V$ is a fixed point under ${\bf T}$ in $X_{R}$ if 
and only if $V$ is a commutative subalgebra of ${\goth r}$ and 
$$ V = {\goth z}_{V} \oplus \bigoplus _{\beta \in {\cal R}_{V}} {\goth a}^{\beta } .$$
In this case, $r_{V} = \vert{\cal R}_{V} \vert$.

{\rm (ii)} The element $V$ is a fixed point under $R$ in $X_{R}$ if and only
if $V$ is a commutative ideal of ${\goth r}$ and ${\goth z}$ is the orthogonal complement
of ${\cal R}_{V}$ in ${\goth t}$. In this case, $r_{V}=\vert {\cal R}_{V}\vert = d^{\#}$.
\end{lemma}

\begin{proof}
If $V$ is a fixed point under ${\bf T}$, 
$$ V = V\cap {\goth t} \oplus \bigoplus _{\beta \in {\cal R}_{V}} {\goth a}^{\beta } .$$

(i) Suppose that $V$ is a fixed point under ${\bf T}$ in $X_{R}\setminus \{{\goth t}\}$. 
Then ${\cal R}_{V}$ is not empty. Let $s$ be an element of ${\goth z}_{V}$ such that 
$\beta (s)\neq 0$ if $\beta $ is not a linear combination of elements of ${\cal R}_{V}$. 
Then $V$ is contained in ${\goth r}^{s}$. So, by Property $({\bf P})$, $V$ is in 
$\overline{R^{s}.{\goth t}}$. By Lemmma~\ref{l2sa2}(i), ${\goth z}_{V}$ is the center of 
${\goth r}^{s}$. Hence ${\goth z}_{V}$ is contained in $V$ and 
${\goth z}_{V}=V\cap {\goth t}$ since $V\cap {\goth t}$ is contained in ${\goth z}_{V}$. 
As a result, ${\goth z}_{V}$ has dimension $d-\vert {\cal R}_{V} \vert$ and 
$r_{V}=\vert {\cal R}_{V} \vert$.

Conversely, suppose that $V$ is a commutative algebra and 
$$ V = {\goth z}_{V} \oplus \bigoplus _{\beta \in {\cal R}_{V}} {\goth a}^{\beta } .$$
Set:
$$ {\goth a}_{V} := \bigoplus _{\beta \in {\cal R}_{V}} {\goth a}^{\beta }, \qquad 
{\goth r}_{V} := {\goth t} \oplus {\goth a}_{V} .$$
Then ${\goth a}_{V}$ is a commutative Lie algebra and ${\goth a}_{V}$ is in 
${\cal C}'_{{\goth t}}$. Moreover, ${\goth z}_{V}$ is the center of ${\goth r}_{V}$
by Lemma~\ref{lsa1}(i). By Lemma~\ref{l2sa1}(ii), $V$ is in the closure
of the orbit of ${\goth t}$ under the action of the adjoint group of ${\goth r}_{V}$
in $\ec {Gr}r{}Vd$, whence the assertion.

(ii) The element $V$ of $\ec {Gr}r{}{}d$ is a fixed point under $R$ if and only if 
$V$ is an ideal of ${\goth r}$. So, by (i), the condition is sufficient. Suppose that $V$
is a fixed point under the action of $R$ in $X_{R}$. By (i),
$$ V = {\goth z}_{V} \oplus \bigoplus _{\beta \in {\cal R}_{V}} {\goth a}^{\beta } .$$
As $V$ is an ideal of ${\goth r}$, ${\goth z}_{V}$ is contained in the kernel of all
elements of ${\cal R}$ so that ${\goth z}_{V}={\goth z}$. In particular, 
$\vert {\cal R}_{V} \vert = d^{\#}$ and the elements of ${\cal R}_{V}$ are linearly 
independent.
\end{proof}

\subsection{On some varieties related to $X_{R}$} \label{sa5}
Let ${\goth a}'$ be an ideal of ${\goth r}$ of dimension $\dim {\goth a}-1$ and contained
in ${\goth a}$. As a subalgebra of ${\goth a}$ normalized by ${\goth t}$, ${\goth a}'$
is in ${\cal C}'_{{\goth t}}$. Denote by ${\goth r}'$ the subalgebra 
${\goth t}+{\goth a}'$ of ${\goth r}$, $A'$ and $R'$ the connected closed subgroups of 
$R$ whose Lie algebras are $\ad {\goth a}'$ and $\ad {\goth r'}$ respectively. Let 
$X_{R'}$ be the closure in $\ec {Gr}r{}{}d$ of the orbit of ${\goth t}$ under $R'$ and 
$\alpha $ the element of ${\cal R}$ such that
$$ {\goth a} = {\goth a}' \oplus {\goth a}^{\alpha } .$$
For $\delta $ in ${\cal R}$ denote again by $\delta $ the character of ${\bf T}$ whose 
differential at the identity is $\ad x \longmapsto \delta (x)$.

Setting:
$$ {\goth G}_{d-1,d,d,d+1} := 
\ec {Gr}r{}{}{d-1}\times \ec {Gr}r{}{}d\times \ec {Gr}r{}{}d\times \ec {Gr}r{}{}{d+1}
\quad  \text{and} \quad 
{\goth G}_{d-1,d,d+1} := 
\ec {Gr}r{}{}{d-1}\times \ec {Gr}r{}{}d\times \ec {Gr}r{}{}{d+1},$$
denote by $\thetaup _{\alpha }$ and $\thetaup '_{\alpha }$ the maps
$$ \xymatrix{ \k\times A' \ar[rr]^{\thetaup _{\alpha }} && {\goth G}_{d-1,d,d,d+1}}, 
\quad
(z,g) \longmapsto (g.{\goth t}_{\alpha },g.{\goth t},g\exp(z \ad x_{\alpha }).{\goth t}
,g.({\goth t}+{\goth a}^{\alpha })) ,$$
$$ \xymatrix{ A' \ar[rr]^{\thetaup '_{\alpha }} && {\goth G}_{d-1,d,d+1}}, 
\quad
g \longmapsto (g.{\goth t}_{\alpha },g.{\goth t},g.({\goth t}+{\goth a}^{\alpha })) .$$
Let $I_{\alpha }$ and $S_{\alpha }$ be the closures in $\ec {Gr}r{}{}{d-1}$ and 
$\ec {Gr}r{}{}{d+1}$ of the orbits of ${\goth t}_{\alpha }$ and 
${\goth t}+{\goth a}^{\alpha }$  under $A'$ respectively.

\begin{lemma}\label{lsa5}
Let $\Gamma $ and $\Gamma '$ be the closures in ${\goth G}_{d-1,d,d,d+1}$ and 
${\goth G}_{d-1,d,d+1}$ of the images of $\thetaup _{\alpha }$ and $\thetaup '_{\alpha }$.

{\rm (i)} The varieties $\Gamma $ and $\Gamma '$ have dimension $n$ and $n-1$ 
respectively. Moreover, they are invariant under the diagonal actions of $R'$ in 
${\goth G}_{d-1,d,d,d+1}$ and ${\goth G}_{d-1,d,d+1}$.

{\rm (ii)} The image of $\Gamma $ by the first, second, third and fourth projections are 
equal to $I_{\alpha }$, $X_{R'}$, $X_{R}$, $S_{\alpha }$ respectively. 

{\rm (iii)} The set $\Gamma '$ is the image of $\Gamma $ by the projection
$$ \xymatrix{{\goth G}_{d-1,d,d,d+1} \ar[rr]^{\varpi } && {\goth G}_{d-1,d,d+1}}, \qquad
(V_{1},V',V,W) \longmapsto (V_{1},V',W) .$$

{\rm (iv)} For all $(V_{1},V',V,W)$ in $\Gamma $, $V_{1}$ is contained in $V'\cap V$ 
and $V'+V$ is contained in $W$.

{\rm (v)} Let $(V_{1},V',V,W)$ be in $\Gamma $ such that $V'$ is contained in 
${\goth t}_{\alpha }+{\goth a}'$. Then $W$ is contained in 
${\goth t}_{\alpha }+{\goth a}$.

{\rm (vi)} Let $(V_{1},V',V,W)$ be in $\Gamma $ such that $V'$ is not contained in 
${\goth t}_{\alpha }+{\goth a}$. Then $W$ is not commutative.
\end{lemma}

\begin{proof}
(i) The maps $\thetaup _{\alpha }$ and $\thetaup '_{\alpha }$ are injective since 
${\goth t}$ is the normalizer of ${\goth t}$ in ${\goth r}$ by Condition (1), whence 
$\dim \Gamma = n$ and $\dim \Gamma ' = n-1$. For $(z,g,k)$ in $\k\times A'\times A'$, 
$\thetaup _{\alpha }(z,kg) = k.\thetaup _{\alpha }(z,g)$ and 
$\thetaup '_{\alpha }(kg) = k.\thetaup '_{\alpha }(g)$. Hence $\Gamma $ and $\Gamma '$ 
are invariant under the diagonal action of $A'$ in ${\goth G}_{d-1,d,d,d+1}$ and 
${\goth G}_{d-1,d,d+1}$. Let $k$ be in ${\bf T}$. For all $(z,g)$ in $\k \times A'$, 
$$ kg.{\goth t}_{\alpha } = kgk^{-1}({\goth t}_{\alpha }), \qquad 
kg.{\goth t} = kgk^{-1}({\goth t}), $$ $$ 
kg.({\goth t}+{\goth a}^{\alpha }) = kgk^{-1}.({\goth t}+{\goth a}^{\alpha }), \qquad 
kg\exp (z\ad x_{\alpha }).{\goth t} = 
kgk^{-1}\exp (z\alpha (k)\ad x_{\alpha }).{\goth t}$$
so that the images of $\thetaup _{\alpha }$ and $\thetaup '_{\alpha }$ are invariant 
under ${\bf T}$, whence the assertion.

(ii) Since $\ec {Gr}r{}{}d$, $\ec {Gr}r{}{}{d-1}$, $\ec {Gr}r{}{}{d+1}$ are projective 
varieties, the images of $\Gamma $ by the first, second, third and
fourth projections are closed subsets of their target varieties. Since the image of 
$\thetaup _{\alpha }$ is contained in the closed subset 
$I_{\alpha }\times X_{R'}\times X_{R}\times S_{\alpha }$ of ${\goth G}_{d-1,d,d,d+1}$,
they are contained in $I_{\alpha }$, $X_{R'}$, $X_{R}$ and $S_{\alpha }$ respectively. By
definition, $R'.{\goth t}_{\alpha }$, $R'.{\goth t}$ and 
$R'.({\goth t}+{\goth a}^{\alpha })$ are contained in the images of $\Gamma $ by the 
first, second and fourth projections and $R.{\goth t}$ is contained in the image of 
$\Gamma $ by the third projection since $A$ is the image of $\k\times A'$ by the map
$(z,g) \mapsto g\exp(z\ad x_{\alpha })$, whence the assertion.

(iii) As $\ec {Gr}r{}{}d$ is a projective variety, $\varpi (\Gamma )$ is a closed 
subset of ${\goth G}_{d-1,d,d+1}$ containing the image of $\thetaup '_{\alpha }$ since 
$\varpi \rond \thetaup _{\alpha }(z,g)=\thetaup '_{\alpha }(g)$ for all $(z,g)$ in 
$\k\times A'$. Moreover, $\Gamma $ is contained in $\varpi ^{-1}(\Gamma ')$, whence
$\Gamma '=\varpi (\Gamma )$.

(iv) The subset $\widetilde{\Gamma }$ of elements $(V_{1},V',V,W)$ of 
${\goth G}_{d-1,d,d,d+1}$ such that $V_{1}$ is contained in $V'$ and $V$ and such that 
$V'$ and $V$ are contained in $W$, is closed. For all $(z,g)$ in $\k\times A'$,
$$ g\exp (z\ad x_{\alpha }).({\goth t}+{\goth a}^{\alpha }) =
g.({\goth t}+{\goth a}^{\alpha }) .$$
Hence the image of $\thetaup _{\alpha }$ and $\Gamma $ are contained in 
$\widetilde{\Gamma }$ so that $V_{1}$ and $V+V'$ are contained in $V'\cap V$ and $W$ 
respectively for all $(V_{1},V',V,W)$ in $\Gamma $.

(v) Denote by $\Gamma _{*}$ the closure in $\ec {Gr}r{}{}d\times \ec {Gr}r{}{}{d+1}$
of the image of the map
$$ \xymatrix{ A' \ar[rr]^{\thetaup _{\alpha ,*}} && 
\ec {Gr}r{}{}d \times \ec {Gr}r{}{}{d+1}},
\qquad g \longmapsto (g({\goth t}),g({\goth t}+{\goth a}^{\alpha })) .$$
For all $(T_{1},T',T,T_{2})$ in the image of $\thetaup _{\alpha }$, $(T',T_{2})$ is in 
the image of $\thetaup _{\alpha ,*}$. Then $\Gamma _{*}$ is the image of $\Gamma $
by the projection
$$ \xymatrix{ {\goth G}_{d-1,d,d,d+1} \ar[rr] && 
\ec {Gr}r{}{}d \times \ec {Gr}r{}{}{d+1}}, \qquad 
(T_{1},T',T,T_{2}) \longmapsto (T',T_{2}) .$$

Denote by $\tau $ the quotient morphism 
$$ \xymatrix{ {\goth r} \ar[rr]^{\tau } && {\goth r}/{\goth a}' = 
{\goth t}+{\goth a}^{\alpha }}.$$  
For $g$ in $A'$ and $x$ in ${\goth r}$, $\tau \rond g(x) = \tau (x)$. Set:
$$ X := \{(g,t,z,z',v,w) \in A'\times {\goth t}_{\alpha }\times \k ^{2}\times 
{\goth r}'\times {\goth r} ; \vert \; v = g(zs + t), \; w = g(zs+t+z'x_{\alpha }) \}$$
and denote by $Y$ the closure in ${\goth r}'\times {\goth r}$ of the image of $X$ by the
canonical projection 
$$ \xymatrix{ A'\times {\goth t}_{\alpha }\times \k ^{2}\times {\goth r}'\times {\goth r}
\ar[rr] && {\goth r}'\times {\goth r} } .$$
As for all $(g,t,z,z',v,w)$ in $X$, 
$$ \tau (v) = zs + t \quad  \text{and} \quad \tau (w) = zs + t + z'x_{\alpha },$$
$$\alpha \rond \pi \rond \tau (v) = \alpha \rond \pi \rond \tau (w)$$
for all $(v,w)$ in $Y$. By definition, for all $(T,T')$ in $\Gamma _{*}$, 
$T\times T'$ is contained in $Y$. By hypothesis, $V'$ is contained in the kernel of 
$\alpha \rond \pi $ and $(V',W)$ is in $\Gamma _{*}$. Hence $W$ is contained in 
the kernel of $\alpha \rond \pi $.

(vi) Denote by $\Gamma '_{*}$ the closure in 
${\goth G}_{d-1,d,d,d+1}\times \ec {Gr}r{}{}1$ of the image of the map
$$ \xymatrix{ \k \times A' \ar[rr]^{\thetaup '_{\alpha ,*}} && 
{\goth G}_{d-1,d,d,d+1}\times \ec {Gr}r{}{}1},
\qquad (z,g) \longmapsto (\thetaup _{\alpha }(z,g),g({\goth a}^{\alpha })) $$
and $\Gamma '_{**}$ the closure in $\ec {Gr}{}{{\goth r}'}{}d\times \ec {Gr}r{}{}1$ of 
the image of the map
$$ \xymatrix{ A' \ar[rr] && \ec {Gr}{}{{\goth r}'}{}d \times \ec {Gr}r{}{}1},
\qquad g \longmapsto (g({\goth t}),g({\goth a}^{\alpha }) .$$
For all $(T_{1},T',T,T_{2},T'_{2})$ in the image of $\thetaup '_{\alpha ,*}$, $T'+T'_{2}$
is contained in $T_{2}$. Then so is it for all $(T_{1},T',T,T_{2},T'_{2})$ in 
$\Gamma '_{*}$. As ${\goth G}_{d-1,d,d,d+1}$ and $\ec {Gr}r{}{}1$ are projective 
varieties, $\Gamma $ and $\Gamma '_{**}$ are the images of $\Gamma '_{*}$ by the 
projections
$$ \xymatrix{ {\goth G}_{d-1,d,d,d+1}\times \ec {Gr}r{}{}1 \ar[rr] &&
{\goth G}_{d-1,d,d,d+1}}, \qquad
(T_{1},T',T,T_{2},T'_{2}) \longmapsto (T_{1},T',T,T_{2}) ,$$
$$ \xymatrix{ {\goth G}_{d-1,d,d,d+1}\times \ec {Gr}r{}{}1 \ar[rr] &&
\ec {Gr}{}{{\goth r}'}{}d\times \ec {Gr}r{}{}1}, \qquad
(T_{1},T',T,T_{2},T'_{2}) \longmapsto (T',T'_{2}) $$
respectively. 

Set: 
$$ X' := \{(g,t,z,v,w) \in 
A'\times {\goth t}\times \k \times {\goth r}'\times {\goth r} \; \vert \;
v = g(t), \; w = g(z x_{\alpha }) \} $$
and denote by $Y'$ the closure in ${\goth r}'\times {\goth r}$ of the image of $X'$ by
the canonical projection 
$$ \xymatrix{A'\times {\goth t}\times \k \times {\goth r}'\times {\goth r} \ar[rr] &&
{\goth r}'\times {\goth r}}.$$
As for all $(g,t,z,v,w)$ in $X'$, 
$$ [v,w] = g([t,zx_{\alpha }]) = \alpha (t) g(zx_{\alpha }) = 
\alpha \rond \pi (v) w ,$$
$[v,w] = \alpha \rond \pi (v) w$ for all $(v,w)$ in $Y'$. By definition, for all
$(T,T')$ in $\Gamma '_{**}$, $T\times T'$ is contained in $Y'$. For some $W'$ in 
$\ec {Gr}r{}{}1$, $(V_{1},V',V,W,W')$ is in $\Gamma '_{*}$. By hypothesis, $V'$ is not 
contained in the kernel of $\alpha \rond \pi $. Hence, for some $v$ in $V'$ and 
$w$ in $W\setminus \{0\}$, $\alpha \rond \pi (v)\neq 0$ and 
$[v,w]=\alpha \rond \pi (v)w$. 
\end{proof}

\begin{coro}\label{csa5}
Suppose that ${\goth a}'$ has Property $({\bf P})$. Let $s$ be in ${\goth t}$ such that
${\goth r}^{s}$ is contained in ${\goth a}'$ and $(V_{1},V',V,W)$ be in $\Gamma $ such 
that $V$ is contained in ${\goth r}^{s}$ and $[s,V']$ is contained in $V'$.

{\rm (i)} If $W$ is not commutative then $V'=V$ and $V$ is in 
$\overline{R^{s}.{\goth t}}$.

{\rm (ii)} Suppose that for some $v$ in ${\goth a}$, $s+v$ is in $V$. Then 
$V'=V$ and $V$ is in $\overline{R^{s}.{\goth t}}$.
\end{coro}

\begin{proof}
By Lemma~\ref{lsa5}(ii), $V$ and $V'$ are in $X_{R}$ and $X_{R'}$ respectively.

(i) If $V'=V$, $V$ is in $\overline{R^{s}.{\goth t}}$ by Property $({\bf P})$ for 
${\goth a}'$. Suppose $V'\neq V$. A contradiction is expected. Then, by 
Lemma~\ref{lsa5}(iv), for some $x$ and $y$ in $W$,
$$ V = V_{1}\oplus \k x, \quad V' = V_{1} \oplus \k y, \quad 
W = V_{1} \oplus \k x \oplus \k y .$$
Moreover, as $V$ is contained in ${\goth r}^{s}$ and $[s,V']\subset V'$, $W$ is contained
in ${\goth r}'$ and we can choose $y$ so that $[s,y]\in \k y$. Since $V$ and $V'$ are 
commutative subalgebras of ${\goth r}$, $[x,y]\neq 0$. We have two cases to consider:
\begin{itemize}
\item [{\rm (a,1)}] $V'$ is contained in ${\goth r}^{s}$,
\item [{\rm (a,2)}] $V'$ is not contained in ${\goth r}^{s}$.
\end{itemize}

(a,1) By Property $({\bf P})$ for ${\goth a}'$, $s$ is in $V'$. Hence $s=ty+v$ for some 
in $(t,v)$ in $\k\times V_{1}$. As $V$ is a commutative subalgebra of ${\goth r}^{s}$, 
containing $V_{1}$ and $x$, 
$$ 0 = [x,s] = t[x,y] .$$
Then $s=v$ is in $V_{1}$, whence a contradiction since $\alpha (s) \neq 0$ and $V_{1}$ is
contained in ${\goth t}_{\alpha }+{\goth a}'$ by Lemma~\ref{lsa5}(ii).

(a,2) For some $a$ in $\k^{*}$, $[s,y]=ay$. Then $y$ is in ${\goth a}'$ and $V'$ is 
contained in ${\goth t}_{\alpha }+{\goth a}'$ since so is $V_{1}$. As a result, 
by Lemma~\ref{lsa5}(v), $V$ and $W$ are contained in ${\goth t}_{\alpha }+{\goth a}'$ 
since $V$ is contained in ${\goth a}'$. As $[s,[x,y]]=a[x,y]$, $[x,y]=by$ for some $b$ in
$\k^{*}$ since the eigenspace of eigenvalue $a$ of the restriction of $\ad s$ to $V'$ is 
generated by $y$. Then $\ad x$ is not nilpotent. Let $x_{\s}$ be in ${\goth r}'$ such 
that $\ad x_{\s}$ is the semisimple component of $\ad x$. Then $x_{\s}$ is in 
${\goth t}_{\alpha }+{\goth a}'$, $[s,x_{s}]=0$ and $[x_{\s},V_{1}]=\{0\}$ since 
$[s,x]=0$ and $[x,V_{1}]=\{0\}$. Moreover, $[ax_{\s}-bs,y]=0$. Then $ax_{\s}-bs$ is a 
semisimple element of ${\goth r}'$ such that $[ax_{\s}-bs,V']=\{0\}$. As it is conjugate 
under $R'$ to an element of ${\goth t}$ by Lemma~\ref{lsa1}(ii), by Property $({\bf P})$ 
for ${\goth a}'$, $ax_{\s}-bs$ is in $V'$, whence a contradiction since $V'$ is contained
in ${\goth t}_{\alpha }+{\goth a}'$ and $ax_{\s}-bs$ is not in 
${\goth t}_{\alpha }+{\goth a}'$.

(ii) If $V=V'$, $V$ is in $\overline{R^{s}.{\goth t}}$ by Property $({\bf P})$ for 
${\goth a}'$. Suppose $V\neq V'$. A contradiction is expected. As $V$ is contained in 
${\goth r}^{s}$, $[s,v]=0$. Let $x$ and $y$ be as in (i). As $V_{1}$ is contained in 
${\goth t}_{\alpha }+{\goth a}'$, $s+v$ is not in $V_{1}$ since $\alpha (s)\neq 0$. So we
can choose $s+v=x$. By (i), $W$ is commutative. Then $[s+v,y]=0$ and $[\ad s,\ad y]=0$ 
since $\ad s$ is the semisimple component of $\ad (s+v)$. 
Hence, by Lemma~\ref{lsa1}(i), $[s,y]=0$ since $[s,y]$ is in ${\goth a}$. As a result,
$V'$ is contained in ${\goth r}^{s}$ since so is $V_{1}$. So, by Property $({\bf P})$ for
${\goth a}'$, $s$ is in $V'$ and $W$ is not commutative by Lemma~\ref{lsa5}(vi), whence 
a contradiction. 
\end{proof}

For $(T_{1},T',T_{2})$ in $\Gamma '$, denote by $\Gamma _{T_{1},T',T_{2}}$ the subset of 
elements $(T_{1},T',T,T_{2})$ of ${\goth G}_{d-1,d,d,d+1}$ such that $T$ is contained in 
$T_{2}$ and contains $T_{1}$. Then $\Gamma _{T_{1},T',T_{2}}$ is a closed subvariety of 
${\goth G}_{d-1,d,d,d+1}$, isomorphic to ${\Bbb P}^{1}(\k)$. Let $(V_{1},V',V,W)$ be a 
fixed point under ${\bf T}$ of $\Gamma $.

\begin{lemma}\label{l2sa5}
{\rm (i)} For some affine open neighborhood $\Omega $ of $(V_{1},V',W)$ in $\Gamma '$,
$\Omega $ is invariant under ${\bf T}$.

{\rm (ii)} For $i=0,\ldots,n-2$, there exist $Y_{i}$ and $O_{i}$ such that
\begin{itemize}
\item [{\rm (a)}] $Y_{i}$ is an irreducible closed subset of dimension $n-1-i$ of 
$\Omega $, containing $(V_{1},V',W)$ and invariant under ${\bf T}$,
\item [{\rm (b)}] $O_{i}$ is a locally closed subvariety of dimension $n-1-i$ of $A'$,
invariant under ${\bf T}$ by conjugation,
\item [{\rm (c)}] $\thetaup '_{\alpha }(O_{i})$ is contained in $Y_{i}$ and 
$(V_{1},V',V,W)$ is in the closure of $\thetaup _{\alpha }(\k\times O_{i})$ in $\Gamma $.
\end{itemize}

{\rm (iii)} There exist a smooth projective curve $C$, an action of ${\bf T}$ on 
$C$, $\poi x1{,\ldots,}{m}{}{}{}$ in $C$ and two morphisms
$$ \xymatrix{ C\setminus \{\poi x1{,\ldots,}{m}{}{}{}\} \ar[rr]^{\mu } && A'}, 
\qquad \xymatrix{ C \ar[rr]^{\nu } && \Gamma '}$$
satisfying the following conditions:
\begin{itemize}
\item [{\rm (a)}] $\poi x1{,\ldots,}{m}{}{}{}$ are the fixed points under ${\bf T}$ in 
$C$, 
\item [{\rm (b)}] for $g$ in ${\bf T}$ and $x$ in 
$C\setminus \{\poi x1{,\ldots,}{m}{}{}{}\}$, $\mu (g.x) = g\mu (x)g^{-1}$ and 
$\nu (g.x)=g.\nu (x)$, 
\item [{\rm (c)}] $\nu (x_{1})=(V_{1},V',W)$,
\item [{\rm (d)}] $(V_{1},V',V,W)$ is in the closure of the image of 
$\k\times (C\setminus \{\poi x1{,\ldots,}{m}{}{}{}\})$ by the map
$(z,x) \longmapsto \thetaup _{\alpha }(z,\mu (x))$.
\end{itemize}
\end{lemma}

\begin{proof}
(i) As $\Gamma '$ is a projective variety with a ${\bf T}$ action and $(V_{1},V',W)$
is a fixed point under ${\bf T}$, there exists an affine open neighborhood $\Omega $
of $(V_{1},V',W)$ in $\Gamma '$, invariant under ${\bf T}$. 

(ii) Prove the assertion by induction on $i$. For $i=0$, $Y_{i}=\Omega $ and 
$O_{i}$ is the inverse image of $\Omega $ by $\thetaup '_{\alpha }$. Suppose that 
$Y_{i}$ and $O_{i}$ are known. Let $Y'_{i}$ be the closure in $\Gamma $ of 
$\thetaup _{\alpha }(\k\times O_{i})$. By Condition (c), $Y'_{i}$ is invariant under 
${\bf T}$ and $\thetaup _{\alpha }(\k\times O_{i})$ is a ${\bf T}$-invariant dense
subset of $Y'_{i}$. So, it contains an ${\bf T}$-invariant dense open subset $O'_{i}$ of 
$Y'_{i}$. As $\thetaup '_{\alpha }$ is an orbital injective morphism, 
$\thetaup '_{\alpha }(O_{i})$ is a dense open subset of $Y_{i}$. Set:
$$Z':=Y'_{i}\setminus O'_{i}, \quad Z := Y_{i}\setminus \theta _{\alpha }(O'_{i}),
\quad Z_{*} := \Omega \cap (\varpi (Z)\cup Z') .$$ 
Then $Z_{*}$ is a ${\bf T}$-invariant closed subset of $Y_{i}$, containing 
$(V_{1},V',W)$. 

Denote by $Z_{**}$ the union of irreducible components of dimension
$\dim Y_{i}-1$ of $Z_{*}$ and $I$ the union of the ideals of definition in $\k[Y_{i}]$ of 
the irreducible components of $Z_{**}$. Let $p$ be in $\k[Y_{i}]\setminus I$, 
semiinvariant under ${\bf T}$ and such that $p((V_{1},V',W))=0$. Denote by $Y'_{i+1}$
an irreducible component of the nullvariety of $p$ in $Y'_{i}\cap \varpi ^{-1}(\Omega )$,
containing $(V_{1},V',V,W)$ and $Y_{i+1}$ the closure in $\Omega $ of 
$\varpi (Y'_{i+1})$. Then $Y_{i+1}$ has dimension $n-i-1$ and its intersection with
$\thetaup '_{\alpha }(O_{i})$ is not empty so that 
$O_{i+1}:={\thetaup '_{\alpha }}^{-1}(O_{i}\cap\thetaup '_{\alpha }(O_{i}))$
is a nonempty locally closed subset of dimension $n-i-1$ of $A'$. Moreover, 
$Y_{i+1}$ and $O_{i+1}$ are invariant under ${\bf T}$ since $p$ is semiinvariant
under ${\bf T}$. As $\thetaup _{\alpha }(\k\times O_{i+1})$ is the intersection of 
$Y'_{i+1}$ and $\thetaup _{\alpha }(\k\times O_{i})$, it is dense in $Y'_{i+1}$ so that 
$(V_{1},V',V,W)$ is in the closure of $\thetaup _{\alpha }(\k\times O_{i+1})$ and 
$(V_{1},V',W)$ is in $Y_{i+1}$.  

(iii) Let $\overline{Y_{n-2}}$ be the closure of $Y_{n-2}$, $C$ its normalization and 
$\nu $ the normalization morphism. Then $C$ is a smooth projective curve. As $Y_{n-2}$ is
invariant under ${\bf T}$, so is $\overline{Y_{n-2}}$ and there is an action of 
${\bf T}$ on $C$ such that $\nu $ is an equivariant morphism. As the restriction of 
$\thetaup '_{\alpha }$ to $O_{n-2}$ is an isomorphism onto a dense open subset of 
$Y_{n-2}$, the actions of ${\bf T}$ on $\overline{Y_{n-2}}$ and $C$ are not trivial
since $\thetaup '_{\alpha }$ is equivariant under the actions of ${\bf T}$. As
a result, ${\bf T}$ has an open orbit $O_{*}$ in $\overline{Y_{n-2}}$ and 
$\overline{Y_{n-2}}\setminus O_{*}$ is the set of fixed points under ${\bf T}$ of 
$\overline{Y_{n-2}}$ since $\overline{Y_{n-2}}$ has dimension $1$. Hence the 
restriction of $\nu $ to $\nu ^{-1}(O_{*})$ is an isomorphism, 
$C\setminus \nu ^{-1}(O_{*})$ is finite, its elements are fixed under ${\bf T}$ and 
there exists a ${\bf T}$-equivariant morphism $\mu $ from $\nu ^{-1}(O_{*})$ to 
$A'$ such that $\thetaup '_{\alpha }\rond \mu = \nu $. As $(V_{1},V',W)$ is a fixed point 
under ${\bf T}$, for some $x_{1}$ in $C\setminus \nu ^{-1}(O_{*})$, 
$\nu (x_{1})=(V_{1},V',W)$ since $(V_{1},V',W)$ is in $\nu (C)$. Moreover, 
$(V_{1},V',V,W)$ is in the closure of the map
$$ \xymatrix{ \k \times (C\setminus \nu ^{-1}(O_{*})) \ar[rr] && \Gamma }, 
\qquad (z,x) \longmapsto \thetaup _{\alpha }(z,\mu (x))$$
since it is in $\overline{\thetaup _{\alpha }(\k\times O_{n-2})}$.
\end{proof}

Denote by $\eta $ the morphism 
$$ \xymatrix{ \k \times (C\setminus \{\poi x1{,\ldots,}{m}{}{}{}\}) \ar[rr]^{\eta } && 
\Gamma }, \qquad (z,x) \longmapsto \thetaup _{\alpha }(z,\mu (x)) $$
and $\Delta $ the closure of the graph of $\eta $ in 
${\Bbb P}^{1}(\k)\times C\times \Gamma $. Let $\upsilon $ be the restriction to 
$\Delta $ of the canonical projection
$$ \xymatrix{ {\Bbb P}^{1}(\k)\times C\times \Gamma \ar[rr] && {\Bbb P}^{1}(\k)\times C}$$
and for $(z,x)$ in ${\Bbb P}^{1}(\k)\times C$, let $F_{z,x}$ be the subset of $\Gamma $
such that $\{(z,x)\}\times F_{z,x}$ is the fiber of $\upsilon $ at $(z,x)$. We have 
an action of ${\bf T}$ in ${\Bbb P}^{1}(\k)$ given by
$$ t.z := \left \{ \begin{array}{ccc} \alpha (t)z & \mbox{ if } & z \in \k^{*} \\
z & \mbox{ if } & z \in \{0,\infty \} \end{array}  \right. .$$

\begin{lemma}\label{l3sa5}
Let $\Delta _{\nu }$ be the graph of $\nu $.

{\rm (i)} The set $\Delta _{\nu }$ is the image of $\Delta $ by the map
$(z,x,y)\mapsto (x,\varpi (y))$.

{\rm (ii)} For $t$ in ${\bf T}$ and $(z,x,y)$ in $\Delta $, $t.(z,x,y) := (t.z,t.x,t.y)$
is in $\Delta $.

{\rm (iii)} For $(z,x)$ in ${\Bbb P}^{1}(\k)\times C$, $\eta $ is regular at $(z,x)$ 
if and only if $F_{z,x}$ has dimension $0$. In this case, $\vert F_{z,x} \vert=1$.

{\rm (iv)} For $(z,x)$ in 
${\Bbb P}^{1}(\k)\times C\setminus \{0,\infty \}\times \{\poi x1{,\ldots,}{m}{}{}{}\}$, 
$\eta $ is regular at $(z,x)$.

{\rm (v)} For $i=1,\ldots,m$, there exists a regular map $\eta _{i}$ from 
${\Bbb P}^{1}(\k)$ to $\Gamma $ such that $\eta _{i}(z)=\eta (z,x_{i})$ for 
all $z$ in $\k^{*}$. Moreover, its image is contained in 
$\varpi ^{-1}(\{\nu (x_{i})\})\cap \Gamma $. 
\end{lemma}

\begin{proof}
(i) As ${\Bbb P}^{1}(\k)$ and $\ec {Gr}r{}{}d$ are projective varieties, the image of 
$\Delta $ by the map $(z,x,y) \mapsto (x,\varpi (y))$ is a closed subset of 
$C\times \Gamma '$ contained in $\Delta _{\nu }$ since 
$\varpi \rond \eta (z,x)=\nu (x)$ for all $(z,x)$ in 
$\k\times (C\setminus \{\poi x1{,\ldots,}{m}{}{}{}\})$, whence the assertion since 
the inverse image of $\Delta _{\nu }$ by this map is a closed subset of 
${\Bbb P}^{1}(\k)\times C\times \Gamma $ containing the graph of $\eta $.

(ii) From the equality
$$ t \exp (z\ad x_{\alpha }) t^{-1} = \exp (\alpha (t)z\ad x_{\alpha })$$
for all $(t,z)$ in ${\bf T}\times \k$, we deduce the equality
$$ t.\eta (z,x) = t.\thetaup _{\alpha }(z,\mu (x))= 
\thetaup _{\alpha }(\alpha (t)z,\mu (t.x) = \eta (t.z,t.x)$$
for all $(t,z,x)$ in ${\bf T}\times \k \times (C\setminus \{\poi x1{,\ldots,}{m}{}{}{}\})$
since $\thetaup _{\alpha }$ and $\mu $ are ${\bf T}$-equivariant, whence the assertion.

(iii) As $\Gamma $ is a projective variety, $\upsilon $ is a projective morphism. 
Moreover, it is birational since $\Delta $ is the closure of the graph of $\eta $. So, by 
Zariski's Main Theorem \cite[Ch. III, Corollary 11.4]{Ha}, the fibers of $\upsilon $ are 
connected of dimension $0$ or $1$ since ${\Bbb P}^{1}(\k)\times C$ is normal of 
dimension $2$. Let $(z,x)$ be in ${\Bbb P}^{1}(\k)\times C$ such that $F_{z,x}$ dimension
$0$. There exists a neighborhood $O_{z,x}$ of $(z,x)$ in ${\Bbb P}^{1}(\k)\times C$ such 
that $F_{y}$ has dimension $0$ for $y$ in $O_{z,x}$. In other words, the restriction of 
$\upsilon $ to $\upsilon ^{-1}(O_{z,x})$ is a quasi finite morphism. Moreover, it is 
birational and surjective. So, again by Zariski's Main Theorem~\cite[\S 9]{Mu}, it is an 
isomorphism. Hence $\eta $ is regular at $(z,x)$. Conversely, if $\eta $ is regular at 
$(z,x)$, $\eta (z,x)$ is an isolated point in $F_{z,x}$, whence $F_{z,x}=\{\eta (z,x)\}$ 
since $F_{z,x}$ is connected.  

(iv) The variety $\k\times (C\setminus \{\poi x1{,\ldots,}{m}{}{}{}\})$ is an open
subset of the smooth variety ${\Bbb P}^{1}(\k)\times C$ and $\Gamma $
is a projective variety. Hence $\eta $ has a regular extension to a big open subset of 
${\Bbb P}^{1}(\k)\times C$ by~\cite[Ch. 6, Theorem 6.1]{Sh}. By Condition (a) of 
Lemma~\ref{l2sa5}(iii), $\{0,\infty \}\times \{\poi x1{,\ldots,}{m}{}{}{}\}$ is the set 
of fixed points under ${\bf T}$ in ${\Bbb P}^{1}(\k)\times C$ and by (ii),
$t.\eta (z,x)=\eta (t.z,t.x)$ for all $(t,z,x)$ in 
${\bf T}\times {\Bbb P}^{1}(\k)\times (C\setminus \{\poi x1{,\ldots,}{m}{}{}{}\})$. Hence
$\eta $ is regular on 
$P^{1}(\k)\times C\setminus \{0,\infty \}\times \{\poi x1{,\ldots,}{m}{}{}{}\}$.

(v) The restriction of $\eta $ to $\k^{*}\times \{x_{i}\}$ is a regular 
map from a dense open subset of the smooth variety ${\Bbb P}^{1}(\k)\times \{x_{i}\}$
to the projective variety $\Gamma $. So, again by~\cite[Ch. 6, Theorem 6.1]{Sh}, 
this map has regular extension to ${\Bbb P}^{1}(\k)\times \{x_{i}\}$, whence the 
assertion by (i).
\end{proof}

Let $I$ be the set of indices such that $\nu (x_{i})=(V_{1},V',W)$. Denote by $S$ the 
image of $\Delta $ by the canonical projection 
$\xymatrix{{\Bbb P}^{1}(\k)\times C\times \Gamma \ar[r] & \Gamma }$, 
$S_{\n}$ its normalization, $\sigma $ the normalization morphism, $S^{{\bf T}}$ and 
$S_{\n}^{{\bf T}}$ the sets of fixed points under ${\bf T}$ in $S$ and $S_{\n}$ 
respectively. Set 
$$ C_{*} := {\Bbb P}^{1}(\k)\times C\setminus 
\{(0,\infty \}\times \{\poi x1{,\ldots,}{m}{}{}{}\} .$$ 
By Lemma~\ref{l2sa5}(iv), $\eta $ is a dominant morphism from $C_{*}$ to $S$, whence 
a commutative diagram
$$ \xymatrix{ C_{*} \ar[rr]^{\eta _{\n}} \ar[rrd]_{\eta } && S_{\n} \ar[d]^{\sigma }\\
&& S}$$ 
since $C_{*}$ is smooth. Let $\Delta _{\n}$ be the closure in 
${\Bbb P}^{1}(\k)\times C\times S_{\n}$ of the graph of $\eta _{\n}$ and $\upsilon _{2}$
the restriction to $\Delta _{\n}$ of the canonical projection
$$\xymatrix{ {\Bbb P}^{1}(\k)\times C\times S_{\n} \ar[rr] && S_{\n}}.$$

\begin{lemma}\label{l4sa5}
Suppose $V'\neq V$ and $V$ and $V'$ contained in ${\goth z}+{\goth a}$. 

{\rm (i)} The variety $\Delta $ is the image of $\Delta _{\n}$ by the map
$(z,x,y)\mapsto (z,x,\sigma (y))$.

{\rm (ii)} The morphism $\upsilon _{2}$ is projective and birational.

{\rm (iii)} There exists a ${\bf T}$-equivariant morphism 
$$ \xymatrix{ (S_{\n}\setminus S_{\n}^{{\bf T}}) \ar[rr]^{\varphi } && C_{*}}$$
such that $\eta \rond \varphi $ is the restriction of $\sigma $ to 
$S_{\n}\setminus S_{\n}^{{\bf T}}$.

{\rm (iv)} For some $i$ in $I$, $\eta _{i}(1)$ is not invariant under ${\bf T}$.
\end{lemma}

\begin{proof}
(i) As $S$ is a projective variety, so are $S_{\n}$, 
${\Bbb P}^{1}(\k)\times C\times S_{\n}$, $\Delta _{\n}$ and the image of 
$\Delta _{\n}$ by the map $(z,x,y)\mapsto (z,x,\sigma (y))$, whence the assertion since
the image of the graph of $\eta _{\n}$ by this map is the graph of $\eta $.

(ii) As $\Delta _{\n}$ is projective so is $\upsilon _{2}$. Since $\thetaup _{\alpha }$
is injective, so is the restriction of $\eta $ to 
$\k \times (C\setminus \{\poi x1{,\ldots,}{m}{}{}{}\})$. Hence $\upsilon _{2}$ is 
birational.

(iii) By (ii) and Zariski's Main Theorem~\cite[Ch. III, Corollary 11.4]{Ha}, the fibers 
of $\upsilon _{2}$ are connected. For $y$ in $S_{\n}\setminus S_{\n}^{{\bf T}}$ and 
$(z,x)$ in ${\Bbb P}^{1}(\k)\times C$ such that $(z,x,y)$ is in $\Delta _{\n}$, 
$\varpi \rond \sigma (y)=\nu (x)$ by (i). If $x$ is not in 
$\{\poi x1{,\ldots,}{m}{}{}{}\}$, $\nu ^{-1}(\varpi \rond \sigma (y))=\{x\}$ by Condition
(b) of Lemma~\ref{l2sa5}(iii) and $z$ is the element of $\k$ such that 
$\thetaup _{\alpha }(z,\mu (x))=\sigma (y)$. Suppose $x=x_{i}$ for some $i=1,\ldots,m$. 
Let $z$ and $z'$ be in $\k^{*}$ such that $(z,x_{i},y)$ and $(z',x_{i},y)$ are in 
$\Delta _{\n}$. Then $(z,x_{i},\sigma (y))$ and $(z',x_{i},\sigma (y))$ are in $\Delta $.
By Lemma~\ref{l3sa5},(iii) and (iv), $\sigma (y)=\eta (z,x_{i})=\eta (z',x_{i})$. For 
some $t$ in ${\bf T}$, $z'=t.z$ so that $t.\sigma (y)=\sigma (y)$. As $y$ is not 
invariant under ${\bf T}$ so is $\sigma (y)$ since the fibers of $\sigma $ are finite. 
Hence the stabilizer of $\sigma (y)$ in ${\bf T}$ is finite and so is the fiber of 
$\upsilon _{2}$ at $y$. As a result, the restriction of $\upsilon _{2}$ to 
$\Delta _{\n}\setminus {\Bbb P}^{1}(\k)\times C\times S_{\n}^{{\bf T}}$ is an injective 
morphism. So, again by Zariski's Main Theorem~\cite[\S 9]{Mu}, this map
is an isomorphism, whence a morphism 
$$ \xymatrix{ (S_{\n}\setminus S_{\n}^{{\bf T}}) \ar[rr]^{\varphi } && C_{*}}.$$
Moreover, $\varphi $ is ${\bf T}$-equivariant since so is $\upsilon _{2}$. For 
$y$ in $S_{\n}$ such that $\sigma (y)=\eta (z,x)$ for some $(z,x)$ in 
$\k^{*}\times (C\setminus \{\poi x1{,\ldots,}{m}{}{}{}\})$, $(z,x,y)$ is the unique 
element of $\Delta _{\n}$ above $y$. Hence $\eta \rond \varphi =\sigma $.

(iv) Suppose that for all $i$ in $I$, $\eta _{i}(1)$ is invariant under ${\bf T}$. A 
contradiction is expected. As $V\neq V'$, $V_{1}=V\cap V'$ and $V+V'=W$ by 
Lemma~\ref{lsa5}(iv). Moreover, since $V$ and $V'$ are contained in 
${\goth z}+{\goth a}$, for some $\beta $ and $\gamma $ in ${\cal R}$, 
$$ V = V_{1}+{\goth a}^{\beta } \quad  \text{and} \quad V' = V_{1}+{\goth a}^{\gamma }.$$
Then $\Gamma _{V_{1},V',W}$ is invariant under ${\bf T}$. More precisely,  
$\Gamma _{V_{1},V',W}$ is a union of one orbit of dimension $1$ and the set 
$\{(V_{1},V',V',W),(V_{1},V',V,W)\}$ of fixed points. As a result,
$\Gamma _{V_{1},V',W}\cap S$ is equal to $\{(V_{1},V',V',W),(V_{1},V',V,W)\}$ or 
$\Gamma _{V_{1},V',W}$ since $S$ is invariant under ${\bf T}$. By Lemma~\ref{l3sa5},(ii) 
and (v), for $i$ in $I$, $\eta _{i}({\Bbb P}^{1}(\k))$ is equal to $(V_{1},V',V',W)$ or 
$(V_{1},V',V,W)$ since $\nu (x_{i})=(V_{1},V',W)$. 

Suppose $\Gamma _{V_{1},V',W}\cap S=\{(V_{1},V',V',W),(V_{1},V',V,W)\}$. By 
Lemma~\ref{l3sa5},(v) and (iii), for all $i$ in $I$, $\eta $ is regular at $(0,x_{i})$ 
and $(\infty ,x_{i})$ since $\nu (x_{i})=(V_{1},V',W)$, whence
$$ \lim _{z\rightarrow 0} \eta _{i}(0) = (V_{1},V',V',W) \quad  \text{and} \quad
\lim _{z\rightarrow \infty } \eta _{i}(\infty ) = (V_{1},V',V,W) ,$$   
whence a contradiction.

Suppose $\Gamma _{V_{1},V',W}\cap S = \Gamma _{V_{1},V',W}$. Let $y$ be in $S_{\n}$
such that 
$$\sigma (y) \in \Gamma _{V_{1},V',W}\setminus \{(V_{1},V',V',W),(V_{1},V',V,W)\}.$$
By (iii), for some $i$ in $I$ and some $z$ in $\k^{*}$, $\varphi (t.y)=(t.z,x_{i})$ and 
$t.\sigma (y)=t.\eta (z,x_{i})=t.\eta _{i}(z)$ for all $t$ in ${\bf T}$, whence a 
contradiction since $(V_{1},V',V',W)$ and $(V_{1},V',V,W)$ are in 
$\overline{{\bf T}.\sigma (y)}$.  
\end{proof}

\begin{coro}\label{c2sa5}
Let $(V_{1},V',V,W)$ be a fixed point under ${\bf T}$ of $\Gamma $ such that
$V\cap {\goth t}=V'\cap {\goth t}={\goth z}$. Then $V'=V$. 
\end{coro}

\begin{proof}
Suppose $V'\neq V$. A contradiction is expected. By Lemma~\ref{lsa5}(iv),
$V_{1}=V\cap V'$ and $W=V+V'$. As $V\cap {\goth t}=V'\cap {\goth t}={\goth z}$, $V$ and 
$V'$ are contained in ${\goth z}+{\goth a}$. So, for some $\beta $ in ${\cal R}$ and 
$\gamma $ in 
${\cal R}\setminus \{\alpha \}$, 
$$ V = V_{1} \oplus {\goth a}^{\beta } \quad  \text{and} \quad 
V' = V_{1}\oplus {\goth a}^{\gamma }.$$
since $(V_{1},V',V,W)$ is invariant under ${\bf T}$. By Lemma~\ref{l4sa5}(iv), for 
some $i$ in $I$, $\eta _{i}(1)$ is not fixed under ${\bf T}$. Then, by 
Lemma~\ref{lsa5}(ii), $\eta _{i}({\Bbb P}^{1}(\k))=\Gamma _{V_{1},V',W}$. Denoting by 
$\eta _{i}(z)_{3}$ the third component of $\eta _{i}(z)$, for all $z$ in 
${\Bbb P}^{1}(\k)$, $V_{1}$ is contained in $\eta _{i}(z)_{3}$ and $\eta _{i}(z)_{3}$ is 
contained in $W$. Hence for some $a$ in $\k^{*}$,   
$$ \eta _{i}(1)_{3} = V_{1} \oplus \k (x_{\beta }+ax_{\gamma }) \quad  \text{and} \quad 
\eta _{i}(\alpha (t))_{3} = V_{1} \oplus \k (\beta (t)x_{\beta }+\gamma (t)ax_{\gamma })$$
for all $t$ in ${\bf T}$. For some $t_{1}$ and $t_{2}$ in ${\bf T}$, for all $\delta $
in ${\cal R}$, $\delta (t_{1})$ and $\delta (t_{2})$ are positive rational numbers and
$$ \alpha (t_{1}) > 1, \qquad \alpha (t_{2}) > 1, \qquad \beta (t_{1}) < \gamma (t_{1}), 
\qquad \beta (t_{2}) > \gamma (t_{2}) .$$ 
Then 
$$ \lim _{k \rightarrow \infty } V_{1} \oplus 
\k (\beta (t_{1}^{k})x_{\beta }+\gamma (t_{1}^{k})ax_{\gamma }) = 
V_{1} \oplus {\goth a}^{\gamma }, \quad 
\lim _{k \rightarrow \infty } V_{1} \oplus 
\k (\beta (t_{2}^{k})x_{\beta }+\gamma (t_{2}^{k})ax_{\gamma }) = 
V_{1} \oplus {\goth a}^{\beta }, $$ $$
\lim _{k \rightarrow \infty } \eta _{i}(\alpha (t_{1}^{k}) = 
\lim _{k \rightarrow \infty } \eta _{i}(\alpha (t_{2}^{k}) = \eta _{i}(\infty ),$$
whence $V=V'$ and the contradiction.
\end{proof}

\subsection{Property $({\bf P})$ and Property $({\bf P}_{1})$} \label{sa6}
In this subsection we suppose that all objects of ${\cal C}'_{{\goth t}}$ of dimension 
smaller than $n$ has Property $({\bf P})$. For $V$ a fixed point of $X_{R}$ under 
${\bf T}$, denote by $\Lambda _{V}$ the orthogonal complement to ${\goth z}_{V}$ in 
${\cal R}$ and set:
$${\goth r}_{V} := {\goth r}_{\Lambda _{V}}, \qquad R_{V} := R_{\Lambda _{V}} .$$

\begin{lemma}\label{lsa6}
Let $V$ be a fixed point under ${\bf T}$ in $X_{R}$. 

{\rm (i)} The action of $R_{V}$ in $\overline{R_{V}.V}$ has fixed points. For 
$V_{\infty }$ such a point, 
$$ V_{\infty } = V\cap {\goth t} \oplus 
\bigoplus _{\beta \in {\cal R}_{V_{\infty }}} {\goth a}^{\beta }, \quad
\vert {\cal R}_{V} \vert = \vert {\cal R}_{V_{\infty }} \vert, \quad
\rk V \geq  \rk {V_{\infty }}.$$

{\rm (ii)} The set ${\cal R}_{V}$ has rank at least $\vert {\cal R}_{V} \vert-1$. 

{\rm (iii)} Suppose that ${\goth a}$ has Property $({\bf P}_{1})$. Then ${\cal R}_{V}$ has
rank $\vert {\cal R}_{V} \vert$. 

{\rm (iv)} If ${\goth a}$ has Property $({\bf P}_{1})$, for $s$ in ${\goth t}$ such that
$V$ is contained in ${\goth r}^{s}$, $V$ is in $\overline{R^{s}.{\goth t}}$.
\end{lemma}

\begin{proof}
(i) As $\overline{R_{V}.V}$ is a projective variety and $R_{V}$ is connected and 
solvable, $R_{V}$ has fixed points in $\overline{R_{V}.V}$. Denote by $V_{\infty }$
such a point. Since $V$ is fixed under ${\bf T}$,
$$ V = V\cap {\goth t} \oplus \bigoplus _{\beta \in {\cal R}_{V}} {\goth a}^{\beta } .$$
Moreover, $V\cap {\goth t}$ is contained in ${\goth z}_{V}$ since $V$ is commutative.
By Lemma~\ref{l2sa2}(ii), ${\goth z}_{V}$ is the center of ${\goth r}_{V}$. Hence
$V\cap {\goth t}$ is contained in all element of $R_{V}.V$. Moreover, all element of 
$R_{V}.V$ is contained in $V\cap {\goth t}+{\goth a}_{\Lambda _{V}}$. Then 
$$ V_{\infty } = V\cap {\goth t} \oplus 
\bigoplus _{\beta \in {\cal R}_{V_{\infty }}} {\goth a}^{\beta } ,$$ 
whence $\vert {\cal R}_{V} \vert = \vert {\cal R}_{V_{\infty }} \vert$. Since 
${\cal R}_{V_{\infty }}$ is contained in $\Lambda _{V}$ and 
$\rk V = d-\dim {\goth z}_{V}$, $\rk V \geq \rk {V_{\infty }}$.

(ii) By (i), we can suppose that $V$ is invariant under $R_{V}$. By Lemma~\ref{lsa2}, 
${\goth a}_{\Lambda _{V}}$ is contained in an ideal ${\goth a}'$ of 
${\goth r}$ of dimension $\dim {\goth a}-1$ and contained in ${\goth a}$. We then use
the notations of Lemma~\ref{lsa5}. Set $\Gamma _{V} := \varpi _{3}^{-1}(V)$. By 
Lemma~\ref{lsa5}(i), $\Gamma _{V}$ is a projective variety invariant under $R_{V}$ since
so is $V$. Then $R_{V}$ has a fixed point in $\Gamma _{V}$. Let $(V_{1},V',V,W)$ be such 
a point. As ${\goth a}'$ has Property $({\bf P})$, by Lemma~\ref{lsa4}(i),
$$ V' = {\goth z}_{V'} \oplus \bigoplus _{\beta \in {\cal R}_{V'}} {\goth a}^{\beta } .$$
and the elements of ${\cal R}_{V'}$ are linearly independent. 

If $V'=V$ then ${\cal R}_{V'}={\cal R}_{V}$ so that 
$\rk V = \rk {V'}=\vert {\cal R}_{V} \vert$. Suppose $V'\neq V$. Then, by 
Lemma~\ref{lsa5}(iv),
$$ V_{1} = {\goth z}_{V'}\cap V\cap {\goth t} \oplus 
\bigoplus _{\beta \in {\cal R}_{V}\cap {\cal R}_{V'}} {\goth a}^{\beta } .$$
As $V_{1}$ has codimension $1$ in $V$ and $V'$, ${\cal R}_{V'}={\cal R}_{V}$ or 
${\goth z}_{V'}=V\cap {\goth t}$. In the first case, $\rk V=\vert {\cal R}_{V} \vert$
and in the second case, 
$$\vert {\cal R}_{V} \cap {\cal R}_{V'} \vert = \vert {\cal R}_{V} \vert -1
= \vert {\cal R}_{V'} \vert -1,$$
whence $\rk V \geq \vert {\cal R}_{V} \vert-1$ since the elements of ${\cal R}_{V'}$
are linearly independent.

(iii) Prove the assertion by induction on $\dim {\goth z}_{V}$. If 
${\goth z}_{V}={\goth z}$, then $r_{V}=\vert {\cal R}_{V} \vert$ by Property 
$({\bf P}_{1})$. Suppose $\dim {\goth z}_{V}=\dim {\goth z}+1$ and 
$V\cap {\goth t}={\goth z}$. Then $\vert {\cal R}_{V} \vert=d^{\#}$ and 
$r_{V}=d^{\#}-1$. By Property $({\bf P}_{1})$, it is impossible. Hence 
$V\cap {\goth t}={\goth z}_{V}$ since $V\cap {\goth t}$ is contained in ${\goth z}_{V}$.
As a result $r_{V}=\vert {\cal R}_{V} \vert$. 
 
Suppose $\dim {\goth z}_{V}\geq 2+\dim {\goth z}$, the assertion true for the integers 
smaller than $\dim {\goth z}_{V}$ and $r_{V}<\vert {\cal R}_{V} \vert$. A contradiction 
is expected. By (ii), $V\cap {\goth t}$ has dimension at least $\dim {\goth z}_{V}-1$. 
Then, for some $\alpha $ in ${\cal R}$, $V\cap {\goth t}_{\alpha }$ is strictly contained
in $V\cap {\goth t}$. Let $\Lambda $ be the orthogonal complement to 
${\goth z}_{V}\cap {\goth t}_{\alpha }$ in ${\cal R}$. As $\overline{R_{\Lambda }.V}$ is 
a projective variety and $R_{\Lambda }$ is connected, $R_{\Lambda }$ has a fixed point in
$\overline{R_{\Lambda }.V}$. Let $V_{\infty }$ be such a point. By 
Lemma~\ref{l2sa2}(ii), ${\goth z}_{V}\cap {\goth t}_{\alpha }$ is the center of 
${\goth r}_{\Lambda }$. Hence $V\cap {\goth t}_{\alpha }$ is contained in all element of 
$R_{\Lambda }.V$. Moreover, all element of $R_{\Lambda }.V$ is contained in 
$V\cap {\goth t}+{\goth a}_{\Lambda }$. As $V_{\infty }$ is an ideal of 
${\goth r}_{\Lambda }$, $V\cap {\goth t}$ is not contained in $V_{\infty }$ since it is 
not contained in the kernel of $\alpha $. Then 
$$ V_{\infty } = V\cap {\goth t}_{\alpha } \oplus 
\bigoplus _{\beta \in {\cal R}_{V_{\infty }}} {\goth a}^{\beta } .$$
By (ii), $r_{V_{\infty }}\geq \vert {\cal R}_{V_{\infty }} \vert -1$, whence
$$ \dim {\goth z}_{V_{\infty }} \leq \dim V\cap {\goth t}_{\alpha } + 1 =  
\dim V\cap {\goth t} < \dim {\goth z}_{V} .$$
So, by induction hypothesis, $\vert {\cal R}_{V_{\infty }} \vert = \rk {V_{\infty }}$ and
${\goth z}_{V_{\infty }}=V\cap {\goth t}_{\alpha }$. Since 
${\goth z}_{V}\cap {\goth t}_{\alpha }$ is the center of ${\goth r}_{\Lambda }$, 
${\goth z}_{V}\cap {\goth t}_{\alpha }$ is contained in ${\goth z}_{V_{\infty }}$,
whence
$$ \dim {\goth z}_{V}-1 \leq \dim V\cap {\goth t}_{\alpha } = \dim V\cap {\goth t} - 1 .$$
As a result, ${\goth z}_{V}=V\cap {\goth t}$ since $V\cap {\goth t}$ is contained in 
${\goth z}_{V}$, whence a contradiction.

(iv) Suppose that ${\goth a}$ has Property $({\bf P}_{1})$. By (iii), 
$$ V = {\goth z}_{V} \oplus \bigoplus _{\beta \in {\cal R}_{V}} {\goth a}^{\beta } $$
and $\rk V = \vert {\cal R}_{V} \vert$. As a result, the centralizer of $V$ in 
${\goth t}$ is equal to ${\goth z}_{V}$. Set
$$ {\goth a}'_{V} = \bigoplus _{\beta \in {\cal R}_{V}} {\goth a}^{\beta }, \qquad
{\goth r}'_{V} := {\goth t} + {\goth a}'_{V} .$$
Denote by $R'_{V}$ the connected closed subgroup of $R$ whose Lie algebra is 
$\ad {\goth r}'_{V}$. The algebra ${\goth a}'_{V}$ is in 
${\cal C}'_{{\goth t}}$ and has dimension $d-\dim {\goth z}_{V}$. Then, by 
Lemma~\ref{l2sa1}(ii), $V$ is in $\overline{R'_{V}.{\goth t}}$, whence the assertion 
since ${\goth r}'_{V}$ is contained in ${\goth r}^{s}$.
\end{proof}

\begin{coro}\label{csa6}
Suppose that ${\goth a}$ has Property $({\bf P}_{1})$. Then ${\goth a}$ has Property
$({\bf P})$.
\end{coro}

\begin{proof}
Let $V$ be in $X_{R}$ and $s$ in ${\goth t}\setminus {\goth z}$ such that $V$ is 
contained in ${\goth r}^{s}$. As $\overline{{\bf T}.V}$ is a projective variety and 
${\bf T}$ is a connected commutative group, ${\bf T}$ has a fixed point in 
$\overline{{\bf T}.V}$. Let $V_{\infty }$ be such a point. Since all element of 
${\bf T}.V$ is contained in ${\goth r}^{s}$, so is $V_{\infty }$. Then, by 
Lemma~\ref{lsa6}(iv), $V_{\infty }$ is in $\overline{R^{s}.{\goth t}}$. In particular, 
$s$ is in $V_{\infty }$. Let $E$ a complement to $V_{\infty }$ in ${\goth r}$, invariant 
under ${\bf T}$. The map 
$$\xymatrix{ {\mathrm {Hom}}_{\k}(V_{\infty },E) \ar[rr]^{\kappa } && \ec {Gr}r{}{}d}, 
\qquad \varphi \longmapsto 
\kappa (\varphi ) := {\mathrm {span}}(\{v+\varphi (v) \; \vert \; v \in V_{\infty }\}) $$
is an isomorphism onto an open neighborhood $\Omega _{E}$ of $V_{\infty }$ in 
$\ec {Gr}r{}{}d$. For $\varphi $ in ${\mathrm {Hom}}_{\k}(V_{\infty },E)$ such that 
$\kappa (\varphi )$ is in ${\bf T}.V$, $\varphi (s)$ is in ${\goth a}^{s}$. Then, for 
some $g$ in ${\bf T}$ and for some $v$ in ${\goth a}^{s}$, $s+v$ is in $g(V)$ and the 
semisimple component of $\ad (s+v)$ is different from $0$ since $s$ is not in 
${\goth z}$. Let $x$ be in ${\goth r}^{s}$ such that $\ad x$ is the semisimple component
of $\ad (s+v)$. By Lemma~\ref{lsa1}(ii), for some $k$ in $R^{s}$, $k(x)$ is in 
${\goth t}$. Then, by Corollary~\ref{csa5}(ii), $kg(V)$ is in 
$\overline{R^{k(x)}.{\goth t}}$. As $k(x)$ is not in ${\goth z}$, ${\goth a}^{k(x)}$
is an object of ${\cal C}'_{{\goth t}}$ of dimension smaller than $n$. By hypothesis, 
${\goth a}^{k(x)}$ has Property $({\bf P})$. Moreover, $kg(V)$ is contained in 
${\goth r}^{s}\cap {\goth r}^{k(x)}$. Hence, by Property $({\bf P})$ for 
${\goth a}^{k(x)}$, $kg(V)$ is in $\overline{R^{s}.{\goth t}}$, whence $V$ is in 
$\overline{R^{s}.{\goth t}}$ since $kg$ is in $R^{s}$. 
\end{proof}

\begin{prop}\label{psa6}
The objects of ${\cal C}'_{{\goth t}}$ have Property $({\bf P})$.
\end{prop}

\begin{proof}
Prove by induction on $n$ that ${\goth a}$ has Property $({\bf P})$. By 
Lemma~\ref{lsa3}, it is true for $n=d^{\#}$. Suppose that it is true for the integers
smaller than $n$. By Corollary~\ref{csa6}, it remains to prove that ${\goth a}$ has 
Property $({\bf P}_{1})$.

Suppose that ${\goth a}$ has not Property $({\bf P}_{1})$. A contradiction is expected. 
For some fixed point $V$ under ${\bf T}$ in $X_{R}$ such that 
$V\cap {\goth t}={\goth z}$, $r_{V}\neq \vert {\cal R}_{V} \vert$. By 
Lemma~\ref{lsa6}(ii), $r_{V}=\vert {\cal R}_{V} \vert -1$. Then the orthogonal complement
of ${\cal R}_{V}$ in ${\goth t}$ is generated by ${\goth z}$ and an element $s$ in 
${\goth t}\setminus {\goth z}$. In particular, $V$ is contained in ${\goth r}^{s}$.
According to Lemma~\ref{lsa2}, for some ideal ${\goth a}'$ of codimension $1$ of 
${\goth a}$, normalized by ${\goth t}$, ${\goth a}^{s}$ is contained in ${\goth a}'$. 
Denote by $\alpha $ the element of ${\cal R}$ such that 
$$ {\goth a} = {\goth a}'\oplus {\goth a}^{\alpha } $$
and consider $\thetaup _{\alpha }$ and $\Gamma $ as in Subsection~\ref{sa5}. Denote 
by $\Gamma _{V}$ the set of elements of $\Gamma $ whose image by the projection
$$ \xymatrix{ \Gamma \ar[rr] && \ec {Gr}r{}{}d}, \qquad (T_{1},T',T,T_{2}) \longmapsto T$$
is equal to $V$. By Lemma~\ref{lsa5}(ii), $\Gamma _{V}$ is not empty and it is invariant 
under ${\bf T}$ by Lemma~\ref{lsa5}(i). As it is a projective variety, it has a fixed 
point under ${\bf T}$. Denote by $(V_{1},V',V,W)$ such a point. As ${\goth a}'$ has
Property $({\bf P})$, it has Property $({\bf P}_{1})$ by Lemma~\ref{lsa4}. Hence 
$r_{V'}=\vert {\cal R}_{V'} \vert$ and $V'\neq V$ since $r_{V}\neq {\cal R}_{V}$. Then, 
by Lemma~\ref{lsa5}(iv), 
$$ V_{1} = V\cap V' \quad  \text{and} \quad W = V'+V .$$ 
As a result, $V'\cap {\goth t}=V\cap {\goth t}={\goth z}$ since 
${\cal R}_{V'}\neq {\cal R}_{V}$ and $V_{1}$ has codimension $1$ in $V$ and $V'$. 
Then $V'=V$ by Corollary~\ref{c2sa5}, whence a contradiction.
\end{proof}

The following corollary results from Proposition~\ref{psa6}, Corollary~\ref{csa3} and
Lemma~\ref{lsa4}.

\begin{coro}\label{c2sa6}
Let $V$ be in $X_{R}$. 

{\rm (i)} The space $V$ is a commutative algebraic subalgebra of ${\goth r}$ and for some
subset $\Lambda $ of ${\cal R}$, the biggest torus contained in $V$ is conjugate to 
${\goth t}_{\Lambda }$ under $R$.

{\rm (ii)} If $V$ is a fixed point under $R$, then $V$ is an ideal of ${\goth r}$ and the
elements of ${\cal R}_{V}$ are linearly independent.
\end{coro}

\section{Solvable algebras and main varieties} \label{sav}
Let ${\goth t}$ be a vector space of positive dimension $d$ and ${\goth a}$ in 
${\cal C}_{{\goth t}}$. Set:
$$ {\cal R} := {\cal R}_{{\goth t},{\goth a}}, \qquad 
{\goth r} := {\goth r}_{{\goth t},{\goth a}} \qquad \pi := \pi _{{\goth t},{\goth a}},
\qquad R := R_{{\goth t},{\goth a}}, \qquad A := A_{{\goth t},{\goth a}}, \qquad 
{\cal E} := {\cal E}_{{\goth t},{\goth a}}, \qquad n := \dim {\goth a} .$$
In this section, we give some results on the singular locus of $X_{R}$.
For ${\goth a}'$ in ${\cal C}_{{\goth t}}$, denote by $X_{R_{{\goth t},{\goth a}'},\n}$ 
the subset of elements of $X_{R_{{\goth t},{\goth a}'}}$ contained in ${\goth a}'$.

\subsection{Subvarieties of  $X_{R}$} \label{sav1}
Denote by ${\cal P}_{c}({\cal R})$ the set of complete subsets of ${\cal R}$ and 
for $\Lambda $ in ${\cal P}_{c}({\cal R})$ denote by $X_{R_{\Lambda }}$ the closure in 
$\ec {Gr}r{}{}d$ of the orbit $R_{\Lambda }.{\goth t}$.

\begin{prop}\label{psav1}
Let $Z$ be an irreducible closed subset of $X_{R}$, invariant under $R$.

{\rm (i)} For a well defined complete subset $\Lambda $ of ${\cal R}$, all element of
a dense open subset of $Z$ is conjugate under $R$ to the sum of ${\goth t}_{\Lambda }$ 
and a subspace of ${\goth a}$.

{\rm (ii)} All element of $Z$ is contained in ${\goth t}_{\Lambda }\oplus {\goth a}$.

{\rm (iii)} For some irreducible closed subset $Z_{\Lambda }$ of $X_{R_{\Lambda }}$, 
invariant under $R_{\Lambda }$, $R.Z_{\Lambda }$ is dense in $Z$.
\end{prop}

\begin{proof}
(i) For $\Lambda $ in ${\cal P}_{c}({\cal R})$, let $Y_{\Lambda }$ be the subset 
of elements $V$ of $Z$ such that $\pi (V)={\goth t}_{\Lambda }$. Since $Z$ is invariant 
under $R$, so is $Y_{\Lambda }$. Moreover, by Corollary~\ref{c2sa6}(i),
$$\overline{Y_{\Lambda }} \subset Y_{\Lambda } \cup 
\bigcup _{\mycom {\Lambda '\in {\cal P}_{c}({\cal R})}
{\Lambda ' \supsetneq \Lambda }} Y_{\Lambda '} . $$ 
According to Corollary~\ref{c2sa6}(i), $Z$ is the union of 
$Y_{\Lambda },\Lambda \in {\cal P}_{c}({\cal R})$. As a result, since 
${\cal R}$ is finite and $Z$ is irreducible, for a well defined complete subset 
$\Lambda $ of ${\cal R}$, $Y_{\Lambda }$ contains a dense open subset of $Z$. By 
Lemma~\ref{lsa1}(v), all element of $Y_{\Lambda }$ is conjugate under 
$R$ to the sum of ${\goth t}_{\Lambda }$ and a subspace of ${\goth a}$. 

(ii) By (i), for all $V$ in a dense subset of $Z$, $V$ is contained in 
${\goth t}_{\Lambda }\oplus {\goth a}$, whence the assertion.

(iii) Let $Z_{*}$ be the subset of elements of $Z$, containing ${\goth t}_{\Lambda }$. 
Denote by $s$ an element of ${\goth t}_{\Lambda }$ such that $\alpha (s)\neq 0$ for all
$\alpha $ in ${\cal R}\setminus \Lambda $. By Lemma~\ref{l2sa2}(i), 
$${\goth r}^{s}={\goth t}\oplus {\goth a}_{\Lambda }.$$
Hence $Z_{*}$ is contained in $X_{R_{\Lambda }}$ by Proposition~\ref{psa6}. Moreover, 
$Z_{*}$ is invariant under $R_{\Lambda }$ since $Z$ is invariant under $R$. By (i), 
$R.Z_{*}$ is dense in $Z$. So, for some irreducible component $Z_{\Lambda }$ of 
$Z_{*}$, $R.Z_{\Lambda }$ is dense in $Z$. Moreover, $Z_{\Lambda }$ is 
invariant under $R_{\Lambda }$ since so is $Z_{*}$.
\end{proof}

For $\Lambda $ in ${\cal P}_{c}({\cal R})$, denote by ${\goth t}_{\Lambda }^{\#}$ a 
complement to ${\goth t}_{\Lambda }$ in ${\goth t}$ and set:
$$ {\goth r}^{\#}_{\Lambda } := {\goth t}^{\#}_{\Lambda } + {\goth a}_{\Lambda } .$$
Let $R^{\#}_{\Lambda }$ be the adjoint group of ${\goth r}^{\#}_{\Lambda }$
and $A^{\#}_{\Lambda }$ the connected closed subgroup of $R^{\#}_{\Lambda }$ whose Lie
algebra is $\ad {\goth a}_{\Lambda }$.

\begin{lemma}\label{lsav1}
Let $\Lambda $ be in ${\cal P}_{c}({\cal R})$, nonempty and strictly contained in 
${\cal R}$.

{\rm (i)} The tori ${\goth t}_{\Lambda }$ and ${\goth t}_{\Lambda }^{\#}$ have positive 
dimension and ${\goth a}_{\Lambda }$ is in ${\cal C}_{{\goth t}_{\Lambda }^{\#}}$.
Moreover,
$$\dim {\goth a}_{\Lambda } - \dim {\goth t}_{\Lambda }^{\#} \leq 
\dim {\goth a} - \dim {\goth t} .$$

{\rm (ii)} The map $V\mapsto V \oplus {\goth t}_{\Lambda }$ is an isomorphism from
$X_{R^{\#}_{\Lambda }}$ onto $X_{R_{\Lambda }}$.
\end{lemma}

\begin{proof}
Since $\Lambda $ is a complete subset of ${\cal R}$ strictly contained in ${\cal R}$, 
${\goth t}_{\Lambda }$ has positive dimension and since $\Lambda $ is not empty, 
${\goth t}_{\Lambda }$ is strictly contained in ${\goth t}$. By definition, $\Lambda $
is the set of weights of ${\goth t}$ in ${\goth a}_{\Lambda }$ so that 
${\goth a}_{\Lambda }$ is in ${\cal C}'_{{\goth t}}$. Then ${\goth a}_{\Lambda }$
is in ${\cal C}_{{\goth t}^{\#}_{\Lambda }}$ and Assertion (ii) results from 
Corollary~\ref{csa1}.

By Lemma~\ref{lsa1},(i) and (iv), ${\cal R}$ generates ${\goth t}^{*}$. Hence
$$ \vert \Lambda  \vert + \dim {\goth t}_{\Lambda } \leq \vert {\cal R} \vert.$$ 
By Condition (2) of Section~\ref{sa}, ${\goth a}$ has dimension $\vert {\cal R} \vert$ 
and ${\goth a}_{\Lambda }$ has dimension $\vert \Lambda  \vert$. As a result,
$$ \dim {\goth a} - \dim {\goth t} = \vert {\cal R} \vert - \dim {\goth t}_{\Lambda }
- \dim {\goth t}_{\Lambda }^{\#} \geq \dim {\goth a}_{\Lambda } - 
\dim {\goth t}_{\Lambda }^{\#} .$$
\end{proof}

\subsection{Smooth points of $X_{R}$ and commutators} \label{sav2}
Denote by ${\goth t}_{\r}$ the complement in ${\goth t}$ to the union of 
${\goth t}_{\alpha }, \alpha \in {\cal R}$ and ${\goth r}_{\r}$ the set of elements $x$ 
of ${\goth r}$ such that ${\goth r}^{x}$ has minimal dimension.

\begin{lemma}\label{lsav2}
{\rm (i)} The set ${\goth t}_{\r}$ is a dense open subset of ${\goth t}$, contained in 
${\goth r}_{\r}$. Moreover, $R.{\goth t}_{\r}$ is a dense open subset of ${\goth r}$.

{\rm (ii)} For all $x$ in ${\goth r}_{\r}$, ${\goth r}^{x}$ is in $X_{R}$.

{\rm (iii)} The set ${\goth r}_{\r}$ is a big open subset of ${\goth r}$.
\end{lemma}
 
\begin{proof}
(i) By definition, ${\goth t}_{\r}$ is a dense open subset of ${\goth t}$. According to 
Lemma~\ref{l2sa2}(i), for $x$ in ${\goth t}_{\r}$, ${\goth r}^{x}={\goth t}$. Then 
$R.x = A.x=x+{\goth a}$ since $A.x$ is a closed subset of $x+{\goth a}$ of 
dimension $\dim {\goth a}$. As a result, $R.{\goth t}_{\r}={\goth t}_{\r}+{\goth a}$ is a
dense open subset of ${\goth r}$. Hence $R.{\goth t}_{\r}$ is contained in 
${\goth r}_{\r}$ since ${\goth r}^{x}$ is conjugate to ${\goth t}$ for all $x$ in 
$R.{\goth t}_{\r}$ and ${\goth r}_{\r}$ is a dense open subset of ${\goth r}$. 

(ii) By (i), for all $x$ in ${\goth r}_{\r}$, ${\goth r}^{x}$ has dimension $d$, whence a
regular map
$$ \xymatrix{ {\goth r}_{\r} \ar[rr]^{\theta } && \ec {Gr}r{}{}d}, \qquad 
x \longmapsto {\goth r}^{x} .$$
As a result, by (i), for all $x$ in ${\goth r}_{\r}$, ${\goth r}^{x}$ is in $X_{R}$.

(iii) Suppose that ${\goth r}_{\r}$ is not a big open subset of ${\goth r}$. A 
contradiction is expected. Let $\Sigma $ be an irreducible component of codimension $1$ 
of ${\goth r}\setminus {\goth r}_{\r}$. Since $\Sigma \cap A.{\goth t}_{\r}$ is empty, 
$\pi (\Sigma )$ is contained in ${\goth t}_{\alpha }$ for some $\alpha $ in ${\goth r}$. 
Then $\Sigma ={\goth t}_{\alpha }+{\goth a}$ since $\Sigma $ has codimension $1$ in 
${\goth r}$. By Condition (3) of Section~\ref{sa}, for some $s$ in ${\goth t}_{\alpha }$, 
$\gamma (s)\neq 0$ for all $\gamma $ in ${\cal R}\setminus \{\alpha \}$. Then
${\goth r}^{s+x_{\alpha }}={\goth t}_{\alpha }+{\goth a}^{\alpha }$ so that 
$s+x_{\alpha }$ is in ${\goth r}_{\r}$ by (i) and Condition (2) of Section~\ref{sa}, 
whence the contradiction.
\end{proof}

Denote by $X'_{R}$ the image of $\theta $. 
 
\begin{prop}\label{psav2}
{\rm (i)} The complement to $R.{\goth t}$ in $X_{R}$ is equidimensional of dimension 
$\dim {\goth a}-1$.

{\rm (ii)} The set $X'_{R}$ is a smooth open subset of $X_{R}$, containing $R.{\goth t}$.
\end{prop}

\begin{proof}
(i) As $R$ is solvable and $R.{\goth t}$ is dense in $X_{R}$, $R.{\goth t}$ is an affine
open subset of $X_{R}$. So, by \cite[Corollaire 21.12.7]{Gro1}, 
$X_{R}\setminus R.{\goth t}$ is equidimensional of dimension $\dim {\goth a}-1$ since 
$X_{R}$ has dimension $\dim {\goth a}$. 

(ii) By definition, ${\cal E}$ is the subvariety of elements $(V,x)$ of 
$X_{R}\times {\goth r}$ such that $x$ is in $V$. Let $\Gamma $ be the image of the 
graph of $\theta $ by the isomorphism 
$$ \xymatrix{{\goth r}\times \ec {Gr}r{}{}d \ar[rr] && \ec {Gr}r{}{}d\times {\goth r}},
\qquad (x,V) \longmapsto (V,x) .$$
Then $\Gamma $ is the intersection of ${\cal E}$ and $X_{R}\times {\goth r}_{\r}$. Since 
$\Gamma $ is isomorphic to ${\goth r}_{\r}$, $\Gamma $ is a smooth open subset of 
${\cal E}$ whose image by the bundle projection is $X'_{R}$. As a result, $X'_{R}$ is a 
smooth open subset of $X_{R}$ by~\cite[Ch. 8, Theorem 23.7]{Mat}.   
\end{proof}

For $\alpha $ in ${\cal R}$, set 
$V_{\alpha } := {\goth t}_{\alpha }\oplus {\goth a}^{\alpha }$ and denote by 
$\theta _{\alpha }$ the map
$$ \xymatrix{\k \ar[rr]^{\theta _{\alpha }} && \ec {Gr}r{}{}d}, \qquad
z \longmapsto \exp( z\ad x_{\alpha })({\goth t}) ,$$
By Condition (2) of Section~\ref{sa}, $V_{\alpha }$ has dimension $d$.

\begin{lemma}\label{l2sav2}
Let $\alpha $ be in ${\cal R}$. Set $X_{R,\alpha } := \overline{A.V_{\alpha }}$.

{\rm (i)} The map $\theta _{\alpha }$ has a regular extension to ${\Bbb P}^{1}(\k)$ such 
that $\theta _{\alpha }(\infty )=V_{\alpha }$.

{\rm (ii)} The variety $X_{R,\alpha }$ has dimension 
$\dim {\goth a}-1$ and it is an irreducible component of $X_{R}\setminus R.{\goth t}$. 

{\rm (iii)} The intersection $X_{R,\alpha }\cap X'_{R}$ is not empty.
\end{lemma}
 
\begin{proof}
(i) Let $h_{\alpha }$ be in ${\goth t}$ such that $\alpha (h_{\alpha })=1$.
Since $X_{R}$ is a projective variety, the map $\theta _{\alpha }$ has a regular 
extension to ${\Bbb P}^{1}(\k)$ by~\cite[Ch. 6, Theorem 6.1]{Sh}. For $z$ in $\k$, 
$$ \theta _{\alpha }(z) = {\goth t}_{\alpha } \oplus \k (h_{\alpha }-zx_{\alpha }) .$$
Hence $\theta _{\alpha }(\infty )=V_{\alpha }$.

(ii) By (i), $X_{R,\alpha }$ is contained in $X_{R}$ and its elements are contained in 
${\goth t}_{\alpha }\oplus {\goth a}$ so that $X_{R,\alpha }$ is contained in 
$X_{R}\setminus R.{\goth t}$. By Condition (3) of Section~\ref{sa}, for $\gamma $ in 
${\cal R}\setminus \{\alpha \}$ and $v$ in ${\goth a}^{\gamma }$, 
$[{\goth t}_{\alpha },v]=\k v$ so that no element of ${\goth a}^{\gamma }$ normalizes 
$V_{\alpha }$. As a result, the normalizer of $V_{\alpha }$ in ${\goth r}$ is equal to 
${\goth t}+{\goth a}^{\alpha }$ so that $X_{R,\alpha }$ has dimension 
$\dim {\goth a} -1$. Hence $X_{R,\alpha }$ is an irreducible component 
$X_{R}\setminus R.{\goth t}$. 

(iii) According to Condition (3) of Section~\ref{sa}, for some $s$ in 
${\goth t}_{\alpha }$, $\gamma (s)\neq 0$ for all $\gamma $ in 
${\cal R}\setminus \{\alpha \}$. Then $V_{\alpha } = {\goth r}^{s+x_{\alpha }}$ so that 
$s+x_{\alpha }$ is in ${\goth r}_{\r}$, whence the assertion.
\end{proof}

\subsection{On the singular locus of $X_{R}$} \label{sav3}
In this subsection we suppose $\dim {\goth a} > d$ and we fix an ideal ${\goth a}'$ of 
codimension $1$ in ${\goth a}$, normalized by ${\goth t}$ and such that ${\goth a}'$ is 
in ${\cal C}_{{\goth t}}$. For example, all ideal of ${\goth r}$ of dimension 
$\dim {\goth a}-1$, contained in ${\goth a}$ and containing a fixed point under $R$ in
$X_{R}$ is in ${\cal C}_{{\goth t}}$ by Corollary~\ref{c2sa6}(ii).
Set:
$$ {\goth r}' := {\goth r}_{{\goth t},{\goth a}'} \qquad 
\pi ':= \pi _{{\goth t},{\goth a}'},
\qquad R' := R_{{\goth t},{\goth a}'}, \qquad A' := A_{{\goth t},{\goth a}'}, \qquad 
{\cal R}' := {\cal R}_{{\goth t},{\goth a}'}.$$
Let $\alpha $ be in ${\cal R}$ such that 
$$ {\goth a} = {\goth a}' \oplus {\goth a}^{\alpha } $$
and $\Gamma $ as in Subsection~\ref{sa5}. Denote by $\varpi _{1}$, $\varpi _{2}$, 
$\varpi _{3}$, $\varpi _{4}$ the restrictions to $\Gamma $ of the first, second, third, 
fourth projections. Let $Z$ be an irreducible component of $X_{R,\n}$. According to 
Lemma~\ref{lsa5}(ii), for some irreducible component $T$ of $\varpi _{3}^{-1}(Z)$, 
$\varpi _{3}(T)=Z$. Denote by $Z'$ the image of $T$ by $\varpi _{2}$ and by $T_{1}$ the 
image of $T$ by the projection $\varpi _{1}\mul \varpi _{4}$. Then $Z'$ and $T_{1}$ are 
irreducible closed subsets of $\ec {Gr}r{}{}d$ and 
$\ec {Gr}r{}{}{d-1}\times \ec {Gr}r{}{}{d+1}$ respectively. Let $T_{0}$ be the subset of 
elements $(V_{1},V',V,W)$ of $T$ such that $V'=V$. Then $T_{0}$ is a closed subset of 
$T$. If $T_{0}=T$, $Z'=Z$ and $Z$ is contained in $X_{R',\n}$.
Otherwise, $O:=T\setminus T_{0}$ is a dense open subset of $T$. According to 
Lemma~\ref{lsa5}(iv), for all $(V_{1},V',V,W)$ in $O$, $V_{1}=V'\cap V$ and $V'+V=W$.
Denote by  $O_{1}$ an open subset of $T_{1}$, contained and dense in 
$\varpi _{1}\mul \varpi _{4}(O)$.

Let $(V_{1},W)$ be in $O_{1}$. Denote by $E$ a complement to $V_{1}$ in ${\goth r}$ and 
by $E'$ a complement to $W$ in ${\goth r}$ contained in $E$. Let $\kappa $ be the map
\begin{eqnarray*}
\xymatrix{ {\mathrm {Hom}}_{\k}(V_{1},W\cap E)\times {\mathrm {Hom}}_{\k}(W,E') 
\ar[rr]^{\kappa } && \ec {Gr}r{}{}{d-1}\times \ec {Gr}r{}{}{d+1} }, \\
(\varphi ,\psi ) \longmapsto 
({\mathrm {span}}(\{v+\varphi (v)+\psi (v)+\psi \rond \varphi (v) 
\; \vert \; v \in V_{1}\}), {\mathrm {span}}(\{v+\psi (v) \; \vert \; v \in W \})) .
\end{eqnarray*}
Then $\kappa $ is an isomorphism from its source to an open neighborhood of $(V_{1},W)$
in the subvariety of elements $(W_{1},W_{2})$ of 
$\ec {Gr}r{}{}{d-1}\times \ec {Gr}r{}{}{d+1}$ such that $W_{1}$ is contained in $W_{2}$.
Denote by $\Omega $ the inverse image by $\kappa $ of the intersection of $T_{1}$ and the 
image of $\kappa $. Let $(e_{1},e_{2})$ be a basis of $W\cap E$ and let $\kappa _{*}$ be 
the map
$$ \xymatrix{ \Omega \times (\k^{2}\setminus \{(0,0)\}) \ar[rr]^{\kappa _{*}} && 
\ec {Gr}r{}{}d}, $$
$$ (\varphi ,\psi ,x_{1},x_{2}) \longmapsto 
{\mathrm {span}}(\{v+\varphi (v)+\psi (v)+\psi \rond \varphi (v) \; \vert \; v \in V_{1}\}
\cup \{x_{1}(e_{1}+\psi (e_{1}))+x_{2}(e_{2}+\psi (e_{2}))\}) .$$

\begin{lemma}\label{lsav3}
Suppose that $O$ is not empty. Denote by $\widetilde{\Omega }$ the image of $\kappa _{*}$
and $\tilde{Z}$ the closure of $\widetilde{\Omega }$ in $\ec {Gr}r{}{}d$.

{\rm (i)} The intersections $\widetilde{\Omega }\cap Z'$ and 
$\widetilde{\Omega }\cap Z$ are dense in $Z'$ and $Z$ respectively. In particular $Z'$ 
and $Z$ are contained in $\tilde{Z}$.

{\rm (ii)} For $V$ in $\widetilde{\Omega }$, there exists $(V',V'')$ in 
$Z'\times Z$ such that 
$$ V'\cap V'' \subset V, \qquad V \subset V'+V'', \qquad 
(V'\cap V'',V'+V'') \in \kappa (\Omega ) .$$

{\rm (iii)} Let $F'$ be the fiber of $\kappa _{*}$ at some element $V$ of 
$\kappa _{*}(\Omega )$. Denote by $F$ the subset of elements $(\varphi ,\psi )$ of 
$\Omega $ such that $V$ contains the first component of $\kappa (\varphi ,\psi )$ and is 
contained in the second component of $\kappa (\varphi ,\psi )$. Then 
$F' = F \times \k^{*}(x_{1},x_{2})$ for some $(x_{1},x_{2})$ in 
$\k^{2}\setminus \{(0,0)\}$.

{\rm (iv)} The varieties $\tilde{Z}$ and $Z$ have dimension at most $\dim Z' + 1$.
\end{lemma}

\begin{proof}
(i) Since $T$ is irreducible so are $T_{1}$ and $\Omega $. Hence $\tilde{Z}$ is 
irreducible. For some $(V',V)$ in $Z'\times Z$, $V_{1}$ is contained in $V'$ and $V$ and 
$V'$ and $V$ are contained in $W$. Since $\kappa (\Omega )$ is an open neighbourhood of 
$(V_{1},W)$ in $T_{1}$, 
$$\varpi _{2}(\varpi _{1}\mul \varpi _{4}^{-1}(\kappa (\Omega ))\cap T) 
\quad  \text{and} \quad
\varpi _{3}(\varpi _{1}\mul \varpi _{4}^{-1}(\kappa (\Omega ))\cap T) $$
are dense subsets of $Z'$ and $Z$ respectively. For all $(\varphi ,\psi )$ in $\Omega $,
all element of 
$\varpi _{2}(\varpi _{1}\mul \varpi _{4}^{-1}(\kappa (\varphi ,\psi ))\cap T)$ 
contains the first component of $\kappa (\varphi ,\psi )$ and is contained in 
the second component of $\kappa (\varphi ,\psi )$. Hence all element of 
$\varpi _{2}(\varpi _{1}\mul \varpi _{4}^{-1}(\kappa (\Omega ))\cap T)$ is in the image 
of $\kappa _{*}$. As a result, $\widetilde{\Omega }\cap Z'$ is dense in $Z'$ and $Z'$ is 
contained in $\tilde{Z}$. In the same way, $\widetilde{\Omega }\cap Z$ is dense in $Z$ 
and $Z$ is contained in $\tilde{Z}$.

(ii) According to Lemma~\ref{lsa5}(iv), for all $(V'_{1},V',V,W')$ in $O$, 
$V'_{1}=V'\cap V$ and $W'=V'+V$. By definition, $\kappa (\Omega )$ is contained in 
$\varpi _{1}\mul \varpi _{4}(O)$ and for $V$ in $\widetilde{\Omega }$, $V'_{1}\subset V$ 
and $V\subset W'$ for some $(V'_{1},W')$ in $\kappa (\Omega )$, whence the assertion.

(iii) For $(\varphi ,\psi )$ in $F$ and for $(x_{1},x_{2})$ in 
$\k^{2}\setminus \{(0,0)\}$ such that
$$ V = \kappa _{*}(\varphi ,\psi ,x_{1},x_{2}),$$
the subset of elements $(y_{1},y_{2})$ of $\k^{2}$ such that 
$(\varphi ,\psi ,y_{1},y_{2})$ is in $F'$ is equal to $\k^{*}.(x_{1},x_{2})$. Moreover,
for all $(\varphi ,\psi ,y_{1},y_{2})$ in $F'$, $(\varphi ,\psi )$ is in $F$, whence the 
assertion.

(iv) In (iii), we can choose $V$ such that $F'$ has minimal dimension so that
$$ \dim \tilde{Z} = \dim \Omega +2 - (\dim F+1) = \dim \Omega - \dim F +1 .$$ 
By (ii), for some $V'$ in $Z'$, for all $(\varphi ,\psi )$ in $F$, $V'$ contains the
first component of $\kappa (\varphi ,\psi )$ and is contained in the second component of 
$\kappa (\varphi ,\psi )$. So, again by (iii) and (ii),
$$ \dim Z' \geq \dim \Omega - \dim F,$$
whence $\dim \tilde{Z} \leq \dim Z' +1$ and $\dim Z \leq \dim Z' +1$ since $Z$ is 
contained in $\tilde{Z}$ by (i).
\end{proof}

\begin{prop}\label{psav3}
The variety $X_{R,\n}$ has dimension at most $n-d$.
\end{prop}

\begin{proof}
Prove this by induction on $n$. According to Lemma~\ref{l2sa1}(ii), it is true
for $n-d=0$. Suppose that $n-d$ is positive and that it is true for all integer smaller
than $n-d$. In particular, $X_{R',\n}$ has dimension at most $n-d-1$. Let $Z$ be an 
irreducible component of $X_{R,\n}$. According to Lemma~\ref{lsa5}(ii), for some 
irreducible component $T$ of $\varpi _{3}^{-1}(Z)$, $\varpi _{3}(T)=Z$. Denote by 
$Z'$ the image of $T$ by $\varpi _{2}$. Let $T_{0}$ be the subset of elements 
$(V_{1},V',V,W)$ of $T$ such that $V'=V$. Consider the following cases:
\begin{itemize}
\item [{\rm (a)}] $T_{0}=T$,  
\item [{\rm (b)}] $T_{0}\neq T$ and $Z'$ is contained in $X_{R',\n}$,
\item [{\rm (c)}] $Z'$ is not contained in $X_{R',\n}$.
\end{itemize}

(a) In this case, $Z'=Z$ and $\dim Z \leq n-d-1$ by induction hypothesis.

(b) By induction hypothesis, $\dim Z' \leq n-d-1$ and by Lemma~\ref{lsav3}(iv), 
$\dim Z \leq \dim Z'+1$, whence $\dim Z \leq n-d$.

(c) In this case, $T_{0}\neq T$, whence $\dim Z \leq \dim Z'+1$ by Lemma~\ref{lsav3}(iv).
Since $Z$ is an irreducible component of $X_{R,\n}$, $Z$ is invariant under $R$. By 
Lemma~\ref{lsa5}(i), $\varpi _{2}$ and $\varpi _{3}$ are equivariant under the action 
of $R'$ in $\Gamma $ so that $T$ and $Z'$ are invariant under $R'$. For all 
$(V_{1},V',V,W)$ in $T\setminus T_{0}$, $V_{1}=V'\cap V$. Hence all element of 
a dense open subset of $Z'$ contains a subspace of dimension $d-1$ of ${\goth a}'$. 
Then, by Proposition~\ref{psav1}, for some complete subset $\Lambda $ of ${\cal R}'$ such 
that ${\goth t}_{\Lambda }$ has dimension $1$ and for some closed subset $Z_{\Lambda }$ 
of $X_{R_{\Lambda }}$, $R'.Z_{\Lambda }$ is dense in $Z'$ so that 
$$ \dim Z' \leq \dim Z_{\Lambda } + \dim {\goth a}'-\dim {\goth a}_{\Lambda }.$$
If $\dim {\goth a}_{\Lambda }-\dim {\goth t}+1 = n-d$, then $\Lambda ={\cal R}'$.
In this case, since ${\goth a}'$ is in ${\cal C}_{{\goth t}}$, $\Lambda $ 
generates ${\goth t}^{*}$. As ${\goth t}_{\Lambda }$ has dimension $1$, it is impossible.
As a result, 
$$ \dim Z_{\Lambda } \leq \dim {\goth a}_{\Lambda }-\dim {\goth t}+1 
\quad  \text{and} \quad \dim Z' \leq n-d $$
by Lemma~\ref{lsav1} and induction hypothesis for ${\goth a}_{\Lambda }$. Then 
$\dim Z \leq n-d+1$. According to Lemma~\ref{lsav3},(i) and (iv), $\tilde{Z}$ is an 
irreducible variety of dimension at most $\dim Z'+1$, containing $Z'$ and $Z$. If 
$\dim Z'=n-d$ and $\dim Z = n-d+1$, then $Z=\tilde{Z}$. In particular, $Z'$ is contained 
in $Z$. It is impossible since all element of $Z$ is contained in ${\goth a}$. As a 
result, $\dim Z \leq n-d$, whence the proposition.
\end{proof}

\begin{coro}\label{csav3}
{\rm (i)} The irreducible components of $X_{R}\setminus R.{\goth t}$ are the 
$X_{R,\alpha },\alpha \in {\cal R}$.  

{\rm (ii)} The set $X'_{R}$ is a smooth big open subset of $X_{R}$, containing 
$R.{\goth t}$.
\end{coro}

\begin{proof}
According to Proposition~\ref{psav2}(ii) and Lemma~\ref{l2sav2}(iii), Assertion (ii) 
results from Assertion (i). Prove Assertion (i) by induction on $n=\dim {\goth a}$. 
For $n=1$, $d=1$ by Lemma~\ref{lsa1},(i) and (iv) so that $X_{R}$ is the union of 
$R.{\goth t}$ and ${\goth a}^{\alpha }$, whence Assertion (i) in this case. Suppose 
$n\geq 2$ and the assertion true for the integers smaller than $n$. By 
Lemma~\ref{lsa1}(i), Condition (2) and Condition (3) of Section~\ref{sa}, $d\geq 2$. 
According to Lemma~\ref{l2sav2}(ii), for all $\alpha $ in ${\cal R}$, $X_{R,\alpha }$ is 
an irreducible component of $X_{R}\setminus R.{\goth t}$. Let $Z$ be an irreducible 
component of $X_{R}\setminus R.{\goth t}$. By Proposition~\ref{psav2}(i), $Z$ has 
dimension $n-1$. So, by Proposition~\ref{psav3}, $Z$ is not contained in $X_{R,\n}$. 
Moreover, $Z$ is invariant under $R$. Then, by Proposition~\ref{psav1}, for some complete 
subset $\Lambda $ of ${\cal R}$, strictly contained in ${\cal R}$ and for some 
irreducible closed subset $Z_{\Lambda }$ of $X_{R_{\Lambda }}$, invariant under 
$R_{\Lambda }$, $R.Z_{\Lambda }$ is dense in $Z$. By Lemma~\ref{lsav1}, 
${\goth a}_{\Lambda }$ is in ${\cal C}_{{\goth t}_{\Lambda }^{\#}}$ 
and $Z_{\Lambda }$ is the image of a closed subset $Z'_{\Lambda }$ of 
$X_{R^{\#}_{\Lambda }}$, invariant by $R^{\#}_{\Lambda }$, by the map 
$V\mapsto V\oplus {\goth t}_{\Lambda }$. Since $Z_{\Lambda }$ is contained in $Z$, 
$Z'_{\Lambda }\cap R^{\#}_{\Lambda }.{\goth t}_{\Lambda }^{\#}$ is empty. As
$\Lambda $ is strictly contained in ${\cal R}$, $\dim {\goth a}_{\Lambda }$ is smaller 
than $n$. So, by induction hypothesis, for some $\alpha $ in $\Lambda $, 
$Z'_{\Lambda }$ is contained in $X_{R^{\#}_{\Lambda },\alpha }$. As a result, 
$Z_{\Lambda }$ and $Z$ are contained in $X_{R,\alpha }$, whence $Z=X_{R,\alpha }$ since 
$Z$ is an irreducible component of $X_{R}\setminus R.{\goth t}$.
\end{proof}

\section{Normality for solvable Lie algebras} \label{ns}
Let ${\goth t}$ be a vector space of positive dimension $d$ and ${\goth a}$ in 
${\cal C}_{{\goth t}}$. Set:
$$ {\cal R} := {\cal R}_{{\goth t},{\goth a}}, \qquad 
{\goth r} := {\goth r}_{{\goth t},{\goth a}} \qquad \pi := \pi _{{\goth t},{\goth a}},
\qquad R := R_{{\goth t},{\goth a}}, \qquad A := A_{{\goth t},{\goth a}}, \qquad 
{\cal E} := {\cal E}_{{\goth t},{\goth a}}, \qquad n := \dim {\goth a} .$$
The goal of the section is to prove that $X_{R}$ is normal and Cohen-Macaulay.

\subsection{The case $\dim {\goth a}=\dim {\goth t}$} \label{ns1}
By Condition (2) of Section~\ref{sa}, ${\cal R}$ has $d$ elements 
$\poi {\beta }1{,\ldots,}{d}{}{}{}$ linearly independent. Denote by 
$\poi t1{,\ldots,}{d}{}{}{}$ the dual basis in ${\goth t}$. For $i=1,\ldots,d$, let 
$v_{i}$ be a generator of ${\goth a}^{\beta _{i}}$.

\begin{lemma}\label{lns1}
If $\dim {\goth a}=\dim {\goth t}$ then $X_{R}$ is a smooth variety. Moreover, for all 
$(\poi z1{,\ldots,}{d}{}{}{})$ in $\k^{d}$, the subspace generated by 
$v_{1}+z_{1}t_{1},\ldots,v_{d}+z_{d}t_{d}$ is in $X_{R}$.
\end{lemma}

\begin{proof}
According to Lemma~\ref{l2sa1}, ${\goth a}$ is in in $X_{R}$ and the map
$$ \xymatrix{ \k^{d} \ar[rr] && X_{R}}, \qquad
(\poi z1{,\ldots,}{d}{}{}{}) \longmapsto 
{\mathrm {span}}(\{v_{1}+z_{1}t_{1},\ldots,v_{d}+z_{d}x_{d}\})$$
is an isomorphism onto an open neighborhood of ${\goth a}$ in $X_{R}$. Hence 
${\goth a}$ is a smooth point of $X_{R}$. By Corollary~\ref{c2sa6}, $R$ has only one 
fixed point ${\goth a}$ in $X_{R}$. Since for all $V$ in $X_{R}$, $R$ has a fixed point 
in $\overline{R.V}$ and ${X_{R}}_{\loc}$ is an open subset of $X_{R}$, invariant under 
$R$, $X_{R}={X_{R}}_{\loc}$.
\end{proof}

\subsection{Cohen-Macaulayness property for some algebras} \label{ns2}
Let $A_{*}$ be an integral domain and a local commutative $\k$-algebra with maximal ideal
${\goth m}$ and $\poi u1{,\ldots,}{s}{}{}{}$ a regular sequence in $A_{*}$ of elements of 
${\goth m}$. Let $\poi T1{,\ldots,}{s}{}{}{}$ be indeterminates. Set 
$B_{s} := A_{*}[\poi T1{,\ldots,}{s}{}{}{}]$ and denote by $P_{s}$ and $P'_{s}$ the 
ideals of $B_{s}$ generated by the sequences 
$u_{j}T_{k}-u_{k}T_{j},1\leq j,k\leq s$ and $u_{j}T_{1}-u_{1}T_{j},1\leq j\leq s$
respectively.

\begin{lemma}\label{lns2}
The ideal $P_{s}$ is a prime ideal of $B_{s}$.
\end{lemma}

\begin{proof}
For $s=1$, $P_{s}=\{0\}$. Suppose $s\geq 2$. Let $\tilde{P}$ be the ideal of 
$B_{s}[T_{1}^{-1}]$ generated by $P_{s}$. For $1\leq j,k\leq s$, 
$$ T_{1}(u_{j}T_{k}-u_{k}T_{j}) = T_{k}(u_{j}T_{1}-u_{1}T_{j}) + 
T_{j}(u_{1}T_{k}-u_{k}T_{1}).$$
Hence $\tilde{P}$ is the ideal of $B_{s}[T_{1}^{-1}]$ generated by 
$P'_{s}$. Setting $S_{j} := T_{j}/T_{1}$ for $j=2,\ldots,s$, 
denote by $C$ the polynomial algebra $A_{*}[\poi S2{,\ldots,}{s}{}{}{}]$ over $A_{*}$ 
so that $B_{s}[T_{1}^{-1}] = C[T_{1},T_{1}^{-1}]$ and $\tilde{P}$
is the ideal of $B_{s}[T_{1}^{-1}]$ generated by $u_{j}-u_{1}S_{j},j=2,\ldots,s$.

\begin{claim}\label{clns2}
Let $Q$ be the ideal of $C$ generated by $u_{j}-u_{1}S_{j},j=2,\ldots,s$. Then $Q$
is prime.
\end{claim}

\begin{proof}{[Proof of Claim~\ref{clns2}]}
Let $\tilde{Q}$ be the ideal of $C[u_{1}^{-1}]$ generated by $Q$. Then 
$\tilde{Q}$ is prime since it is generated by $u_{j}u_{1}^{-1}-S_{j},j=2,\ldots,s$. As a 
result, for $p$ and $q$ in $C$ such that $pq$ is in $Q$, for some nonnegative 
integer $m$, $u_{1}^{m}p$ or $u_{1}^{m}q$ is in $Q$. So it remains to prove that 
for $p$ in $C$, $p$ is in $Q$ if so is $u_{1}p$.

Let $p$ be in $C$ such that $u_{1}p$ is in $Q$. For some $\poi q2{,\ldots,}{s}{}{}{}$ in 
$C$, 
$$ u_{1}p = \sum_{j=2}^{s} q_{j}(u_{j}-u_{1}S_{j}) \quad  \text{whence} \quad
\sum_{j=1}^{s} q_{j}u_{j} = 0 \quad  \text{with} \quad 
q_{1} := -(p + \sum_{j=2}^{s} q_{j}S_{j}) .$$
By hypothesis, the sequence $\poi u1{,\ldots,}{s}{}{}{}$ is regular in $C$. So for some 
sequence $q_{j,k},1\leq j,k\leq s$ in $C$ such that $q_{j,k}=-q_{k,j}$, 
$$ q_{j} = \sum_{k=1}^{s} q_{j,k}u_{k}$$ 
for $j=1,\ldots,s$. As a result,
\begin{eqnarray*}
u_{1}p = & \sum_{j=2}^{s} \sum_{k=1}^{s} q_{j,k}u_{k} (u_{j}-u_{1}S_{j}) \\
= &  \sum_{j=2}^{s} q_{j,1}u_{j}u_{1} - 
\sum_{j=2}^{s} \sum_{k=1}^{s} q_{j,k}u_{k}u_{1}S_{j} \\ = &  
u_{1} (\sum_{j=2}^{s} q_{j,1}(u_{j}-u_{1}S_{j}) + 
\sum_{2\leq j<k\leq s} q_{j,k}(u_{j}S_{k} - u_{k}S_{j})) .
\end{eqnarray*}
For $2\leq j,k\leq s$, 
$$ u_{j}S_{k} - u_{k}S_{j} = (u_{j}-u_{1}S_{j})S_{k} - (u_{k}-u_{1}S_{k})S_{j}
\in Q,$$
whence the claim.
\end{proof}

By the claim, $\tilde{P}$ is a prime ideal of $B_{s}[T_{1}^{-1}]$ since it is 
generated by $Q$. As a result for $p$ and $q$ in $B_{s}$ such that $pq$ is in $P_{s}$, 
for some nonnegative integer $m$, $T_{1}^{m}p$ or $T_{1}^{m}q$ is in $P'_{s}$ since
$T_{1}P_{s}$ is contained in $P'_{s}$ by the equality
$$ T_{1}(u_{j}T_{k}-u_{k}T_{j}) = T_{k}(u_{j}T_{1}-u_{1}T_{j})
+T_{j}(u_{1}T_{k}-u_{k}T_{1})$$
for $1\leq i,j\leq s$. So it remains to prove that for $p$ in $B_{s}$, $p$ is in $P_{s}$ 
if $T_{1}p$ is in $P'_{s}$.

Let $p$ be in $B_{s}$ such that $T_{1}p$ is in $P'_{s}$. For some 
$\poi r2{,\ldots,}{s}{}{}{}$ in $B_{s}$, 
$$ T_{1}p = \sum_{j=2}^{s} r_{j}(u_{j}T_{1}-u_{1}T_{j}) .$$
For $j=2,\ldots,s$, $r_{j}$ has an expansion
$$ r_{j} = r_{j,0} + T_{1} r_{j,1}$$
with $r_{j,0}$ and $r_{j,1}$ in $B'_{s}:=A_{*}[\poi T2{,\ldots,}{s}{}{}{}]$ and $B_{s}$ 
respectively. Set:
$$ p' := p - \sum_{j=2}^{s} r_{j,1}(u_{j}T_{1}-u_{1}T_{j}) .$$
Then 
$$T_{1}p'= \sum_{j=2}^{s} r_{j,0} (u_{j}T_{1}-u_{1}T_{j}) $$
so that the element 
$$ \sum_{j=2}^{s} r_{j,0} u_{1}T_{j} \in B'_{s}$$
is divisible by $T_{1}$ in $B_{s}$, whence
$$ \sum_{j=2}^{s} r_{j,0} T_{j} = 0 .$$
As $\poi T2{,\ldots,}{s}{}{}{}$ is a regular sequence in $B_{s}$, for some sequence 
$r_{j,k,0},2\leq j,k\leq s$ in $B_{s}$ such that $r_{j,k,0}=-r_{k,j,0}$ for all
$(j,k)$, 
$$ r_{j,0} = \sum_{k=2}^{s} r_{j,k,0}T_{k}$$
for $j=2,\ldots,s$. Then
$$ T_{1}p' = \sum_{2\leq j,k\leq s} r_{j,k,0}T_{k}(u_{j}T_{1}-u_{1}T_{j})
=  T_{1} \sum_{2\leq j<k\leq s} r_{j,k,0}(T_{k}u_{j}-T_{j}u_{k}) .$$
As a result $p'$ and $p$ are in $P_{s}$, whence the lemma.
\end{proof}

Denote by $P''_{s}$ the ideal of $B_{s}$ generated by $P_{s-1}$ and 
$u_{s}T_{1}-u_{1}T_{s}$. Let ${\goth B}_{s}$ and ${\goth B}'_{s}$ be the quotients of
$B_{s}$ by $P_{s}$ and $P''_{s}$ respectively. The restrictions to $A_{*}$ of the 
quotient morphisms $\xymatrix{B_{s} \ar[r] & {\goth B}'_{s}}$ and 
$\xymatrix{ B_{s} \ar[r] & {\goth B}_{s}}$ are embeddings. For $j=1,\ldots,s$, denote
again by $T_{j}$ its images in ${\goth B}'_{s}$ and ${\goth B}_{s}$ by these morphisms.

\begin{lemma}\label{l2ns2}
Denote by $\overline{P_{s}}$ the image in ${\goth B}'_{s}$ of $P_{s}$ by the quotient 
morphism.

{\rm (i)} The intersection of $\overline{P_{s}}$ and $T_{1}{\goth B}'_{s}$ is equal 
to $\{0\}$.

{\rm (ii)} The ${\goth B}'_{s}$-modules $T_{1}{\goth B}'_{s}$ and ${\goth B}_{s}$ are 
isomorphic. 
\end{lemma}

\begin{proof}
Let $a$ be in $B_{s}$ such that $T_{1}a$ is in $P_{s}$. According to Lemma~\ref{lns2}, 
$P_{s}$ is a prime ideal of $B_{s}$. Hence $a$ is in $P_{s}$ since $T_{1}$ is not in 
$P_{s}$. Moreover, for $j=1,\ldots,s$,
$$ T_{1}(u_{j}T_{s}-u_{s}T_{j}) = T_{s}(u_{j}T_{1}-u_{1}T_{j})
+T_{j}(u_{1}T_{s}-u_{s}T_{1}).$$
Hence $T_{1}P_{s}$ is contained in $P''_{s}$. As a result, $\overline{P_{s}}$ is the 
kernel of the endomorphism $a\mapsto T_{1}a$ of ${\goth B}'_{s}$ and the intersection
of $\overline{P_{s}}$ and $T_{1}{\goth B}'_{s}$ is equal to $\{0\}$. As 
${\goth B}_{s}$ is the quotient of ${\goth B}'_{s}$ by $\overline{P_{s}}$, the 
endomorphism $a\mapsto T_{1}a$ defines through the quotient an isomorphism 
$$\xymatrix{{\goth B}_{s} \ar[rr] && T_{1}{\goth B}'_{s}}$$
of ${\goth B}'_{s}$-modules.
\end{proof}

Let $Q_{s}$ be the ideal of the polynomial algebra $A_{*}[\poi T2{,\ldots,}{s}{}{}{}]$
generated by the sequence $u_{i}T_{k}-u_{k}T_{i},\; 2\leq i,k\leq s$ and denote by 
${\goth B}_{s}^{\#}$ the quotient of $A_{*}[\poi T2{,\ldots,}{s}{}{}{}]$ by $Q_{s}$.

\begin{lemma}\label{l3ns2}
{\rm (i)} The quotient of the algebra ${\goth B}_{s}/T_{1}{\goth B}_{s}$ by the ideal
generated by $u_{1}$ is equal to the quotient of ${\goth B}_{s}^{\#}$ by the ideal
generated by $u_{1}$.

{\rm (ii)} The canonical map $\xymatrix{ A_{*} \ar[r] & {\goth B}_{s}/T_{1}{\goth B}_{s}}$
is an embedding.

{\rm (iii)} The ideal of ${\goth B}_{s}/T_{1}{\goth B}_{s}$ generated by $u_{1}$ is 
isomorphic to $A_{*}$
\end{lemma}

\begin{proof}
Denote by $Q'_{s}$ the ideal of $B_{s}$ generated by $P_{s}$ and $T_{1}$.

(i) As the ideal of $B_{s}$ generated by $Q'_{s}$ and $u_{1}$ is equal to the
ideal generated by $u_{1}$, $T_{1}$ and $Q_{s}$, 
${\goth B}_{s}^{\#}/u_{1}{\goth B}_{s}^{\#}$ is equal to the quotient of 
${\goth B}_{s}/T_{1}{\goth B}_{s}$ by the ideal generated by $u_{1}$.

(ii) Since the intersection of $A_{*}$ and $Q'_{s}$ is equal to $\{0\}$, the canonical 
map $\xymatrix{ A_{*} \ar[r] & {\goth B}_{s}/T_{1}{\goth B}_{s}}$ is an embedding.

(iii) For $k=2,\ldots,s$, $u_{1}T_{k}$ is in $Q'_{s}$. Hence $u_{1}B_{s}$ is contained 
in the sum of $u_{1}A_{*}$ and $Q'_{s}$. As a result, $u_{1}A_{*}$ is equal to
$u_{1}{\goth B}_{s}/T_{1}{\goth B}_{s}$ by (ii), whence the assertion since $A_{*}$ is
an integral domain.
\end{proof}

\begin{prop}\label{pns2}
Suppose that $A_{*}$ is Cohen-Macaulay. 

{\rm (i)} The algebra ${\goth B}_{s}$ is an integral domain and a Cohen-Macaulay algebra 
of dimension $\dim A_{*}+1$.

{\rm (ii)} For $\poi a1{,\ldots,}{m}{}{}{}$ regular sequence in $A_{*}$ of elements of
${\goth m}$ an for ${\goth p}$ prime ideal of ${\goth B}_{s}$ containing it, 
$\poi a1{,\ldots,}{m}{}{}{}$ is a regular sequence in the localization of 
${\goth B}_{s}$ at ${\goth p}$.
\end{prop}

\begin{proof}
(i) Prove the assertion by induction on $s$. As ${\goth B}_{1}$ is the polynomial algebra 
$A_{*}[T_{1}]$, the assertion is true for $s=1$ since $A_{*}$ an integral domain and a 
Cohen-Macaulay algebra. Suppose the assertion true for $s-1$. By induction hypothesis, 
${\goth B}_{s-1}[T_{s}]$ is an integral domain and a Cohen-Macaulay algebra as a 
polynomial algebra over ${\goth B}_{s-1}$ and its dimension is equal to $\dim A_{*}+2$. 
As a result, ${\goth B}'_{s}$ is Cohen-Macaulay of dimension $\dim A_{*}+1$ as the 
quotient of the integral domain and a Cohen-Macaulay algebra ${\goth B}_{s-1}[T_{s}]$ by 
the ideal generated by $T_{s}u_{1}-T_{1}u_{s}$. As ${\goth B}_{s}$ is the quotient of 
${\goth B}'_{s}$ by $\overline{P_{s}}$, ${\goth B}_{s}$ has dimension at most 
$\dim A_{*}+1$. By Lemma~\ref{lns2}, ${\goth B}_{s}$ is an integral domain so that 
${\goth B}_{s}/T_{1}{\goth B}_{s}$ has dimension at most $\dim A_{*}$.

By induction hypothesis again, ${\goth B}_{s}^{\#}$ is an integral domain and a Cohen-
Macaulay algebra of dimension $\dim A_{*}+1$. Hence 
${\goth B}_{s}^{\#}/u_{1}{\goth B}_{s}^{\#}$ is Cohen-Macaulay of dimension $\dim A_{*}$.
According to Lemma~\ref{l3ns2}, we have a short exact sequence
$$ \xymatrix{ 0 \ar[r] & A_{*} \ar[r] & {\goth B}_{s}/T_{1}{\goth B}_{s} \ar[r] &
{\goth B}_{s}^{\#}/u_{1}{\goth B}_{s}^{\#} \ar[r] & 0} .$$  
Hence the algebra ${\goth B}_{s}/T_{1}{\goth B}_{s}$ is Cohen-Macaulay of dimension 
$\dim A_{*}$ since $A_{*}$ and ${\goth B}_{s}^{\#}/u_{1}{\goth B}_{s}^{\#}$ are 
Cohen-Macaulay of dimension $\dim A_{*}$ and ${\goth B}_{s}/T_{1}{\goth B}_{s}$ has 
dimension at most $\dim A_{*}$. As a result, ${\goth B}_{s}$ has dimension 
$\dim A_{*}+1$. As ${\goth B}_{s}$ is the quotient of 
${\goth B}'_{s}$ by $\overline{P_{s}}$, we have a short exact sequence
$$ \xymatrix{ 0 \ar[r] & \overline{P_{s}}+T_{1}{\goth B}'_{s} \ar[r] & 
{\goth B}'_{s} \ar[r] & {\goth B}_{s}/T_{1}{\goth B}_{s} \ar[r] & 0 }.$$
Then, setting $M := \overline{P_{s}}+T_{1}{\goth B}'_{s}$ and denoting by $M_{*}$ the 
localization of $M$ at a maximal ideal of ${\goth B}'_{s}$, containing $T_{1}$, 
$$ {\mathrm {Ext}}^{j}(\k,M_{*}) = 0$$
for $j\leq \dim A_{*}$ since ${\goth B}'_{s}$ and ${\goth B}_{s}/T_{1}{\goth B}_{s}$
have dimension $\dim A_{*}+1$ and $\dim A_{*}$. By Lemma~\ref{l2ns2}(i), 
$M$ is the direct sum $\overline{P_{s}}$ and $T_{1}{\goth B}'_{s}$. So, denoting
by $(T_{1}{\goth B}'_{s})_{*}$ the localization of $T_{1}{\goth B}'_{s}$ at a maximal
ideal of ${\goth B}'_{s}$, 
$$ {\mathrm {Ext}}^{j}(\k,(T_{1}{\goth B}'_{s})_{*}) = 0$$
for $j\leq \dim A_{*}$ since $(T_{1}{\goth B}'_{s})_{*}$ is the localization of 
${\goth B}'_{s}$ at this maximal ideal when it does not contain $T_{1}$. As a result, 
by Lemma~\ref{l2ns2}(ii), ${\goth B}_{s}$ is Cohen-Macaulay since it has dimension 
$\dim A_{*}+1$.

(ii) Let ${\goth q}$ be a minimal prime ideal of ${\goth B}_{s}$, containing 
$\poi a1{,\ldots,}{m}{}{}{}$. Since $A_{*}$ is embedded in ${\goth B}_{s}$, 
${\goth q}\cap A_{*}$ is a prime ideal of $A_{*}$ containing 
$\poi a1{,\ldots,}{m}{}{}{}$. As $A_{*}$ is Cohen-Macaulay and 
$\poi a1{,\ldots,}{m}{}{}{}$ is a regular sequence in $A_{*}$, ${\goth q}\cap A_{*}$ has 
height at least $m$ and $A_{*}/{\goth q}\cap A_{*}$ has dimension at most 
$\dim A_{*}-m$ by \cite[Ch. 6, Theorem 17.4]{Mat}. Then ${\goth B}_{s}/{\goth q}$ has 
dimension at most $\dim A_{*}+1-m$ since the fraction field of ${\goth B}_{s}/{\goth q}$
is generated by the fraction field of $A_{*}/{\goth q}\cap A_{*}$ and the image of $T_{1}$
by the quotient morphism $\xymatrix{ B_{s} \ar[r] & {\goth B}_{s}/{\goth q}}$.
As a result, by (i) and \cite[Ch. 6, Theorem 17.4]{Mat}, ${\goth q}$ has height at least 
$m$. As a minimal prime ideal of ${\goth B}_{s}$ containing $m$ elements, ${\goth q}$ has
height at most $m$. Hence all minimal prime ideal of ${\goth B}_{s}$, containing 
$\poi a1{,\ldots,}{m}{}{}{}$, has height $m$. So, by (i) and 
\cite[Ch. 6, Theorem 17.4]{Mat}, $\poi a1{,\ldots,}{m}{}{}{}$ is a regular sequence in 
the localization of ${\goth B}_{s}$ at ${\goth p}$.
\end{proof}

\subsection{Normality and Cohen-Macaulayness property for $X_{R}$} \label{ns3}
Let $V_{0}$ be a fixed point under the action of $R$ in $X_{R}$ and  
$\poi {\beta }1{,\ldots,}{d}{}{}{}$ the elements of ${\cal R}_{V_{0}}$. By 
Corollary~\ref{c2sa6}(ii), $\poi {\beta }1{,\ldots,}{d}{}{}{}$ is a basis of 
${\goth t}^{*}$. Let $\poi t1{,\ldots,}{d}{}{}{}$ be the dual basis. Denote by $m$ the 
codimension of $V_{0}$ in ${\goth a}$. According to Lie's Theorem, for $m>0$, the 
elements $\poi {\gamma }1{,\ldots,}{m}{}{}{}$ of 
${\cal R}\setminus \{\poi {\beta }1{,\ldots,}{d}{}{}{}\}$ can be ordered so that
$$ {\goth a}_{i} := 
V_{0} \oplus {\goth a}^{\gamma _{1}} \oplus \cdots \oplus {\goth a}^{\gamma _{i}}$$
is an algebra of codimension $m-i$ of ${\goth a}$ for $i=1,\ldots,m$. Set:
$$ {\cal R}' := {\cal R}\setminus \{\gamma _{m}\}, \qquad {\goth a}'={\goth a}_{m-1},
\qquad  {\goth r}' := {\goth r}_{{\goth t},{\goth a}'}, \qquad 
\pi ':= \pi _{{\goth t},{\goth a}'}, \qquad R' := R_{{\goth t},{\goth a}'}, \qquad 
A' := A_{{\goth t},{\goth a}'},$$ 
$$ E := \bigoplus _{i=1}^{m} {\goth a}^{\gamma _{i}}, \qquad E' := E\cap {\goth a}' .$$
Denote by $\kappa $ the map
$$ \xymatrix{ {\mathrm {Hom}}_{\k}(V_{0},E\oplus {\goth t}) \ar[rr]^{\kappa } &&
\ec {Gr}r{}{}d}, \qquad \varphi \longmapsto 
{\mathrm {span}}(\{v+\varphi (v) \; \vert \; v \in V_{0}\}) .$$
Then $\kappa $ is an isomorphism from ${\mathrm {Hom}}_{\k}(V_{0},E\oplus {\goth t})$ onto
an affine open neighbourhood of $V_{0}$ in $\ec {Gr}r{}{}d$. Moreover, there is a short
exact sequence 
$$ \xymatrix{0 \ar[r] & {\mathrm {Hom}}_{\k}(V_{0},\k x_{\gamma _{m}}) \ar[r] & 
{\mathrm {Hom}}_{\k}(V_{0},E\oplus {\goth t}) \ar[r]^{p} &
{\mathrm {Hom}}_{\k}(V_{0},E'\oplus {\goth t}) \ar[r] &  0}.$$ 
Let $\Omega $ and $\Omega '$ be the inverse images by $\kappa $ of the intersections of 
the image of $\kappa $ with $X_{R}$ and $X_{R'}$ respectively. For $\varphi $ in 
$\Omega $ and $i=1,\ldots,d$
$$ \varphi (v_{i}) = \sum_{i=1}^{d} z_{i,j}(\varphi ) t_{j} + 
\sum_{j=1}^{m} a_{i,j}(\varphi ) x_{\gamma _{j}}$$
so that the $z_{i,j}$'s, $1\leq i,j\leq d$ and the $a_{i,j}$'s, $1\leq i\leq d$ and 
$1\leq j\leq m$ are regular functions on $\Omega $. 

Let $\psi $ be the map
$$ \xymatrix{\k \times \Omega ' \ar[rr]^{\psi } && X_{R} }, \qquad
(s,\varphi ) \longmapsto \exp(s\ad x_{\gamma _{m}}).\kappa (\varphi ) .$$

\begin{lemma}\label{lns3}
Let $O$ be the subset of elements $(s,\varphi )$ of $\k\times \Omega '$ such that
$\psi (s,\varphi )$ is in $\kappa (\Omega )$.

{\rm (i)} The subset $O$ of $\k\times \Omega '$ is open and contains 
$\{0\}\times \Omega '$.

{\rm (ii)} The map
$$ \xymatrix{ O \ar[rr]^{\overline{\psi }} && \Omega },\qquad  
(s,\varphi ) \longmapsto \kappa ^{-1}\rond \psi (s,\varphi )$$
is a birational morphism from $O$ to $\Omega $. In particular, the function 
$(s,\varphi )\mapsto s$ is in $\k(\Omega )$.
\end{lemma}

\begin{proof}
(i) As $\kappa (\Omega )$  is an open neighborhood of $V_{0}$ in $X_{R}$, $O$ is 
an open subset of $\k\times \Omega '$, containing $\{0\}\times \Omega '$ since
$\psi $ is a regular map such that $\psi (0,\varphi )=\kappa (\varphi )$ for all 
$\varphi $ in $\Omega '$.

(ii) Let $\Omega ^{c}$ be the subset of elements $\varphi $ of $\Omega $ such that 
$\kappa (\varphi )$ is in  $A.{\goth t}$. Then $\Omega ^{c}$ is a dense open subset of 
$\Omega $. Let $O'$ be the inverse image of $\Omega ^{c}$ by $\overline{\psi }$. 
Let $(s,\varphi )$ and $(s',\varphi ')$ be in $O'$ such that
$\overline{\psi }(s,\varphi )=\overline{\psi }(s',\varphi ')$, that is
$$ \exp(s\ad x_{\gamma _{m}}).\kappa (\varphi ) =
\exp(s'\ad x_{\gamma _{m}}).\kappa (\varphi ') \quad  \text{whence} \quad
\exp((s-s')\ad x_{\gamma _{m}}).\kappa (\varphi ) = \kappa (\varphi ') . $$
According to the above notations, for $i=1,\ldots,d$, 
$$ \varphi (v_{i}) = \sum_{j=1}^{d} z_{i,j}(\varphi ) t_{j} + 
\sum_{j=1}^{m-1} a_{i,j}(\varphi ) x_{\gamma _{j}}.$$
Since $\kappa (\varphi )$ is in $A.{\goth t}$, 
$$ \det ([z_{i,j}(\varphi ),1\leq i,j\leq d]) \neq 0 .$$
For $i=1,\ldots,d$, 
$$ \exp((s-s')\ad x_{\gamma _{m}})(\sum_{j=1}^{d} z_{i,j}(\varphi )t_{j}) = 
\sum_{j=1}^{d} z_{i,j}(\varphi )t_{j} -(s-s')(\sum_{j=1}^{d} z_{i,j}(\varphi )
\gamma _{m}(t_{j}))x_{\gamma _{m}} .$$
For some $j$, $\gamma _{m}(t_{j}) \neq 0$, whence $s=s'$ since $\kappa (\varphi ')$ is
contained in ${\goth r}'$. As a result, the restriction of $\overline{\psi }$ to $O'$
is injective, whence the assertion since $\overline{\psi }$ is a dominant morphism.
\end{proof}

For $i=1,\ldots,d$ and $\gamma $ in ${\goth t}^{*}$, denote by $u_{i,\gamma }$ the 
function on $\Omega $,
$$ u_{i,j} := z_{i,1}\gamma (t_{1}) + \cdots + z_{i,d} \gamma (t_{d}) .$$
Let ${\goth A}$ be the subalgebra of $\k[\Omega ]$ generated by the functions 
$z_{i,j}$'s, $1\leq i,j\leq d$ and $a_{i,j}$'s, $1\leq i\leq d$ and 
$1\leq j\leq m-1$.

\begin{lemma}\label{l2ns3}
Let $\iota $ be the restriction morphism from $\Omega $ to $\Omega '$.

{\rm (i)} The restriction of $\iota $ to ${\goth A}$ is an isomorphism onto 
$\k[\Omega ']$.

{\rm (ii)} For $1\leq i,j\leq d$, 
$u_{i,\gamma _{m}}a_{j,m}-u_{j,\gamma _{m}}a_{i,m}$ is equal to $0$.

{\rm (iii)} For $i=1,\ldots,d$ and $\gamma $ in ${\goth t}^{*}$, if 
$\gamma (t_{i})\neq 0$ then $u_{i,\gamma }$ is different from $0$.  
\end{lemma}

\begin{proof}
(i) For $1\leq i,j\leq d$, denote by 
$z'_{i,j}$ the restriction of $z_{i,j}$ to $\Omega '$ and for $1\leq i\leq d$ and 
$1\leq j\leq m-1$ denote by $a'_{i,j}$ the restriction of $a_{i,j}$ to $\Omega '$. Since 
$\k[\Omega ']$ is generated by the functions 
$$z'_{i,j}, \ 1\leq i,j\leq d \quad  \text{and} \quad a'_{i,j}, \
1\leq i\leq d,1\leq j\leq m-1,$$
the restriction of $\iota $ to ${\goth A}$ is surjective. Let ${\goth p}$ be the kernel
of the restriction of $\iota $ to ${\goth A}$. It remains to prove ${\goth p}=\{0\}$.

For $1\leq i,j\leq d$ and $k=1,\ldots,m-1$, denote by $\overline{z}_{i,j}$ and 
$\overline{a}_{i,k}$ the functions on $\k \times \Omega '$ such that
\begin{eqnarray*}
\exp(s\ad x_{\gamma _{m}})(v_{i}+\sum_{j=1}^{d} z'_{i,j}(\varphi ) t_{j} +
\sum_{k=1}^{m-1} a'_{i,k}(\varphi ) x_{\gamma _{k}}) - \\ 
(\sum_{j=1}^{d} \overline{z}_{i,j}(s,\varphi ) t_{j} -
\sum_{j=1}^{d} sz_{i,j}(\varphi )\gamma _{m}(t_{j})x_{\gamma _{m}} +
\sum_{k=1}^{m-1} \overline{a}_{i,k}(s,\varphi ) x_{\gamma _{k}}) \in V_{0} .
\end{eqnarray*}
Then $\overline{z}_{i,j}$ and $\overline{a}_{i,k}$ are regular functions on 
$\k\times \Omega '$ as restrictions to $\k\times \Omega '$ of regular functions on 
$\k\times {\mathrm {Hom}}(V_{0},E'\oplus {\goth t})$. Let $\overline{{\goth A}}$ be the 
subalgebra of $\k[\Omega '][s]$ generated by the functions 
$$\overline{z}_{i,j},i,j=1,\ldots,d \quad  \text{and} \quad
\overline{a}_{i,k},i=1,\ldots,d,k=1,\ldots,m-1 .$$
Since $z'_{i,j}(\varphi )=\overline{z}_{i,j}(0,\varphi )$ and 
$a'_{i,k}(\varphi )=\overline{a}_{i,k}(0,\varphi )$ for all $\varphi $ in $\Omega '$, the
restriction to $\overline{{\goth A}}$ of the quotient morphism 
$\xymatrix{\k[\Omega '][s] \ar[r]& \k[\Omega ']}$ is surjective. As a result, 
$\overline{{\goth A}}$ has dimension $n$ or $n-1$ since $\Omega '$ and 
$\k[\Omega '][s]$ have dimension $n-1$ and $n$ respectively. As 
$\exp(s\ad x_{\gamma _{m}})(v_{i})$ is not necessarily equal to $v_{i}$, 
$$p\rond \psi \neq 
(\overline{z}_{i,j},\overline{a}_{i,j}, \; 1\leq i\leq d,1\leq j \leq m-1) .$$ 
Moreover, $\Omega '$ is contained in $p(\Omega )$ by Lemma~\ref{lns3}(i) but the
inclusion may be strict.

\begin{claim}\label{clns3}
The algebra $\overline{{\goth A}}$ has dimension $n-1$.
\end{claim}

\begin{proof}{[Proof of Claim~\ref{clns3}]}
There are two cases to consider:
\begin{itemize}
\item [{\rm (1)}] for $i=1,\ldots,m-1$, 
$[{\goth a}^{\gamma _{m}},{\goth a}^{\gamma _{i}}]$ is contained in $V_{0}$,
\item [{\rm (2)}] for some $i$ in $\{1,\ldots,m-1\}$,
$[{\goth a}^{\gamma _{m}},{\goth a}^{\gamma _{i}}]$ is not contained in $V_{0}$.
\end{itemize}
In the first case, $\overline{{\goth A}}=\k[\Omega ']$. Otherwise, denote by 
$j$ the biggest integer such that $[{\goth a}^{\gamma _{m}},{\goth a}^{\gamma _{j}}]$ is 
not contained in $V_{0}$ and $a'_{i,j}\neq 0$ for some $i=1,\ldots,d$. Then, for some 
$j'$ smaller than $j$, $\gamma _{m}+\gamma _{j}=\gamma _{j'}$. Furthermore, 
for $k<j$ such that $[{\goth a}^{\gamma _{m}},{\goth a}^{\gamma _{k}}]$ is not contained 
in $V_{0}$, $\gamma _{m}+\gamma _{k}$ is in 
${\cal R}\setminus \{\poi {\gamma }{j'}{,\ldots,}{m}{}{}{}\}$. Then for $k\geq j'$ and 
$i=1,\ldots,d$, $a'_{i,k}=\overline{a}_{i,k}$ and for all $(s,\varphi )$ in 
$\k\times \Omega '$, 
$$ \overline{a}_{i,j'}(s,\varphi ) = a'_{i,j'}(\varphi ) + s a'_{i,j}(\varphi ) .$$
As a result, by induction on $m-k$, for $i=1,\ldots,d$, 
$$ a'_{i,k} - \overline{a}_{i,k} \in s\overline{{\goth A}}[s] .$$
Hence $\k[\Omega '][s]=\overline{{\goth A}}[s]$ and there exists a surjective morphism 
$\xymatrix{ \k[\Omega '] \ar[r]& \overline{{\goth A}}}$ so that $\overline{{\goth A}}$ 
has dimension $n-1$.
\end{proof}

According to Lemma~\ref{lns3}(ii), the comorphism of $\overline{\psi }$ is an embedding 
of $\k[\Omega ]$ into $\k[O]$ and from this embedding results an isomorphism from 
$\k(\Omega )$ onto $\k(\Omega ')(s)$. Moreover, $\overline{{\goth A}}$ is the image of 
${\goth A}$ by this embedding so that ${\goth A}$ has dimension $n-1$. As a result, 
${\goth p}=\{0\}$ since $\iota $ is surjective and $\Omega '$ has dimension $n-1$.

(ii) Let $\varphi $ be in $\Omega $. Since $\kappa (\varphi )$ is a commutative algebra, 
for $1\leq i,j\leq d$, 
$$ 0 = [v_{i}+\varphi (v_{i}),v_{j}+\varphi (v_{j})] = [v_{i},\varphi (v_{j})] +
[\varphi (v_{i}),v_{j}] + [\varphi (v_{i}),\varphi (v_{j})] .$$
The component on $x_{\gamma _{m}}$ of the right hand side is
$$ \sum_{k=1}^{d} (z_{i,k}a_{j,m}(\varphi )-z_{j,k}a_{i,m}(\varphi )) 
[t_{k},x_{\gamma _{m}}] = 
(u_{i,\gamma _{m}}a_{j,m}-u_{j,\gamma _{m}}a_{i,m})(\varphi )x_{\gamma _{m}},$$
whence the assertion.

(iii) Denote by $R_{0}$ the adjoint group of 
${\goth r}_{0} := {\goth t}+V_{0}$ and $X_{R_{0}}$ the closure in $\ec {Gr}r{}0{d}$ of 
$R_{0}.{\goth t}$. Let $\Omega _{0}$ be the inverse image of $X_{R_{0}}$ by $\kappa $. 
According to Lemma~\ref{lns1}, for $i,j=1,\ldots,d$, the restriction to $\Omega _{0}$ of 
$z_{i,j}$ is equal to $0$ if $j\neq i$, otherwise it is different from $0$. As a result, 
for $i=1,\ldots,d$ and $\gamma $ in ${\goth t}^{*}$, the restriction of $u_{i,\gamma }$ 
to $\Omega _{0}$ is equal to $\overline{z_{i,i}}\gamma (t_{i})$ with $\overline{z_{i,i}}$
the restriction of $z_{i,i}$ to $\Omega _{0}$, whence the assertion.
\end{proof}

For $\gamma $ in ${\goth t}^{*}$, set:
$$ I_{\gamma } := \{j \in \{1,\ldots,d\} \ \vert \ \gamma (t_{j}) \neq 0 \}.$$

\begin{prop}\label{pns3}
Denote by $\k[\Omega ]_{0}$ the localization of $\k[\Omega ]$ at $0$.

{\rm (i)} The local algebra $\k[\Omega ]_{0}$ is Cohen-Macaulay.

{\rm (ii)} For $\gamma $ in ${\goth t}^{*}$, $u_{i,\gamma },i\in I_{\gamma }$ is a 
regular sequence in $\k[\Omega ]_{0}$ of elements of its maximal ideal.
\end{prop}

\begin{proof}
Prove the proposition by induction on $m$. By Lemma~\ref{lns1}, for $m=0$, $\k[\Omega ]$ 
is a polynomial algebra of dimension $d$, generated by 
$\poi z{1,1,0}{,\ldots,}{d,d,0}{}{}{}$. Moreover, for $i=1,\ldots,d$ and $\gamma $ in 
${\goth t}^{*}$, $u_{i,\gamma }=z_{i,i}\gamma (t_{i})$, whence the proposition for $m=0$.
Suppose $m>0$ and the proposition true for $m-1$ and use the notations of 
Lemma~\ref{l2ns3}.

According to Lemma~\ref{l2ns3}(i) and the induction hypothesis, the localization 
${\goth A}_{*}$ of ${\goth A}$ at $0$ is Cohen-Macaulay and for $\gamma $ in 
${\goth t}^{*}$, $u_{i,\gamma },i\in I_{\gamma }$ is a regular sequence in 
${\goth A}_{*}$ of elements of its maximal ideal. Denote by ${\goth B}$ the polynomial 
algebra ${\goth A}_{*}[T_{i},i\in I_{\gamma _{m}}]$ and by $P$ the ideal of ${\goth B}$ 
generated by the sequence 
$u_{i,\gamma _{m}}T_{j}-u_{j,\gamma _{m}}T_{i},(i,j)\in I_{\gamma _{m}}^{2}$. 
According to Condition (3) of Section~\ref{sa}, 
$s := \vert I_{\gamma _{m}} \vert \geq 2$. By Lemma~\ref{l2ns3}(ii), 
$\k[\Omega ]_{0}$ is a quotient of the localization at $0$ of ${\goth B}/P$ and by 
Lemma~\ref{lns2}, $P$ is a prime ideal of ${\goth B}$. By Proposition~\ref{pns2}(i), 
${\goth B}/P$ is an integral domain and a Cohen-Macaulay algebra of dimension $n$ since 
$\k[\Omega ']$ has dimension $n-1$. Hence $\k[\Omega ]_{0}$ is the localization of 
${\goth B}/P$ at $0$ since $\k[\Omega ]_{0}$ is an integral domain of dimension $n$. As a
result, $\k[\Omega ]_{0}$ is Cohen-Macaulay and by Proposition~\ref{pns2}(ii), for 
$\gamma $ in ${\goth t}^{*}$, the sequence $u_{i,\gamma },i\in I_{\gamma }$ is regular in
$\k[\Omega ]_{0}$.
\end{proof}

\begin{theo}\label{tns3}
The variety $X_{R}$ is normal and Cohen-Macaulay.
\end{theo}

\begin{proof}
By Corollary~\ref{csav3}, $X_{R}$ is smooth in codimension $1$. So, by Serre's normality
criterion~\cite[\S 1,no 10, Th\'eor\`eme 4]{Bou1}, it suffices to prove that $X_{R}$ is 
Cohen-Macaulay. According to~\cite[Ch. 8, Theorem 24.5]{Mat}, the set of points $x$ of 
$X_{R}$ such that $\an {X_{R}}x$ is Cohen-Macaulay, is open. For $x$ in $X_{R}$, the 
closure in $X_{R}$ of $R.x$ contains a fixed point. So it suffices to prove that for $x$ 
a fixed point under the action of $R$ in $X_{R}$, $\an {X_{R}}x$ is Cohen-Macaulay. Let 
$V_{0}$ and $\Omega $ be as in Lemma~\ref{lns3}. Then $\Omega $ is an 
affine open neighborhood of $V_{0}$ in $X_{R}$. By Proposition~\ref{pns3}(i), 
$\an {\Omega }0$ is Cohen-Macaulay, whence the theorem since $\kappa $ is an isomorphism 
from $\Omega $ onto an open neighborhood of $V_{0}$ in $X_{R}$ and $\kappa (0)=V_{0}$.
\end{proof}

\subsection{Nipotent cone and regular sequence in ${\cal O}_{{\cal E}}$}\label{ns4}
Let $\poi {\beta }1{,\ldots,}{d}{}{}{}$ be a basis of ${\goth t}^{*}$. For 
$i=1,\ldots,d$, denote again by $\beta _{i}$ the element of ${\goth r}^{*}$ extending
$\beta _{i}$ and equal to $0$ on ${\goth a}$. For $\Lambda $ a complete subset of 
${\cal R}$, denote by ${\goth t}_{\Lambda }^{\#}$ a complement to
${\goth t}_{\Lambda }$ in ${\goth t}$ and set 
$$ R'_{\Lambda } := R_{{\goth t}_{\Lambda }^{\#},{\goth a}_{\Lambda }}
\quad  \text{and} \quad {\cal E}_{\Lambda } := 
{\cal E}_{{\goth t}_{\Lambda }^{\#},{\goth a}_{\Lambda }}.$$ 
For $Y$ closed subset of $X_{R'_{\Lambda }}$, denote by ${\cal E}_{\Lambda ,Y}$ the 
restriction of ${\cal E}_{\Lambda }$ to $Y$. Let ${\cal N}'_{\Lambda }$ be the image 
of the map
$$ \xymatrix{ {\cal E}_{\Lambda , X_{R'_{\Lambda },\n}} \ar[rr] && {\cal E}}, \qquad 
(V,x) \longmapsto (V\oplus {\goth t}_{\Lambda },x) $$
and ${\cal N}_{\Lambda }$ the closure in ${\cal E}$ of $R.{\cal N}'_{\Lambda }$.

\begin{lemma}\label{lns4}
For $i=1,\ldots,d$, let $\tilde{\beta }_{i}$ be the function on ${\cal E}$ defined by
$\tilde{\beta }_{i}(V,x)=\beta _{i}(x)$. Denote by ${\cal N}$ the nullvariety of 
$\poi {\tilde{\beta }}1{,\ldots,}{d}{}{}{}$ in ${\cal E}$.

{\rm (i)} For all complete subset $\Lambda $ of ${\cal R}$, ${\cal N}_{\Lambda }$ is
a subvariety of ${\cal N}$ of dimension at most $n$.

{\rm (ii)} The varieyt ${\cal N}$ is the union of 
${\cal N}_{\Lambda }, \Lambda \in {\cal P}_{c}({\cal R})$.

{\rm (iii)} The variety ${\cal N}$ is equidimensional of dimension $n$. 
\end{lemma}

\begin{proof}
(i) Since ${\goth a}$ is the nullvariety of $\poi {\beta }1{,\ldots,}{d}{}{}{}$ in 
${\goth r}$, ${\cal N}$ is the intersection of ${\cal E}$ and 
$X_{R}\times {\goth a}$. By definition ${\cal N}'_{\Lambda }$ is contained in 
$X_{R}\times {\goth a}$. Hence ${\cal N}_{\Lambda }$ is contained in ${\cal N}$. 
By Proposition~\ref{psav3}, 
$$ \dim {\cal N}'_{\Lambda } = \dim {\goth t}_{\Lambda }^{\#} + \dim X_{R'_{\Lambda },\n}
\leq \dim {\goth a}_{\Lambda }.$$
Since the image of $X_{R'_{\Lambda },\n}$ by the map 
$V\mapsto V\oplus {\goth t}_{\Lambda }$ is invariant by $R_{\Lambda }$,
$$ \dim {\cal N}_{\Lambda } \leq \dim {\cal N}'_{\Lambda }+
\dim {\goth a}-\dim {\goth a}_{\Lambda } \leq \dim {\goth a} .$$

(ii) Let $\varpi _{1}$ be the bundle projection of the vector bundle ${\cal E}$ over 
$X_{R}$ and $\tau _{1}$ the restriction to ${\cal E}$ of the projection
$\xymatrix{X_{R}\times {\goth r} \ar[r] & {\goth r}}$. Let $T$ be an irreducible 
component of ${\cal N}$. For all $V$ in $\varpi _{1}(T)$,
$\tau _{1}(\varpi _{1}^{-1}(V)\cap T)$ is a closed cone of ${\goth a}$. Hence 
$\varpi _{1}(T)\times \{0\}$ is the intersection of $T$ and $X_{R}\times \{0\}$ so that
$\varpi _{1}(T)$ is a closed subset of $X_{R}$. Since ${\cal N}$ is the intersection of 
${\cal E}$ and $X_{R}\times {\goth a}$, ${\cal N}$ and its irreducible components
are invariant under $R$. As a result, $\varpi _{1}(T)$ is invariant under $R$ and by
Proposition~\ref{psav1}, for some complete subset $\Lambda $ of ${\cal R}$ and for some 
closed subset of $Z_{\Lambda }$ of $X_{R_{\Lambda }}$, 
$\varpi _{1}(T)=\overline{R.Z_{\Lambda }}$. Moreover, by Lemma~\ref{lsav1}, for some
closed subset $Z'_{\Lambda }$ of $X_{R'_{\Lambda },\n}$, $Z_{\Lambda }$ is the image of 
$Z'_{\Lambda }$ by the map $V\mapsto V\oplus {\goth t}_{\Lambda }$. As a result, 
$$ {\cal E}_{\Lambda ,Z'_{\Lambda }} \subset {\cal E}_{\Lambda ,X_{R'_{\Lambda },\n}} 
\quad  \text{and} \quad
\varpi _{1}^{-1}(Z_{\Lambda })\cap X_{R}\times {\goth a} \subset {\cal N}'_{\Lambda } .$$
Then $T$ is contained in ${\cal N}_{\Lambda }$, whence the assertion by (i).

(iii) By (i) and (ii), since ${\cal R}$ is finite, the irreducible components of 
${\cal N}$ have dimension at most $n$. As the nullvariety of $d$ functions 
on the irreducible variety ${\cal E}_{X_{R}}$, the irreducible components of ${\cal N}$
have dimension at least $n$, whence the assertion.
\end{proof}

For $x$ in ${\cal E}$, denote by $I_{x}$ the subset of elements $i$ of 
$\{1,\ldots,d\}$ such that $\tilde{\beta }_{i}(x)=0$.

\begin{coro}\label{cns4}
For all $x$ in ${\cal E}$, the sequence $\tilde{\beta }_{i},i\in I_{x}$ is regular 
in $\an {{\cal E}}x$.
\end{coro}

\begin{proof}
According to Lemma~\ref{lns4}, for all subset $I$ of $\{1,\ldots,d\}$, the nullvariety
of $\tilde{\beta }_{i},i\in I$ in ${\cal E}$ is equidimensional of dimension 
$n+d-\vert I \vert$. By Theorem~\ref{tns3} and Lemma~\ref{lsi}(iii),
${\cal E}$ is Cohen-Macaulay as a vector bundle over a Cohen-Macaulay variety, whence 
the corollary by \cite[Ch. 6, Theorem 17.4]{Mat}.
\end{proof}

\section{Rational singularities for solvable Lie algebras} \label{rss}
Let ${\goth t}$ be a vector space of positive dimension $d$. Denote by 
${\cal C}_{{\goth t},*}$ the full subcategory of ${\cal C}_{{\goth t}}$ whose objects 
${\goth a}$ satisfy the following condition: 
\begin{itemize}
\item [{\rm (4)}] there exist regular maps $\poi {\varepsilon }1{,\ldots,}{d}{}{}{}$ from 
${\goth r}_{{\goth t},{\goth a}}$ to ${\goth r}_{{\goth t},{\goth a}}$ such that
$\poi x{}{,\ldots,}{}{\varepsilon }{1}{d}$ is a basis of 
${\goth r}_{{\goth t},{\goth a}}^{x}$ for all $x$ in 
${{\goth r}_{{\goth t},{\goth a}}}_{\r}$. 
\end{itemize}
According to~\cite[Theorem 9]{Ko}, ${\goth u}$ is in ${\cal C}_{{\goth h},*}$.

\begin{lemma}\label{lrss}
Let ${\goth a}$ be in ${\cal C}_{{\goth t},*}$ and ${\goth a}'$
an ideal of ${\goth t}+{\goth a}$, contained in ${\goth a}$ and containing a fixed point 
under the action of $R_{{\goth t},{\goth a}}$ in $X_{R_{{\goth t},{\goth a}}}$. Then 
${\goth a}'$ is in ${\cal C}_{{\goth t},*}$.
\end{lemma}

\begin{proof}
Set ${\goth r}:={\goth t}+{\goth a}$ and ${\goth r}':={\goth t}+{\goth a}'$. According
to Corollary~\ref{c2sa6}(ii), ${\goth a}'$ is in ${\cal C}_{{\goth t}}$ since it is in 
${\cal C}'_{{\goth t}}$. Set ${\goth t}_{\r}:={\goth r}_{\r}\cap {\goth t}$. As 
${\cal R}_{{\goth t},{\goth a}'}$ is contained in ${\cal R}_{{\goth t},{\goth a}}$, 
${\goth t}_{\r}$ is contained in ${\goth r}'_{\r}$ by Lemma~\ref{lsav2}(i). Then 
${\goth r}'_{\r}$ is contained in ${\goth r}_{\r}$ and for all $x$ in 
$A_{{\goth t},{\goth a'}}.{\goth t}_{\r}$, ${\goth r}^{x}={{\goth r}'}^{x}$ since 
$A_{{\goth t},{\goth a'}}.{\goth t}_{\r}$ is a dense open subset of ${\goth r}'$ by 
Lemma~\ref{lsav2}(i). So, for all regular map $\varepsilon $ from ${\goth r}$ to 
${\goth r}$ such that $[x,\varepsilon (x)]=0$ for all $x$ in ${\goth r}$, 
$\varepsilon (x)$ is in ${\goth r}'$ for all $x$ in ${\goth r}'$, whence the lemma.
\end{proof}

Let ${\goth a}$ be in ${\cal C}_{{\goth t},*}$. Set:
$$ {\cal R} := {\cal R}_{{\goth t},{\goth a}}, \qquad 
{\goth r} := {\goth r}_{{\goth t},{\goth a}} \qquad \pi := \pi _{{\goth t},{\goth a}},
\qquad R := R_{{\goth t},{\goth a}}, \qquad A := A_{{\goth t},{\goth a}}, \qquad 
{\cal E} := {\cal E}_{{\goth t},{\goth a}}, \qquad n := \dim {\goth a}.$$
The goal of the section is to prove that $X_{R}$ is Gorenstein with rational 
singularities. 

For $k$ positive integer, set:
$$ {\cal E}^{(k)} := \{(u,\poi x1{,\ldots,}{k}{}{}{}) \in X_{R}\times {\goth r}^{k} 
\ \vert \ \poi {u\ni x}1{,\ldots,}{k}{}{}{}\}$$
and denote by ${\goth X}_{R,k}$ the image of ${\cal E}^{(k)}$ by the projection
$$ (u,\poi x1{,\ldots,}{k}{}{}{}) \longmapsto (\poi x1{,\ldots,}{k}{}{}{}) .$$
Since $X_{R}$ is a projective variety, ${\goth X}_{R,k}$ is a closed subset of 
${\goth r}^{k}$, invariant under the diagonal action of $R$ in ${\goth r}^{k}$.

\subsection{Differential forms on some smooth open subsets of ${\goth X}_{R,k}$} 
\label{rss1}
For $j=1,\ldots,k$, let $V_{j}^{(k)}$ be the subset of elements of ${\goth X}_{R,k}$ whose
$j$-th component is in ${\goth r}_{\r}$.

\begin{lemma}\label{lrss1}
For $j=1,\ldots,k$, $V_{j}^{(k)}$ is a smooth open subset of ${\goth X}_{R,k}$. Moreover,
$\Omega _{V_{j}^{(k)}}$ has a global section without zero.
\end{lemma}

\begin{proof}
Denoting by $\sigma _{j}$ the automorphism of ${\goth r}_{k}$ which permutes the first
and the $j$-th component, ${\goth X}_{R,k}$ is invariant under $\sigma _{j}$ and 
$\sigma _{j}(V_{1}^{(k)}) = V_{j}^{(k)}$ so that we can suppose $j=1$. Moreover, for 
$k=1$, ${\goth X}_{R,k}={\goth r}$ so that we can suppose $k\geq 2$. By definition, 
$V_{1}^{(k)}$ is the intersection of ${\goth r}_{\r}\times {\goth r}^{k-1}$ and 
${\goth X}_{R,k}$. Hence $V_{1}^{(k)}$ is an open susbet of ${\goth X}_{R,k}$ since 
${\goth r}_{\r}$ is an open subset of ${\goth r}$.

Let $\poi {\varepsilon }1{,\ldots,}{d}{}{}{}$ satisfying Condition (4) with respect to
${\goth r}$. Let $\theta $ be the map
$$ \xymatrix{ {\goth r}_{\r}\times {\mathrm {M}}_{k-1,d}(\k) \ar[rr]^{\theta } && 
{\goth r}^{k}}, \qquad (x,a_{i,j},2\leq i\leq k,1\leq j\leq d) \longmapsto 
(x,\sum_{j=1}^{d} a_{i,j}\varepsilon _{j}(x)) .$$
Since for all $(x,\poi x2{,\ldots,}{k}{}{}{})$ in $V_{1}^{(k)}$, 
$\poi x2{,\ldots,}{k}{}{}{}$ are in ${\goth r}^{x}$, $\theta $ is a bijective map
onto $V_{1}^{(k)}$. The open subset ${\goth r}_{\r}$ has a cover by open subsets 
$V$ such that for some $\poi e1{,\ldots,}{n}{}{}{}$ in ${\goth r}$, 
$\poi x{}{,\ldots,}{}{\varepsilon }{1}{d},\poi e1{,\ldots,}{n}{}{}{}$ is a basis of 
${\goth r}$ for all $x$ in $V$. Then there exist regular functions 
$\poi {\varphi }1{,\ldots,}{d}{}{}{}$ on $V\times {\goth r}$ such that 
$$ v-\sum_{j=1}^{\rg} \varphi _{j}(x,v)\varepsilon _{j}(x) \in 
{\mathrm {span}}(\{\poi e1{,\ldots,}{n}{}{}{}\})$$
for all $(x,v)$ in $V\times {\goth r}$, so that the restriction of $\theta $ to 
$V\times {\mathrm {M}_{k-1,d}}(\k)$ is an isomorphism onto 
${\goth X}_{R,k}\cap V\times {\goth r}^{k-1}$ whose inverse is 
$$(\poi x1{,\ldots,}{k}{}{}{}) \longmapsto 
(x_{1},((\poi {x_{1},x_{i}}{}{,\ldots,}{}{\varphi }{1}{d}),i=2,\ldots,k))$$
As a result, $\theta $ is an isomorphism and $V_{1}^{(k)}$ is a smooth variety. 
Since ${\goth r}_{\r}$ is a smooth open subset of the vector space ${\goth r}$, there 
exists a regular differential form $\omega $ of top degree on 
${\goth r}_{\r}\times {\mathrm {M}_{k-1,\rg}}(\k)$, without zero. Then 
$\theta _{*}(\omega )$ is a regular differential form of top degree on 
$V_{1}^{(k)}$, without zero.
\end{proof}

For $k\geq 2$ set:
$$V^{(k)} := V_{1}^{(k)} \cup V_{2}^{(k)} \quad  \text{and} \quad 
V_{1,2}^{(k)} := V_{1}^{(k)}\cap V_{2}^{(k)} .$$
For $2\leq k'\leq k$, the projection
$$ \xymatrix{ {\goth r}^{k} \ar[rr] && {\goth r}^{k'}}, \qquad
(\poi x1{,\ldots,}{k}{}{}{}) \longmapsto (\poi x1{,\ldots,}{k'}{}{}{})$$
induces the projection 
$$ \xymatrix{ {\goth X}_{R,k} \ar[rr] && {\goth X}_{R,k'}}, \qquad
\xymatrix{ V_{j}^{(k)} \ar[rr] && V_{j}^{(k')}} $$
for $j=1,\ldots,k'$. 

\begin{lemma}\label{l2rss1}
Suppose $k\geq 2$. Let $\omega $ be a regular differential form of top degree on 
$V_{1}^{(k)}$, without zero. Denote by $\omega '$ its restriction to $V_{1,2}^{(k)}$.

{\rm (i)} For $\varphi $ in $\k[V_{1}^{(k)}]$, if $\varphi $ has no zero then $\varphi $
is in $\k^{*}$.

{\rm (ii)} For some invertible element $\psi $ of $\k[V_{1,2}^{(2)}]$, 
$\omega ' = \psi {\sigma _{2}}_{*}(\omega ')$.

{\rm (iii)} The function $\psi (\psi \rond \sigma _{2})$ on $V_{1,2}^{(k)}$ is equal 
to $1$. 
\end{lemma}

\begin{proof}
The existence of $\omega $ results from Lemma~\ref{lrss1}.

(i) According to Lemma~\ref{lrss1}, there is an isomorphism $\theta $ from 
${\goth r}_{\r}\times {\mathrm {M}_{k-1,d}}(\k)$ onto $V_{1}^{(k)}$. Since 
$\varphi $ is invertible, $\varphi \rond \theta $ is an invertible element of 
$\k[{\goth r}_{\r}]$. According to Lemma~\ref{lsav2}(iii), 
$\k[{\goth r}_{\r}]=\k[{\goth r}]$. Hence $\varphi $ is in $\k^{*}$.

(ii) The open subset $V_{1,2}^{(k)}$ is invariant under $\sigma _{2}$ so that 
$\omega '$ and ${\sigma _{2}}_{*}(\omega ')$ are regular differential forms of top
degree on $V_{1,2}^{(k)}$, without zero. Then for some invertible element $\psi $ 
of $\k[V_{1,2}^{(k)}]$, $\omega ' = \psi {\sigma _{2}}_{*}(\omega ')$. Let 
$O_{2}$ be the set of elements $(x,a_{i,j},1\leq i \leq k-1,1\leq j\leq d)$ of 
${\goth r}_{\r}\times {\mathrm {M}_{k-1,d}}(\k)$ such that 
$$ a_{1,1}\varepsilon _{1}(x) + \cdots + a_{1,\rg} \varepsilon _{\rg}(x) \in 
{\goth r}_{\r}.$$
Then $O_{2}$ is the inverse image of $V_{1,2}^{(k)}$ by $\theta $. As a result,
$\k[V_{1,2}^{(k)}]$ is a polynomial algebra over $\k[V_{1,2}^{(2)}]$ since for 
$k=2$, $O_{2}$ is the inverse image by $\theta $ of $V_{1,2}^{(2)}$. Hence 
$\psi $ is in $\k[V_{1,2}^{(2)}]$ since $\psi $ is invertible.

(iii) Since the restriction of $\sigma _{2}$ to $V_{1,2}^{(k)}$ is an involution,
$$ {\sigma _{2}}_{*}(\omega ') = (\psi \rond \sigma _{2}) \omega ' = 
(\psi \rond \sigma _{2})\psi {\sigma _{2}}_{*}(\omega '),$$
whence $(\psi \rond \sigma _{2})\psi = 1$.
\end{proof}

\begin{coro}\label{crss1}
The function $\psi $ is invariant under the action of $R$ in $V_{1,2}^{(k)}$ and for
some sequence $m_{\alpha },\alpha \in {\cal R}$ in ${\Bbb Z}$, 
$$ \psi (\poi x1{,\ldots,}{k}{}{}{}) = \pm \prod_{\alpha \in {\cal R}}
(\alpha (x_{1})\alpha (x_{2})^{-1})^{m_{\alpha }},$$
for all $(\poi x1{,\ldots,}{k}{}{}{})$ in ${\goth t}_{\r}^{2}\times {\goth t}^{k-2}$.
\end{coro}

\begin{proof}
First of all, since $V_{1}^{(k)}$ and $V_{2}^{(k)}$ are invariant under the action of
$R$ in ${\goth X}_{R,k}$, so is $V_{1,2}^{(k)}$. Let $g$ be in $R$. As $\omega $ has 
no zero, $g.\omega =p_{g}\omega $ for some invertible element $p_{g}$ of 
$\k[V_{1}^{(k)}]$. By Lemma~\ref{l2rss1}(i), $p_{g}$ is in $\k^{*}$. Since 
$\sigma _{2}$ is a $R$-equivariant isomorphism from $V_{1}^{(k)}$ onto $V_{2}^{(k)}$, 
$$ g.{\sigma _{2}}_{*}(\omega )=p_{g}{\sigma _{2}}_{*}(\omega ) \quad  \text{and} \quad
p_{g} \omega ' = g.\omega ' = (g.\psi ) g.{\sigma _{2}}_{*}(\omega ') = 
p_{g} (g.\psi ) {\sigma _{2}}_{*}(\omega '),$$ 
whence $g.\psi =\psi $.

The open subset ${\goth t}_{\r}^{2}$ of ${\goth t}^{2}$ is the complement to the 
nullvariety of the function
$$ (x,y) \longmapsto \prod_{\alpha \in {\cal R}} \alpha (x)\alpha (y).$$
Then, by Lemma~\ref{l2rss1}(ii), for some $a$ in $\k^{*}$ and for some 
sequences $m_{\alpha },\alpha \in {\cal R}$ and 
$n_{\alpha },\alpha \in {\cal R}$ in ${\Bbb Z}$,
$$ \psi (\poi x1{,\ldots,}{k}{}{}{}) = a \prod_{\alpha \in {\cal R}}
\alpha (x_{1})^{m_{\alpha }}\alpha (x_{2})^{n_{\alpha }},$$
for all $(\poi x1{,\ldots,}{k}{}{}{})$ in ${\goth t}_{\r}^{2}\times {\goth t}^{k-2}$.
By Lemma~\ref{l2rss1}(iii),
$$ a^{2} \prod_{\alpha \in {\cal R}} \alpha (x)^{m_{\alpha }+n_{\alpha }}
\alpha (y)^{m_{\alpha }+n_{\alpha }} = 1,$$ 
for all $(x,y)$ in ${\goth t}_{\r}^{2}$. Hence $a^{2}=1$ and $m_{\alpha }+n_{\alpha }=0$ 
for all $\alpha $ in ${\cal R}$.
\end{proof}

According to Lemma~\ref{l2sav2}(i), for $\alpha $ in ${\cal R}$, $\theta _{\alpha }$ is 
a bijective regular map from ${\Bbb P}^{1}(\k)$ onto the closed subset $Z_{\alpha }$ of 
$X_{R}$ such that $\theta _{\alpha }(\infty )=V_{\alpha }$. Recall that $x_{\alpha }$ is
a generator of ${\goth a}^{\alpha }$ and $h_{\alpha }$ is an element of ${\goth t}$ such 
that $\alpha (h_{\alpha })=1$. Denote by ${\goth t}'_{\alpha }$ the subset of elements
$x$ of ${\goth t}_{\alpha }$ such that $\gamma (x) \neq 0$ for all $\gamma $ in 
${\cal R}\setminus \{\alpha \}$. According to Condition (3) of Section~\ref{sa}, 
${\goth t}'_{\alpha }$ is a dense open subset of ${\goth t}_{\alpha }$. Let 
$x_{-\alpha }$ be in ${\goth r}^{*}$ orthogonal to ${\goth t}+{\goth a}^{\gamma }$ for 
all $\gamma $ in ${\cal R}\setminus \{\alpha \}$ and such that 
$x_{-\alpha }(x_{\alpha })=1$.

\begin{lemma}\label{l3rss1}
Suppose $k\geq 2$. Let $\alpha $ be in ${\cal R}$, $x_{0}$ and $y_{0}$ in 
${\goth t}'_{\alpha }$. Set:
$$ E := \k x_{0} \oplus \k h_{\alpha } \oplus {\goth a}^{\alpha }, \quad
E_{*} := x_{0} \oplus \k h_{\alpha } \oplus {\goth a}^{\alpha } , \quad
E_{*,1} := x_{0} \oplus \k h_{\alpha } \oplus ({\goth a}^{\alpha }\setminus \{0\}),
\quad 
E_{*,2} = y_{0} \oplus \k h_{\alpha } \oplus ({\goth a}^{\alpha }\setminus \{0\}).$$

{\rm (i)} For $x$ in $E_{*}$, ${\goth r}^{x}$ is contained in ${\goth t}_{\alpha }+E$.

{\rm (ii)} For $V$ subspace of dimension $d$ of ${\goth t}_{\alpha }+E$, $V$ is in 
$X_{R}$ if and only if it is in $Z_{\alpha }$.

{\rm (iii)} The intersection of $E_{*,1}\times E_{*,2}$ and ${\goth X}_{R,2}$ is the 
nullvariety of the function 
$$ (x,y) \longmapsto x_{-\alpha }(y) \alpha (x) - x_{-\alpha }(x) \alpha (y) $$
on $E_{*,1}\times E_{*,2}$.
\end{lemma}

\begin{proof}
(i) If $x$ is regular semisimple, its component on $h_{\alpha }$ is different 
from $0$ so that ${\goth r}^{x}=\theta _{\alpha }(z)$ for some $z$ in $\k$. Suppose that 
$x$ is not regular semisimple. Then $x$ is in $x_{0}+{\goth a}^{\alpha }$. Hence
${\goth r}^{x}$ is contained in ${\goth t}_{\alpha }+E$ since so is ${\goth r}^{x_{0}}$.

(ii) All element of $Z_{\alpha }$ is contained in ${\goth t}_{\alpha }+E$. Let $V$ be an
element of $X_{R}$, contained in ${\goth t}_{\alpha }+E$. According 
to Corollary~\ref{c2sa6}(i), $V$ is an algebraic commutative 
subalgebra of dimension $d$ of ${\goth r}$. By (i), $V=\theta _{\alpha }(z)$ for some $z$
in $\k$ if $V$ is in $A.{\goth t}$. Otherwise, $x_{\alpha }$ is in $V$. Then 
$V=\theta _{\alpha }(\infty )$ since $\theta _{\alpha }(\infty )$ is the centralizer 
of $x_{\alpha }$ in ${\goth t}_{\alpha }+E$.

(iii) Let $(x,y)$ be in $E_{*,1}\times E_{*,2}\cap {\goth X}_{R,2}$. By definition, 
for some $V$ in $X_{R}$, $x$ and $y$ are in $V$. By (i) and (ii), 
$V=\theta_{\alpha }(z)$ for some $z$ in ${\Bbb P}^{1}(\k)$. For $z$ in $\k$, 
$$ x = x_{0}+s(h_{\alpha }-z x_{\alpha }) \quad  \text{and} \quad 
y = y_{0} + s'(h_{\alpha } - z x_{\alpha })$$
for some $s$, $s'$ in $\k$, whence the equality of the assertion. For 
$z=\infty $, $$ x = x_{0}+sx_{\alpha } \quad  \text{and} \quad 
y = y_{0} + s'x_{\alpha } $$
for some $s$, $s'$ in $\k$ so that $\alpha (x)=\alpha (y) = 0$. Conversely,
let $(x,y)$ be in $E_{*,1}\times E_{*,2}$ such that
$$ x_{-\alpha }(y) \alpha (x) - x_{-\alpha }(x) \alpha (y) = 0.$$
If $\alpha (x)=0$ then $\alpha (y) = 0$ and $x$ and $y$ are in 
$V_{\alpha }=\theta _{\alpha }(\infty )$. If $\alpha (x)\neq 0$, then 
$\alpha (y) \neq 0$ and
$$x \in \theta _{\alpha } (-\frac{x_{-\alpha }(x)}{\alpha (x)}) 
\quad  \text{and} \quad
y \in \theta _{\alpha } (-\frac{x_{-\alpha }(x)}{\alpha (x)}) ,$$
whence the assertion.
\end{proof}

Set $V^{(1)}:={\goth r}_{\r}$. 

\begin{prop}\label{prss1}
For $k$ positive integer, there exists on $V^{(k)}$ a regular differential form of top 
degree without zero.
\end{prop}

\begin{proof}
For $k=1$, it is true since ${\goth r}_{\r}$ is an open subset of the vector sapce 
${\goth r}$. So we can suppose $k\geq 2$. According to Corollary~\ref{crss1}, it suffices
to prove $m_{\alpha }=0$ for all $\alpha $ in ${\cal R}$. Indeed, if so, by 
Corollary~\ref{crss1}, $\psi =\pm 1$ on the open subset 
$R.({\goth t}_{\r}^{2}\times {\goth t}^{k-2})$ of $V^{(k)}$ so that $\psi =\pm 1$ on 
$V_{1,2}^{(k)}$. Then, by Lemma~\ref{l2rss1}(ii), $\omega $ and 
$\pm {\sigma _{2}}_{*}(\omega )$ have the same restriction to $V_{1,2}^{(k)}$ so that
there exists a regular differential form of top degree $\tilde{\omega }$ on $V^{(k)}$
whose restrictions to $V_{1}^{(k)}$ and $V_{2}^{(k)}$ are $\omega $ and 
$\pm {\sigma _{2}}_{*}(\omega )$ respectively. Moreover, $\tilde{\omega }$ has no zero
since so has $\omega $.

Since $\psi $ is in $\k[V_{1,2}^{(2)}]$ by Lemma~\ref{l2rss1}(ii), we can suppose $k=2$.
Let $\alpha $ be in ${\cal R}$, $E$, $E_{*}$,$E_{*,1}$, $E_{*,2}$ as in 
Lemma~\ref{l2rss1}. Suppose $m_{\alpha }\neq 0$. A contradiction is expected. The 
restriction of $\psi $ to $E_{*,1}\times E_{*,2}\cap V_{1,2}^{(2)}$ is given by 
$$ \psi (x,y) = a x_{-\alpha }(x)^{m}x_{-\alpha }(y)^{n} ,$$  
with $a$ in $\k^{*}$ and $(m,n)$ in ${\Bbb Z}^{2}$ since $\psi $ is an invertible element
of $\k[V_{1,2}^{(2)}]$. According to Lemma~\ref{l2rss1}(iii), $n=-m$ and $a=\pm 1$.
Interchanging the role of $x$ and $y$, we can suppose $m$ in ${\Bbb N}$. For $(x,y)$ in 
$E_{*,1}\times E_{*,2}\cap V_{1,2}^{(2)}$ such that $\alpha (x)\neq 0$, 
$\alpha (y)\neq 0$ and 
$$ \psi (x,y) = \pm x_{-\alpha }(x)^{m} 
(\frac{x_{-\alpha }(x) \alpha (y)}{\alpha (x)})^{-m} = 
\pm \alpha (x)^{m} \alpha (y)^{-m} .$$ 
As a result, by Corollary~\ref{crss1}, for $x$ in $x_{0}+\k^{*} h_{\alpha }$ and 
$y$ in $y_{0}+\k^{*}h_{\alpha }$, 
\begin{eqnarray}\label{eqrss1}
\pm \alpha (x)^{m}\alpha (y)^{-m} = \pm \prod_{\gamma \in {\cal R}}
\gamma (x)^{m_{\gamma }}\gamma (y)^{-m_{\gamma }} .
\end{eqnarray}
For $\gamma $ in ${\cal R}$, 
$$ \gamma (x) = \gamma (x_{0}) + \alpha (x)\gamma (h_{\alpha }) 
\quad  \text{and} \quad
\gamma (y) = \gamma (y_{0}) + \alpha (y)\gamma (h_{\alpha }) .$$
Since $m$ is in ${\Bbb N}$,  
\begin{eqnarray}\label{eq2rss1}
\pm \alpha (x)^{m} \prod_{\mycom{\gamma \in {\cal R}}{m_{\gamma }>0}}
(\gamma (y_{0}) + \alpha (y)\gamma (h_{\alpha }))^{m_{\gamma }}
\prod_{\mycom{\gamma \in {\cal R}}{m_{\gamma }< 0}}
(\gamma (x_{0}) + \alpha (x)\gamma (h_{\alpha }))^{-m_{\gamma }} =  \\
\nonumber \pm \alpha (y)^{m} \prod_{\mycom{\gamma \in {\cal R}}{m_{\gamma }>0}}
(\gamma (x_{0}) + \alpha (x)\gamma (h_{\alpha }))^{m_{\gamma }}
\prod_{\mycom{\gamma \in {\cal R}}{m_{\gamma } < 0}}
(\gamma (y_{0}) + \alpha (y)\gamma (h_{\alpha }))^{-m_{\gamma }} .
\end{eqnarray}
For $m_{\alpha }$ positive, the terms of lowest degree in $(\alpha (x),\alpha (y))$ of 
left and right sides are
$$ \pm \alpha (x)^{m} \alpha (y)^{m_{\alpha }}
\prod_{\mycom{\gamma \in {\cal R}\setminus \{\alpha \}}{m_{\gamma }>0}} 
\gamma (y_{0})^{m_{\gamma }} 
\prod_{\mycom{\gamma \in {\cal R}\setminus \{\alpha \}}{m_{\gamma }<0}} 
\gamma (x_{0})^{-m_{\gamma }} 
\quad  \text{and} \quad 
\pm \alpha (y)^{m} \alpha (x)^{m_{\alpha }} 
\prod_{\mycom{\gamma \in {\cal R}\setminus \{\alpha \}}{m_{\gamma }>0}} 
\gamma (x_{0})^{m_{\gamma }} 
\prod_{\mycom{\gamma \in {\cal R}\setminus \{\alpha \}}{m_{\gamma }<0}} 
\gamma (y_{0})^{-m_{\gamma }} $$
respectively and for $m_{\alpha }$ negative, the terms of lowest degree in 
$(\alpha (x),\alpha (y))$ of left and right sides are
$$ \pm \alpha (x)^{m+m_{\alpha }} 
\prod_{\mycom{\gamma \in {\cal R}\setminus \{\alpha \}}{m_{\gamma }>0}} 
\gamma (y_{0})^{m_{\gamma }} 
\prod_{\mycom{\gamma \in {\cal R}\setminus \{\alpha \}}{m_{\gamma }<0}} 
\gamma (x_{0})^{-m_{\gamma }} 
\quad  \text{and} \quad 
\pm \alpha (y)^{m+m_{\alpha }} 
\prod_{\mycom{\gamma \in {\cal R}\setminus \{\alpha \}}{m_{\gamma }>0}} 
\gamma (x_{0})^{m_{\gamma }} 
\prod_{\mycom{\gamma \in {\cal R}\setminus \{\alpha \}}{m_{\gamma }<0}} 
\gamma (y_{0})^{-m_{\gamma }} $$
respectively. From the equality of these terms, we deduce $m = \pm m_{\alpha }$ and
$$ \prod_{\mycom{\gamma \in {\cal R}\setminus \{\alpha \}}{m_{\gamma }>0}} 
\gamma (y_{0})^{m_{\gamma }} 
\prod_{\mycom{\gamma \in {\cal R}\setminus \{\alpha \}}{m_{\gamma }<0}} 
\gamma (x_{0})^{-m_{\gamma }} =
\pm \prod_{\mycom{\gamma \in {\cal R}\setminus \{\alpha \}}{m_{\gamma }>0}} 
\gamma (x_{0})^{m_{\gamma }} 
\prod_{\mycom{\gamma \in {\cal R}\setminus \{\alpha \}}{m_{\gamma }<0}} 
\gamma (y_{0})^{-m_{\gamma }} .$$
Since the last equality does not depend on the choice of $x_{0}$ and $y_{0}$ in 
${\goth t}'_{\alpha }$, this equality remains true for all
$(x_{0},y_{0})$ in ${\goth t}_{\alpha }\times {\goth t}_{\alpha }$. As a result,
as the degrees in $\alpha (x)$ of the left and right sides of Equality (\ref{eq2rss1})
are the same,
\begin{eqnarray}\label{eq3rss1}
m - \sum_{\mycom{\gamma \in {\cal R}}{m_{\gamma }< 0\;  \text{and} \; 
\gamma (h_{\alpha })\neq 0}} m_{\gamma } = 
\sum_{\mycom{\gamma \in {\cal R}}{m_{\gamma } > 0 \;  \text{and} \; 
\gamma (h_{\alpha })\neq 0}} m_{\gamma } .
\end{eqnarray}

Suppose $m=m_{\alpha }$. By Equality (\ref{eqrss1}), 
$$ \prod_{\gamma \in {\cal R}\setminus \{\alpha \}} 
\gamma (x)^{m_{\gamma }}\gamma (y)^{-m_{\gamma }} = \pm 1.$$
Since this equality does not depend on the choice of $x_{0}$ and $y_{0}$ in 
${\goth t}'_{\alpha }$, it holds for all $(x,y)$ in 
${\goth t}_{\r}\times {\goth t}_{\r}$. Hence $m_{\gamma }=0$ for all $\gamma $ in 
${\cal R}\setminus \{\alpha \}$ and $m=0$ by Equality (\ref{eq3rss1}). It is 
impossible since $m_{\alpha }\neq 0$. Hence $m=-m_{\alpha }$. Then, by 
Equality (\ref{eqrss1})
$$ \prod_{\gamma \in {\cal R}\setminus \{\alpha \}} 
\gamma (x)^{m_{\gamma }}\gamma (y)^{-m_{\gamma }} = 
\pm \alpha (x)^{2m}\alpha (y)^{-2m} .$$
Since this equality does not depend on the choice of $x_{0}$ and $y_{0}$ in 
${\goth t}'_{\alpha }$, it holds for all $(x,y)$ in 
${\goth t}_{\r}\times {\goth t}_{\r}$. Then $m=0$, whence the contradiction.
\end{proof}

\subsection{Rational singularities and Gorensteinness of $X_{R}$} \label{rss2}
For $Y$ subvariety of $\ec {Gr}r{}{}d$, denote by ${\cal E}_{Y}$ the restriction to $Y$ 
of the tautological vector bundle of rank $d$ over $\ec {Gr}r{}{}d$. In particular, for 
$Y$ contained in $X_{R}$, ${\cal E}_{Y}$ is a subvariety of ${\cal E}$. For $k$ positive 
integer, denote by $\tau _{k}$ and $\varpi _{k}$ the restrictions to ${\cal E}^{(k)}$ of 
the canonical projections
$$ \xymatrix{X_{R}\times {\goth r}^{k}\ar[rr]^{\tau _{k}} && {\goth r}^{k}} 
\quad  \text{and} \quad
\xymatrix{X_{R}\times {\goth r}^{k}\ar[rr]^{\varpi _{k}} && X_{R}}  . $$

\begin{lemma}\label{lrss2}
{\rm (i)} The morphism $\tau _{k}$ is a projective and birational morphism onto
${\goth X}_{R,k}$.

{\rm (ii)} The sets $V^{(k)}$ and $\tau _{k}^{-1}(V^{(k)})$ are smooth open subsets
of ${\goth X}_{R,k}$ and ${\cal E}^{(k)}$. Moreover, for $k\geq 2$, they are big open subsets
of ${\goth X}_{R,k}$ and ${\cal E}^{(k)}$.

{\rm (iii)} The restriction of $\tau _{k}$ to $\tau _{k}^{-1}(V^{(k)})$ is an isomorphism 
onto $V^{(k)}$.
\end{lemma}

\begin{proof}
Since $X_{R}$ is a projective variety, $\tau _{k}$ is projective and its image 
is ${\goth X}_{R,k}$ by definition. For $(\poi x1{,\ldots,}{k}{}{}{})$ in $V^{(k)}$ and 
$(u,\poi x1{,\ldots,}{k}{}{}{})$ in $\tau _{k}^{-1}((\poi x1{,\ldots,}{k}{}{}{}))$, 
$u={\goth r}^{x_{1}}$ if $x_{1}$ is in ${\goth r}_{\r}$ and $u={\goth r}^{x_{2}}$ if 
$x_{2}$ is in ${\goth r}_{\r}$. As a result, the restriction of $\tau _{k}$ to 
$\tau _{k}^{-1}(V^{(k)})$ is a bijective morphism onto $V^{(k)}$. Hence $\tau _{k}$
is a birational morphism and by Zariski's Main Theorem~\cite[\S 9]{Mu}, this restriction
is an isomorphism since $V^{(k)}$ is a smooth variety by Lemma~\ref{lrss1}. So it
remains to prove that for $k\geq 2$, $\tau _{k}^{-1}(V^{(k)})$ is a big open subset of 
${\cal E}^{(k)}$

Suppose that ${\cal E}^{(k)}\setminus \tau _{k}^{-1}(V^{(k)})$ has an irreducible 
component $\Sigma $ of dimension $\dim {\cal E}^{(k)}-1$. A contradiction is expected. 
Since ${\cal E}^{(k)}$ and $\tau _{k}^{-1}(V_{k})$ are invariant under the automorphisms 
of $X_{R}\times {\goth r}^{k}$,
$$ (u,\poi x1{,\ldots,}{k}{}{}{}) \longmapsto (u,\poi {tx}1{,\ldots,}{k}{}{}{}), 
\qquad (t\in \k^{*}),$$ 
so is $\Sigma $. Then $\Sigma \cap X_{R}\times \{0\}=\varpi _{k}(\Sigma )\times \{0\}$ so 
that $\varpi _{k}(\Sigma )$ is a closed subset of $X_{R}$. Since 
$\dim \Sigma = \dim {\cal E}^{(k)}-1$, 
$\dim \varpi _{k}(\Sigma )\geq \dim X_{R}-1$. Suppose $\dim \Sigma =\dim X_{R}-1$. Then 
for all $u$ in $\varpi _{k}(\Sigma )$, $\{u\}\times u^{k}$ is in $\Sigma $. It is 
impossible since for all $u$ in a dense open subset of $\varpi _{k}(\Sigma )$, 
$u={\goth r}^{x}$ for some $x$ in ${\goth r}_{\r}$ by Corollary~\ref{csav3}. Hence 
$\varpi _{k}(\Sigma )=X_{R}$. Then for all $u$ in a dense open subset of $X'_{R}$, 
$\{u\}\times u^{k}\cap \Sigma $ has codimension $1$ in $\{u\}\times u^{k}$. Since
the image of $\{u\}\times u^{k}\cap \Sigma $ by the projection
$$ (u,\poi x1{,\ldots,}{k}{}{}{}) \longmapsto x_{1}$$
is not dense in $u$, for all $x_{1}$ in a dense open subset of its image, 
$\{u\}\times \{x_{1}\}\times u^{k-1}$ is contained in $\Sigma $, whence the contradiction 
since $u\cap {\goth r}_{\r}$ is not empty.
\end{proof}

By definition, ${\cal E}^{(k)}$ is the inverse image of $X_{R}$ by the bundle projection 
of the vector bundle 
$$ \{u,\poi x1{,\ldots,}{k}{}{}{}) \in \ec {Gr}r{}{}d\times {\goth r}^{k} \; \vert \;
\poi {u\ni x}1{,\ldots,}{k}{}{}{}\} $$
over $\ec {Gr}r{}{}d$ so that ${\cal E}^{(k)}$ is vector bundle of rank $kd$ over $X_{R}$.
In particular, ${\cal E}^{(1)}={\cal E}$. According to~\cite{Hi}, there exists a 
desingulartization $\Gamma $ of $X_{R}$ with morphism $\rho $ such that the restriction 
of $\rho $ to $\rho ^{-1}({X_{R}}_{\loc})$ is an isomorphism onto ${X_{R}}_{\loc}$. 
Let $\widetilde{{\cal E}^{(1)}}$ be the following fiber product
$$ \xymatrix{ \widetilde{{\cal E}^{(1)}} \ar[rr]^{\overline{\rho }} \ar[d] && 
{\cal E}^{(1)} \ar[d]^{\varpi _{1}}\\ \Gamma \ar[rr]_{\rho } && X_{R}}$$
with $\overline{\rho }$ the restriction map. Then $\widetilde{{\cal E}^{(1)}}$ is
a vector bundle of rank $d$ over $\Gamma $. In particular, it is a smooth variety
since $\Gamma $ is smooth. 

Let $O$ be a trivialization open subset of the vector bundle ${\cal E}^{(1)}$ and 
let $\Phi _{1}$ be a local trivialization over $O$ of 
${\cal E}^{(1)}$, whence the following commutative diagram
$$ \xymatrix{ \varpi _{1}^{-1}(O) \ar[rr]^{\Phi _{1}}  \ar[rrd]_{\varpi _{1}} && 
O \times \k^{d} \ar[d]^{\pr {1}} \\ && O} .$$  
Then $O$ is a trivialization open subset of the vector bundle ${\cal E}^{(k)}$. The 
variety ${\cal E}^{(1)}$ is a closed subbundle of ${\cal E}^{(k)}$ over 
$X_{R}$ and for some local trivialization $\Phi $ over $O$ of 
${\cal E}^{(k)}$, we have the following commutative diagram
$$ \xymatrix{ \varpi _{k} ^{-1}(O) \ar[rr]^{\Phi }  \ar[rrd]_{\varpi _{k}} && 
O \times \k^{kd} \ar[d]^{\pr {1}} \\ && O} ,$$   
$\Phi _{1}$ is the restriction of $\Phi $ to $\varpi _{1}^{-1}(O)$ and 
$\Phi (\varpi _{1}^{-1}(O)) = O \times \k^{d}\times \{0\}$.  

\begin{lemma}\label{l2rss2}
Suppose $k\geq 2$. Denote by $\mu $ a generator of $\Omega _{\k^{kd}}$ and by 
$\tilde{\rho }$ the restriction of $\rho \mul {\mathrm {id}}_{\k^{kd}}$ to 
$\rho ^{-1}(O)\times \k^{kd}$.

{\rm (i)} The sheaf $\Omega _{{{\cal E}^{(k)}}_{\loc}}$ has a global section 
$\omega $ without zero.

{\rm (ii)} The sheaf $\Omega _{O_{{\mathrm {\loc}}}}$ has a global section $\omega _{O}$
without zero.

{\rm (iii)} For some $p$ in $\k[O\times \k^{kd}]\setminus \{0\}$, 
$\tilde{\rho }^{*}(p(\omega _{O}\wedge \mu ))$ has a regular extension to 
$\rho ^{-1}(O)\times \k^{kd}$.
\end{lemma}

\begin{proof}
(i) According to Proposition~\ref{prss1} and Lemma~\ref{lrss2}(iii), 
$\Omega _{\tau _{k}^{-1}(V^{(k)})}$ has a global section without zero. 
By Lemma~\ref{lrss2}(ii), $\tau _{k}^{-1}(V^{(k)})$ is a smooth big open subset
of ${\cal E}^{(k)}$. So, by Lemma~\ref{lars}, 
$\Omega _{{{\cal E}^{(k)}}_{\loc}}$ has a global section without zero.

(ii) Since $\mu $ is a generator of $\Omega _{\k^{kd}}$, there exists a unique $\nu $ in 
$\tk {\k}{\k[\k^{kd}]}\Gamma (O_{\loc},\Omega _{O_{\loc}})$ such that 
$$ \Phi _{*}(\omega \left \vert \right._{\varpi _{k}^{-1}(O_{\loc})}) = 
\nu \wedge \mu .$$
Moreover, $\nu $ has no zero since so has $\omega $. Let $V$ be an affine open subset of 
$O_{\loc}$ such that the restriction of $\Omega _{O_{\loc}}$ to $V$ is locally free, 
generated by the local section $\omega _{V}$. Then for some $p_{V}$ in 
$\k[V\times \k^{kd}]$, 
\begin{eqnarray}\label{eqrss2}
\Phi _{*}(\omega \left \vert \right._{\varpi _{k}^{-1}(V)}) = p_{V}\omega _{V}\wedge \mu .
\end{eqnarray}
Then $p_{V}$ has no zero since so has $\nu \wedge \mu $. As a result, $p_{V}$ is in 
$\k[V]$ and $p_{V}\omega _{V}$ is a local section of $\Omega _{O_{\loc}}$ without zero.
By the unicity of the decomposition (\ref{eqrss2}), for two different affine open subsets
$V$ and $V'$ as above, the differential forms $p_{V}\omega _{V}$ and $p_{V'}\omega _{V'}$
have the same restriction to $V\cap V'$. As a result, since such affine open subsets
cover $O_{\loc}$, for some global section $\omega _{O}$ of $\Omega _{O_{\loc}}$,
$$ \Phi _{*}(\omega \left \vert \right._{\varpi _{k}^{-1}(O_{\loc})}) = 
\omega _{O} \wedge \mu .$$
Moreover, $\omega _{O}$ is unique and has no zero.

(iii) Let $\omega _{1}$ be a generator of $\Omega _{{\goth r}}$ and let $\mu _{1}$ be
a generator of $\Omega _{\k^{d}}$. By (i), $\omega _{O}\wedge \mu _{1}$ is a global 
section of $\Omega _{O_{\loc}\times \k^{d}}$, without zero. So for some regular 
function $p$ on $O_{\loc}\times \k^{d}$, 
\begin{eqnarray}\label{eq2rs}
{\Phi _{1}}_{*}((\tau _{1})^{*}(\omega _{1}) 
\left \vert \right. _{\varpi _{1}^{-1}(O_{\loc})}) = p \omega _{O}\wedge \mu _{1}
. \end{eqnarray}
According to Theorem~\ref{tns3}, $X_{R}$ is normal. Then so is $O$ and $p$ has a regular 
extension to $O\times \k^{d}$. Denote again by $p$ this extension. According to Equality 
(\ref{eq2rs}), the differential form 
$\tilde{\rho }^{*}(p\omega _{O}\wedge \mu _{1})$ on $\rho ^{-1}(O_{\loc})\times \k^{d}$ 
has a regular extension to $\rho ^{-1}(O)\times \k^{d}$. In fact, denoting by 
$\overline{\Phi _{1}}$ the local trivialization over $\rho ^{-1}(O)$ of 
$\widetilde{{\cal E}^{(1)}}$ such that the following diagram 
$$ \xymatrix{ (\varpi _{1}\rond \overline{\rho }^{-1})(O) 
\ar[rr]^{\overline{\Phi _{1}}} \ar[d]_{\overline{\rho }} && 
\rho ^{-1}(O)\times \k^{d} \ar[d]^{\tilde{\rho }} \\ 
\varpi _{1}^{-1}(O) \ar[rr]_{\Phi _{1}} && O\times \k^{d} }$$ 
is commutative, it is the restriction to $\rho ^{-1}(O_{\loc})\times \k^{d}$ of 
$$\overline{\Phi _{1}}_{*}((\tau_{1}\rond \overline{\rho })^{*}(\omega _{1}) 
\left \vert \right. _{(\varpi _{1}\rond \overline{\rho }^{-1})^{-1}(O)} ).$$
For some generator $\mu '$ of $\Omega _{\k^{(k-1)d}}$, $\mu = \mu _{1}\wedge \mu '$
and $\k[O\times \k^{d}]$ is naturally embedded in $\k[O\times \k^{kd}]$.
As a result, $\tilde{\rho }^{*}(p\omega _{O}\wedge \mu )$
has a regular extension to $\rho ^{-1}(O)\times \k^{kd}$.
\end{proof}

\begin{prop}\label{prss2}
The variety $X_{R}$ is Gorenstein with rational singularities.
\end{prop}

\begin{proof}
According to Theorem~\ref{tns3}, $X_{R}$ is normal and Cohen-Macaulay. Then
by Lemma~\ref{l2rss2},(ii) and (iii) and Corollary~\ref{cars}, 
$O\times \k^{kd}$ is Gorenstein with rational singularities. Then so is $O$ by 
Lemma~\ref{lsi},(i) and (ii). Since there is a cover of $X_{R}$ by open subsets as $O$, 
$X_{R}$ is Gorenstein with rational singularities.
\end{proof}

As already mentioned, ${\goth u}$ is in ${\cal C}_{{\goth h},*}$, whence 
Theorem~\ref{tint} by Proposition~\ref{prss2}.

\appendix

\section{Rational Singularities} \label{ars}
Let $X$ be an affine irreducible normal variety. 

\begin{lemma}\label{lars}
Let $Y$ be a smooth big open subset of $X$. 

{\rm (i)} All regular differential form of top degree on $Y$ has a unique regular 
extension to $X_{\loc}$.

{\rm (ii)} Suppose that $\omega $ is a regular differential form of top degree on $Y$, 
without zero. Then the regular extension of $\omega $ to $X_{\loc}$ has no zero.
\end{lemma}

\begin{proof}
(i) Since $\Omega _{X_{\loc}}$ is a locally free module of rank one, there is an affine 
open cover $\poi O1{,\ldots,}{k}{}{}{}$ of $X_{\loc}$ such that the restriction of 
$\Omega _{X_{\loc}}$ to $O_{i}$ is a free $\an {O_{i}}{}$-module generated by some 
section $\omega _{i}$. For $i=1,\ldots,k$, set $O'_{i} := O_{i}\cap Y$. Let $\omega $ be 
a regular differential form of top degree on $Y$. For $i=1,\ldots,k$, for some regular 
function $a_{i}$ on $O'_{i}$, $a_{i}\omega _{i}$ is the restriction of $\omega $ to 
$O'_{i}$. As $Y$ is a big open subset of $X$, $O'_{i}$ is a big open subset of $O_{i}$. 
Hence $a_{i}$ has a regular extension to $O_{i}$ since $O_{i}$ is normal. Denoting again 
by $a_{i}$ this extension, for $1\leq i,j\leq k$, $a_{i}\omega _{i}$ and 
$a_{j}\omega _{j}$ have the same restriction to $O'_{i}\cap O'_{j}$ and $O_{i}\cap O_{j}$ since $\Omega _{X_{\loc}}$ is torsion free as a locally free module. Let $\omega '$ be 
the global section of $\Omega _{X_{\loc}}$ extending the $a_{i}\omega _{i}$'s. Then 
$\omega '$ is a regular extension of $\omega $ to $X_{\loc}$ and this extension is unique
since $Y$ is dense in $X_{\loc}$ and $\Omega _{X_{\loc}}$ is torsion free.

(ii)  Suppose that $\omega $ has no zero. Let $\Sigma $ be the nullvariety of $\omega '$ 
in $X_{\loc}$. If it is not empty, $\Sigma $ has codimension $1$ in $X_{\loc}$. As $Y$ is
a big open subset of $X$, $\Sigma \cap X_{\loc}$ is not empty if so is $\Sigma $. As a 
result, $\Sigma $ is empty.
\end{proof}

Denote by $\iota $ the canonical injection from $X_{\loc}$ into $X$.

\begin{lemma}\label{l2ars}
Suppose that $\Omega _{X_{\loc}}$ has a global section $\omega $ without zero. Then 
the $\an X{}$-module $\iota _{*}(\Omega _{X_{\loc}})$ is free of rank $1$. 
More precisely, the morphism $\theta $:
$$ \xymatrix{ \an X{} \ar[rr]^{\theta } && \iota _{*}(\Omega _{X_{\loc}})}, \qquad
\psi \longmapsto \psi \omega $$
is an isomorphism.
\end{lemma}

\begin{proof}
For $\varphi $ a local section of $\iota _{*}(\Omega _{X_{\loc}})$ above the open subset 
$U$ of $X$, for some regular function $\psi $ on $U\cap X_{\loc}$, 
$\psi \omega $ is the restriction of $\varphi $ to $U\cap X_{\loc}$. Since $X$ is normal,
so is $U$ and $U\cap X_{\loc}$ is a big open subset of $U$. Hence $\psi $ has a regular 
extension to $U$. As a result, there exists a well defined morphism from 
$\iota _{*}(\Omega _{X_{\loc}})$ to $\an X{}$ whose inverse is $\theta $.
\end{proof}

According to \cite{Hi}, $X$ has a desingularization $Z$ with morphism $\tau $ such that 
the restriction of $\tau $ to $\tau ^{-1}(X_{\loc})$ is an isomorphism onto $X_{\loc}$. 
For $U$ open subset of $X$, denote by $\tau _{U}$ the restriction of $\tau $ to 
$\tau ^{-1}(U)$.

\begin{prop}\label{pars}
Suppose that $X$ is Cohen-Macaulay and that there exists a morphism 
$\mu : \xymatrix{ \an Z{} \ar[r] & \Omega _{Z}}$
such that for some $p$ in $\k[X]$, $\tau _{*}\mu $ is an isomorphism onto 
$p\tau _{*}(\Omega _{Z})$. Then $X$ has rational singularities. 
\end{prop}

The following proof is the weak variation of the proof of~\cite[Lemma 2.3]{Hin}.

\begin{proof}
Since $Z$ and $X$ are varieties over $\k$, we have the commutative diagram
$$\xymatrix{ Z \ar[rr]^{\tau } \ar[rd]_{p} &&
X \ar[ld]^{q} \\ & {\mathrm {Spec}}(\k) & } .$$
According to ~\cite[V. \S 10.2]{Ha0}, $p^{!}(\k)$ and $q^{!}(\k)$ are dualizing complexes
over $Z$ and $X$ respectively. Furthermore, by ~\cite[VII, 3.4]{Ha0} or 
\cite[4.3,(ii)]{Hin}, $p^{!}(\k)[-\dim Z]$ equals $\Omega _{Z}$. Set 
$\mathpzc{D} := q^{!}(\k)[-\dim Z]$ so that 
$\tau ^{!}(\mathpzc{D})=\Omega _{Z}$ by ~\cite[VII, 3.4]{Ha0} or \cite[4.3,(iv)]{Hin}.
In particular, $\mathpzc{D}$ is dualizing over $X$.
 
Since $\tau $ is a projective morphism, we have the isomorphism 
\begin{eqnarray}\label{eqars}
{\mathrm {R}}\tau _{*}({\mathrm {R}}\hhom_{Z}(\Omega _{Z},\Omega _{Z}))
\longrightarrow {\mathrm {R}}\hhom_{X}
({\mathrm {R}}(\tau )_{*}(\Omega _{Z}),\mathpzc{D})
\end{eqnarray}
by ~\cite[VII, 3.4]{Ha0} or \cite[4.3,(iii)]{Hin}. 
Since ${\mathrm {H}}^{i}({\mathrm {R}}\hhom_{Z}(\Omega _{Z},\Omega _{Z}))=\an Z{}$ for
$i=0$ and $0$ for $i>0$, the left hand side of (\ref{eqars}) can be identified with
${\mathrm {R}\tau _{*}}(\an Z{})$.

According to Grauert-Riemenschneider Theorem \cite{GR}, 
${\mathrm {R}}\tau _{*}(\Omega _{Z})$ has only cohomology in degree $0$. Since
$\tau $ is projective and birational and $Z$ is normal, $\tau _{*}(\an Z{})=\an X{}$.
So by assumption of the proposition,
$$ {\mathrm {R}}\tau _{*}(\Omega _{Z}) \approx \frac{1}{p}\an X{}{},$$
whence
$${\mathrm {R}}\hhom_{X}({\mathrm {R}}(\tau )_{*}(\Omega _{Z}),\mathpzc{D}) 
\approx \tk{\an X{}}{p\an X{}{}}\mathpzc{D}$$
and (\ref{eqars}) can be rewritten as
\begin{eqnarray}\label{eq2ars}
{\mathrm {R}}\tau _{*}(\an Z{}) \approx \tk{\an X{}}{p\an X{}{}}\mathpzc{D} .
\end{eqnarray}
Since $X$ is Cohen-Macaulay, $\mathpzc{D}$ has cohomology in only one degree. So, 
by flatness of the $\an X{}$-module $p\an X{}$, $\tk{\an X{}}{p\an X{}{}}\mathpzc{D}$
has cohomology in only one degree. As a result, by (\ref{eq2ars}), 
${\mathrm {R}}^{i}\tau _{*}(\an Z{})=0$ for $i>0$, that is $X$ has rational singularities.
\end{proof}

Denote by ${\cal M}$ the cohomology in degree $0$ of $\mathpzc{D}$.

\begin{lemma}\label{l3ars}
Suppose that $X$ has rational singularities. Then the $\an X{}$-modules 
$\tau _{*}(\Omega _{Z})$ and ${\cal M}$ are isomorphic. In particular, 
$\tau _{*}(\Omega _{Z})$ has finite injective dimension.
\end{lemma}

\begin{proof}
Since $X$ has rational singularities, ${\mathrm {R}\tau _{*}}(\an Z{})=\an X{}$ and 
$\mathpzc{D}$ has only cohomology in degree $0$. Moreover, by Grauert-Riemenschneider 
Theorem \cite{GR}, 
${\mathrm {R}}\tau _{*}(\Omega _{Z})$ has only cohomology in degree $0$, whence 
$R\tau _{*}(\Omega _{Z}) = \tau _{*}(\Omega _{Z})$. Then, by (\ref{eqars}), we have
the isomorphism 
$$ \xymatrix{ \an X{} \ar[rr] &&  \hhom_{X} ((\tau )_{*}(\Omega _{Z}),{\cal M})} .$$
As $\mathpzc{D}$ is dualizing, we have the isomorphism 
$$ \xymatrix{ R\tau _{*}(\Omega _{Z}) \ar[rr] && 
{\mathrm {R}}\hhom_{X}({\mathrm {R}}\hhom_{X}(R\tau _{*}(\Omega _{Z}),\mathpzc{D}),
\mathpzc{D})}$$
whence the isomorphism $\xymatrix{ \tau _{*}(\Omega _{Z}) \ar[r] & {\cal M}}$ by 
(\ref{eqars}). As a result, $\tau _{*}(\Omega _{Z})$ has finite injective dimension since
so has ${\cal M}$. 
\end{proof}

\begin{coro}\label{cars}
Let $Y$ be a smooth big open subset of $X$. Suppose that the following conditions are 
verified:
\begin{itemize}
\item [{\rm (1)}] $X$ is Cohen-Macaulay,
\item [{\rm (2)}] $\Omega _{Y}$ has a global section $\omega $ without zero,
\item [{\rm (3)}] for some global section $\omega _{Z}$ of $\Omega _{Z}$ and for some $p$
in $\k[X]\setminus \{0\}$, the restriction of $\omega _{Z}$ to $\tau ^{-1}(Y)$ is equal 
to $p\tau _{Y}^{*}(\omega )$.
\end{itemize}
Then $X$ is Gorenstein with rational singularities. Moreover, its canonical module is free
of rank $1$.
\end{coro}

\begin{proof}
According to Lemma~\ref{lars}(ii), $\omega $ has a unique regular extension to 
$X_{\loc}$ and this extension has no zero. Denote again by $\omega $ this extension. 
Since $Z$ is irreducible, $\tau ^{-1}(Y)$ is dense in $\tau ^{-1}(X_{\loc})$ so that the 
restriction of $\omega _{Z}$ to $\tau ^{-1}(X_{\loc})$ is equal to 
$p\tau _{X_{\loc}}^{*}(\omega )$ since $\Omega _{Z}$ has no torsion. Denote by $\mu $ the
morphism 
$$ \xymatrix{ \an Z{} \ar[rr]^{\mu } && \Omega _{Z}}, \qquad
\varphi \longmapsto \varphi \omega _{Z} .$$ 
Let $U$ be an open subset of $X$ and $\nu $ a local section of $\tau _{*}(\Omega _{Z})$
above $U$. Since $\omega $ has no zero and $\tau _{U_{\loc}}$ is an isomorphism onto 
$U_{\loc}$, 
$$ \nu \left \vert \right. _{\tau ^{-1}(U_{\loc})} = \tau _{U_{\loc}}^{*}
(\varphi \omega  \left \vert \right. _{U_{\loc}}) $$
for some $\varphi $ in $\k[U_{\loc}]$, whence
$$ p \nu \left \vert \right. _{\tau ^{-1}(U_{\loc})} = \varphi \rond \tau _{U_{\loc}}
(\omega _{Z} \left \vert \right. _{\tau ^{-1}(U_{\loc})}) $$
by Condition (3). Since $X$ is normal, so is $U$ and $U_{\loc}$ is a big open subset of 
$U$. Hence $\varphi $ has a regular extension to $U$. Denoting again by $\varphi $ this 
extension,
$$p\nu = \varphi \rond \tau _{U}(\omega _{Z} \left \vert \right. _{\tau ^{-1}(U)})$$
since $Z$ is irreducible and $\Omega _{Z}$ has no torsion. As a result the morphism 
$$ \tau _{*}\mu  : \xymatrix{ \tau _{*}(\an Z{}) \ar[rr] && p\tau _{*}(\Omega _{Z})}$$
is an isomorphism since it is obviously injective. So, by Proposition~\ref{pars}, 
$X$ has rational singularities. In particular, by~\cite[p.50]{KK}, 
$\tau _{*}(\Omega _{X})=\iota _{*}(\Omega _{X})$. Then, by Lemma~\ref{l2ars}, 
the canonical module of $X$ is free of rank $1$ and by Lemma~\ref{l3ars}, $X$ is 
Gorenstein.
\end{proof}

\section{About singularities} \label{si}
In this section we recall a well known result. Let $X$ be a variety and $Y$ a vector
bundle over $X$. Denote by $\tau $ the bundle projection.

\begin{lemma}\label{lsi}
{\rm (i)} If $Y$ is Gorenstein, then $X$ is Gorenstein.

{\rm (ii)} The variety $X$ has rational singularities if and only if so has $Y$.

{\rm (iii)} If $X$ is Cohen-Macaualay, then so is $Y$.
\end{lemma}

\begin{proof}
Let $y$ be in $Y$, $x:= \tau (y)$. Denote by $\widehat{\an Xx{}}$ and 
$\widehat{\an Yy{}}$ the completions of the local rings $\an Xx$ and $\an Yy$ 
respectively.

(i) Since $Y$ is a vector bundle over $X$, $\widehat{\an Yy}$ is a ring of formal series
over $\widehat{\an Xx}$. By~\cite[Proposition 3.1.19,(c)]{Br}, $\widehat{\an Yy{}}$ is
Gorenstein. So, by~\cite[Proposition 3.1.19,(b)]{Br}, $\widehat{\an Xx}$ is Gorenstein. 
Then by \cite[Proposition 3.1.19,(c)]{Br}, $\an Xx$ is Gorenstein, whence the assertion.

(ii) Since $Y$ is a vector bundle over $X$, then there exists a cover of $X$ by open 
subsets $O$, such that $\tau ^{-1}(O)$ is isomorphic to $O\times \k^{m}$ with 
$m=\dim Y - \dim X$. According to~\cite[p.50]{KK}, $O\times \k^{m}$ has rational 
singularities if and only if so has $O$, whence the assertion since a variety has 
rational singularities if and only it has a cover by open subsets having rational
singularities.

(iii) According to~\cite[Ch. 6, Theorem 17.7]{Mat}, a polynomial algebra over a 
Cohen-Macaulay algebra is Cohen-Macaulay. Hence for $O$ open subset of $X$ as in (ii), 
$\tau ^{-1}(O)$ is Cohen-Macaulay, whence the assertion since there is a cover of $Y$
by open subsets as $\tau ^{-1}(O)$.
\end{proof}

\end{document}